\documentclass[12pt]{article}
\usepackage{layout,yfonts, graphicx, amsmath, amsthm, amssymb, euscript, enumerate, bbold, bbm, mathrsfs}
\usepackage{mathtools}
\usepackage{enumitem}
\usepackage{esvect}
\numberwithin{equation}{section}
\usepackage{xcolor}
\definecolor{webgreen}{rgb}{0,.5,0}
\definecolor{webbrown}{rgb}{.8,0,0}
\definecolor{emphcolor}{rgb}{0.95,0.95,0.95}

\usepackage{hyperref}
\hypersetup{%
	colorlinks=true,
	linkcolor=webbrown,
	filecolor=webbrown,
	citecolor=webgreen,
	breaklinks=true}
\ifpdf \hypersetup{pdftex,
	bookmarksopen=true,
	bookmarksnumbered=true
} \else \hypersetup{dvips} \fi

\usepackage{blindtext} %Para enumerar
\usepackage{scrextend}

\addtokomafont{labelinglabel}{\sffamily}

\usepackage[colorinlistoftodos, textwidth=21mm, shadow,textsize=footnotesize]{todonotes}
\reversemarginpar%pone las notas del otro lado
\usepackage[margin=1in]{geometry}
\theoremstyle{plain}
  \newtheorem{teor}{Theorem}[section]
  
  \newtheorem{prop}[teor]{Proposition}
  \newtheorem{lema}[teor]{Lemma}
  
\theoremstyle{remark}  
  \newtheorem{rem}[teor]{Remark}
\theoremstyle{plain}

\newtheoremstyle{hyp}{}{}{\itshape}{}{}{}{0pt}{}
\theoremstyle{hyp}

\DeclareMathAlphabet{\mathpzc}{OT1}{pzc}{m}{it}

\DeclareMathOperator*{\infess}{inf\, ess}
\DeclareMathOperator{\deri}{D}
\DeclareMathOperator{\dist}{dist}

\DeclareMathOperator*{\argmin}{arg\, min}

\DeclareMathOperator{\Lp}{L}
\DeclareMathOperator{\trans}{T}
\DeclareMathOperator{\comp}{c} 
\DeclareMathOperator{\loc}{loc}
\DeclareMathOperator{\hol}{C}
\DeclareMathOperator{\tr}{tr}
\DeclareMathOperator{\sob}{W}

\DeclareMathOperator{\expo}{e}
\DeclareMathOperator{\sop}{supp}

\newcommand{\w}{\mathpzc{w}}
\newcommand{\sol}{\mathpzc{u}}

\newcommand{\E}{\mathbbm{E}}
\newcommand{\tmt}{\mathpzc{t}}
\newcommand{\tms}{\mathpzc{s}}
\newcommand{\set}{\mathcal{O}}
\newcommand{\dif}{\mathcal{L}}

\newcommand{\R}{\mathbbm{R}}
\newcommand{\N}{\mathbbm{N}}
\newcommand{\uno}{\mathbbm{1}}
\newcommand{\der}{\mathrm{d}}
\newcommand{\F}{\mathbbm{F}}
\newcommand{\Pro}{\mathbbm{P}}
\newcommand{\BE}{\begin{equation}}

\newcommand{\EE}{\end{equation}}
\newcommand {\BA}{\begin{align}}
\newcommand{\EA}{\end{align}}
\newcommand{\eqdef}{\raisebox{0.4pt}{\ensuremath{:}}\hspace*{-1mm}=}
\newcommand{\defeq}{=\hspace*{-1mm}\raisebox{0.4pt}{\ensuremath{:}}}

\title{{ On a mixed singular/switching control problem with multiples regimes \footnote{This study has been funded by the Russian Academic Excellence Project `5-100'.} }}
\author{ Mark Kelbert\quad\quad Harold A. Moreno-Franco\footnote{Corresponding author: hmoreno@hse.ru}\\ {\small\it Department of Statistics and Data Analysis}\\ {\small\it Laboratory of Stochastic Analysis and its Applications}\\{\small\it National Research University Higher School of Economics} \\ {\small\it Moscow, Russia}}
\date{}
\begin{document}	
%\listoftodos %contenido de las notas	
%\layout
%\fontfamily{ptm}\selectfont
\maketitle
\begin{abstract}
	\noindent This paper studies {a} mixed singular/switching stochastic control problem for a multidimensional diffusion with multiples regimes  on a bounded domain. Using  probabilistic, partial differential equation (PDE) and penalization techniques, we show that the value function associated with this problem  agrees with the  solution to a Hamilton-Jacobi-Bellman (HJB) equation. In that way, we see that the  regularity of the  value function is  $\hol^{0,1}\cap\sob^{2,\infty}_{\loc}$. 
\end{abstract}
\section{Introduction}
Singular and switching stochastic control problems have been of great research interest in  control theory due to their  applicability to diverse problems of finance, economy, biology and other fields; see, e.g., \cite{KM2019,P2009} and the reference therein. For that reason, new techniques and problems are continuously developed. One of those problems, called mixed singular/switching stochastic control problem, concerns with applying both singular and switching  controls in optimal way  on  some stochastic process that can change regime. Within a regime, a singular control  is executed.  

This paper is mainly  concerned  with determining the regularity of the value function in a mixed singular/switching stochastic control problem for a multidimensional diffusion with multiples regimes   on a bounded domain. Our study is focused on the stochastic controlled process $(X^{\xi,\varsigma},I^{\varsigma})$ that evolves as:
\begin{equation}\label{es1}
\begin{split}
&X^{\xi,\varsigma}_{\tmt}=X^{\xi,\varsigma}_{\tilde{\tau}_{i}} -\displaystyle\int_{\tilde{\tau}_{i}}^{\tmt}b(X^{\xi,\varsigma}_{\tms},\ell_{i})\der \tms+\displaystyle\int_{\tilde{\tau}_{i}}^{\tmt}\sigma(X^{\xi,\varsigma}_{\tms},\ell_{i})\der W_{\tms}-\displaystyle\int_{[\tilde{\tau}_{i},\tmt)}\mathbb{n}_{\tms}\der\zeta_{\tms},\\
&I^{\varsigma}_{\tmt}=\ell_{i}\quad\text{for}\ \tmt\in[{\tilde\tau}_{i},\tilde{\tau}_{i+1})\ \text{ and }\ i\geq0,
\end{split}
\end{equation} 
where $X^{\xi,\varsigma}_{0}=\tilde{x}\in\overline{\set}\subset\R^{d}$, $I^{\varsigma}_{0-}=\tilde{\ell}\in\mathbb{I}\eqdef\{1,2,\dots,m\}$, $\tilde{\tau}_{i}\eqdef\tau_{i}\wedge\tau$, and { $\tau$ represents the first exit time of the process $X$ from the set $\set$.}  Here $W=\{W_{t}:t\geq0\}$ is a $k$-dimensional standard Brownian motion. The couple $(\xi,\varsigma)$ is a stochastic control strategy, defined later on (see  \eqref{cont.2}, \eqref{cont.1}), which is  comprised of a singular control  $\xi\eqdef(\mathbb{n},\zeta)\in\R^{d}\times\R_{+}$ and a switching control  $\varsigma\eqdef(\tau_{i},\ell_{i})_{i\geq0}$ with  $\tau_{i}$ a  stopping time  and $\ell_{i}\in\mathbb{I}$ for $i\geq1$.   The process $(\xi,\varsigma)$ is chosen in such a way that it will minimize the cost criterion 
\begin{align}\label{esd1.1.1}
V_{\xi,\varsigma}(\tilde{x},\tilde{\ell})&\eqdef\sum_{i\geq0}\E_{\tilde{x},\tilde{\ell}}\big[\expo^{-r(\tau_{i+1})}\vartheta_{\ell_{i},\ell_{i+1}}\uno_{\{\tau_{i+1}<\tau\}}\big]\notag\\
&\quad+\sum_{i\geq0}\E_{\tilde{x},\tilde{\ell}}\bigg[\int_{[\tilde{\tau}_{i},\tilde{\tau}_{i+1})}\expo^{-r(\tms)}[h_{\ell_{i}}(X^{\xi,\varsigma}_{\tms})\der \tms+g_{\ell_{i}}(X^{\xi,\varsigma}_{\tms-})\circ\der\zeta_{\tms}]\bigg].
\end{align}
%\red{where $h_{\ell_{i}}(X_{\tms}^{\xi,\varsigma})$ is some running cost function depending on the current state and the
%regime $\ell_{i}$ at time $\tms$, and $g_{\ell_{i}}(X^{\xi,\varsigma}_{\tms-})\circ\der\zeta_{\tms}$ is the cost for executing the control $\xi$  on the process $X^{\xi,\varsigma}$ at time $\tms$  while this is on the regime $\ell_{i}$. The cost for switching from regime $\ell_{i}$ to $\ell_{i+1}$ is represented  by $\vartheta_{\ell_{i},\ell_{i+1}}$.}
%where $r$ is a discount rate and  $\vartheta_{\ell_{i},\ell_{i+1}}\geq0$, $h_{\ell_{i}}$ and $g_{\ell_{i}}\circ\der\zeta$ are the cost for switching from regime $\ell_{i}$ to $\ell_{i+1}$, the running cost and  the cost for executing the singular control $\xi$, respectively. 

One of  {the} main  goals of this paper  is to verify that the value function
\begin{equation}\label{vf1.0}
V_{\tilde{\ell}}(\tilde{x})\eqdef\inf_{\xi,\varsigma}V_{\xi,\varsigma}(\tilde{x},\tilde{\ell}),\ \text{for}\ (\tilde{x},\tilde{\ell})\in\overline{\set}\times \mathbb{I},
\end{equation}
is in $\hol^{0,1}\cap\sob^{2,p}_{\loc}$; see Theorem \ref{M1}. 

Previous to this paper,  the  problem \eqref{vf1.0} has been studied   extensively (both in bounded and unbounded set domains) for separate cases: when the process $X^{\xi,\varsigma}$ cannot change its regime or when the singular control $\xi$ is not executed on $X^{\xi,\varsigma}$; see, e.g., \cite{KM2019,lenh1983,P2009,Y1988} and the reference therein. It is   natural to study  stochastic control problems where both  controls are involved, and see how they can be applied.

For example, both (singular and switching) controls    have been  used recently   in    the study of interactions between {dividend} and investment policies. In that context, a firm that operates under an uncertain environment and risk constraints wants to determine an optimal control on the dividend and investment {policies} (see  \cite{AM2015,  CLS2013, CGL2017, pham2008}). { The authors  assumed that the cash reserve process of a firm  switches between $m$-regimes governed by  a  Brownian motion with the same volatility  {but} different  drifts \cite{CLS2013, pham2008}, a two-dimensional Brownian motion with the same stochastic volatility  {but} different stochastic  drifts \cite{CGL2017}, or different compound Poisson processes with drifts \cite{AM2015}. The costs and benefits for switching  regime are made automatically in the firm's cash reserve and   they are not considered in the expected returns.}

{The optimal dividend/investment policy strategy problem  mentioned above was studied} on the whole spaces $\R$ or $\R^{2}$; see \cite{AM2015,  CLS2013, pham2008} and  \cite{CGL2017}, respectively, and using viscosity solution approaches, {the authors} got qualitative descriptions of the value function 
\begin{equation}\label{py1}
\sup_{\zeta}\bigg\{\E_{\tilde{x},\tilde{\ell}}\bigg[\int_{0}^{\tau}\expo^{-ct}\der\zeta_{t}\bigg]\bigg\},\quad\text{with $c>0$ a fixed constant.}
\end{equation}
It must be highlighted  that the explicit or quasi-explicit solution of the optimal strategy for the first two cases mentioned in the previous paragraph (see \cite{ CLS2013, CGL2017, pham2008}), has not been found yet and still remains an open problem. For the case when the cash reserve process switches between $m$-regimes governed by different Poisson processes with drifts,  Azcue and Muler \cite{AM2015}  proved that there exists an optimal dividend/switching strategy, which is stationary with a band structure.

{In addition,   Guo and Tomecek  \cite{GP2008}  studied a connection between   singular control of finite variation and optimal switching problems.}

Notice that the mixed singular/switching  stochastic control problem, proposed in \eqref{es1}--\eqref{vf1.0}, is defined on an open bounded domain $\set\subset\R^{d}$ and the costs for switching of regime,  represented by {$\vartheta_{\ell,\kappa}$}, are considered in the functional costs $V_{\xi,\varsigma}$.  Additionally, it must be highlighted that our problem is given in a general way, meaning that for every regime $\ell$, the drift and the volatility of the process $X^{\xi,\varsigma}$ are stochastic, and there are two types of costs: $g_{\ell}(X^{\xi,\varsigma})\circ\der\zeta$ when the singular control $\xi$  is exercised or $h_{\ell}(X^{\xi,\varsigma})$ if not. For more details about the cost $g_{\ell}(X^{\xi,\varsigma})\circ\der\zeta$, see the next section.

The contribution of this paper is twofold.
\begin{enumerate}

\item We characterize the solution $u$ to the HJB equation, which is closely related to the  value function $V$ given in \eqref{vf1.0},   as   limit of a sequence of solutions to another { system of variational inequalities} that is related to  $\varepsilon$-penalized absolutely stochastic continuous/switching control problems; see Subsection \ref{penal1}. 

\item We give an explicit representation of the optimal strategy of these $\varepsilon$-penalized control problems. Then, by probabilistic methods and  construction of an approximating   sequence of solutions  mentioned previously, we verify that the value function \eqref{vf1.0} and the solution $u$ to  the HJB equation \eqref{esd5} agree on $\overline{\set}$, showing that $V_{\ell}$ belongs to $\hol^{0,1}(\overline{\set})\cap\sob^{2,p}_{\loc}(\overline{\set})$; see Section \ref{Pp1}.
\end{enumerate}

%Stochastic optimization control problems that involved both the singular control and switching are becoming
%timely in many applied fields including the Mathematical Finance. A popular example in this direction is the optimal
%dividend distribution policy of insurance company with a reversible investment into a new technology, see \cite{pham2008} and the references therein.
 %proving the existence and uniqueness theorem for a singular
%Hamilton-Jacobi- Bellman (HJB) equation in a rather general setting of coupled variational inequalities. Moreover, we provide a probabilistic representation
%of the solution as the optimal cost functional in a singular stochastic control problem. 

{The rest of this} document is organized as follows: in Section \ref{apr1} we consistently formulate  the stochastic control problem  studied here; see \eqref{vf1},  and give the assumptions for obtaining the main results of this paper; Theorems \ref{M1}, \ref{verf2}. Also we introduce the $\varepsilon$-penalized absolutely continuous/switching control problem and its HJB equation \eqref{p13.0.1.0}. Then, in Section   \ref{euHJB1},  a non-linear partial differential system (NPDS) is introduced, where some \textit{\`a priori} estimates are given. Afterwards,  using  Lemma \ref{Lb1}, Proposition \ref{princ1.0}, Arzel\`a-Ascoli  compactness criterion and the reflexivity of $\Lp^{p}_{\loc}(\set)$;  see \cite[Section C.8, p. 718]{evans2} and \cite[Thm. 2.46, p. 49]{adams} respectively, we prove the existence, regularity and uniqueness of the solution $u^{\varepsilon}$ to \eqref{p13.0.1.0}; see Subsection \ref{prop1}. In Subsection \ref{pro2}, Theorem \ref{M1} is verified whose proof is done by {selecting} a sub-sequence of $\{u^{\varepsilon}\}_{\varepsilon\in(0,1)}$.  Then, in Section \ref{Pp1}, a verification lemma of the  $\varepsilon$-penalized absolutely continuous/switching control problem  is presented. This  lemma is divided into two parts: Lemmas  \ref{lv1} and \ref{optim1}. Afterwards, we give the proof of Theorem \ref{M1}.    In section \ref{conc},   we draw  {our}  conclusions and discuss possible extensions of this paper. Finally, the proofs of  Lemma \ref{Lb1} and Proposition \ref{princ1.0} are given in the appendix. To finalize this section, let us say that the notations and definitions of the function spaces that are used in this paper are standard and the reader can find them in \cite{adams,Cs1,evans2, garroni,gilb}.

\section{Model formulation, assumptions and main results}\label{apr1} 

Let $W=\{W_{\tmt}:\tmt\geq0\}$ be the $k$-dimensional standard Brownian motion  defined on a complete probability space  $(\Omega,\mathcal{F},\Pro)$. Let $\F=\{\mathcal{F}_{\tmt}\}_{\tmt\geq0}$ be the filtration  generated by $W$. The process $(X^{\xi,\varsigma},I^{\varsigma})$ is governed by the SDE \eqref{es1}, { where the parameters $b_{\ell}\eqdef b(\cdot,\ell):\overline{\set}\longrightarrow\R^{d}$ and $\sigma_{\ell}\eqdef\sigma(\cdot,\ell):\overline{\set}\longrightarrow\R^{d}\times\R^{k}$, with $\ell\in\mathbb{I}$ fixed,  satisfy appropriate conditions to ensure the well-definiteness of the stochastic differential equation (SDE) \eqref{es1}; see Subsection \ref{as1}}.   The control process $(\xi,\varsigma)$ is in $\mathcal{U}\times\mathcal{S}$ where the singular control $\xi=(\mathbb{n},\zeta)$ belongs to the class $\mathcal{U}$ of admissible controls that satisfy
\begin{equation}\label{cont.1}
\begin{cases}
(\mathbb{n}_{\tmt},\zeta_{\tmt})\in\R^{d}\times\R_+,\ t\geq0,\ \text{such that}\ X^{\xi,\varsigma}_{\tmt}\in\set\,\ \tmt\in[0,\tau),\\ 
(\mathbb{n},\zeta)\ \text{is adapted to the filtration}\  \F,\\
\zeta_{0-}=0\ \text{and}\ \zeta_{\tmt}\ \text{is non-decreasing and is right continuous} \\
\text{with left hand limits,}\ \tmt\geq0,\ \text{and }\ |\mathbb{n}_{\tmt}|=1\ {\der\zeta_{\tmt}\text{-a.s.},\ \tmt\geq0} ,
\end{cases}
\end{equation}
and the switching   control process   $\varsigma\eqdef(\tau_{i},\ell_{i})_{i\geq0}$ belongs to the class $\mathcal{S}$ of switching regime sequences that satisfy
\begin{equation}\label{cont.2}
\begin{cases}
\varsigma\ \text{is a sequence of $\F$-stopping times and regimes in $\mathbb{I}$, i.e,}\\ 
\text{$\varsigma =(\tau_{i},\ell_{i})_{i\geq0}$ is such that $0=\tau_{0}\leq\tau_{1}<\tau_{2}<\cdots$, $\tau_{i}\uparrow\infty$ as $i\uparrow \infty$ $\Pro$-a.s.,}\\
\text{ and for each $i\geq0$, $ \ell_{i}\in\mathbb{I}\eqdef\{1,2,\dots,m\}$.}
\end{cases}
\end{equation}
Notice that  $I^{\varsigma}$   is a c\`adl\`ag process that starts in $\tilde{\ell}$ which has a possible jump at $0$, i.e., if $\tau_{1}=0$, $I^{\varsigma}_{\tau_{1}}=\ell_{1}$. Without the influence of the singular control $\xi$ in $X^{\xi,\varsigma}$, i.e. $\zeta\equiv0$,  the infinitesimal generator of $X^{\xi,\varsigma}$, within the regime $\ell\in\mathbb{I}$, is given by 
\begin{equation}
\dif_{\ell} u_{\ell} =\tr[a_{\ell} \deri^{2}u_{\ell} ]-\langle b _{\ell} ,\deri^{1}u_{\ell} \rangle,\label{eq2}
\end{equation}
with $a_{\ell}=(a_{\ell\, ij})_{d\times d}$ is such that $a_{\ell\,ij}\eqdef\frac{1}{2}[\sigma_{\ell}\sigma_{\ell}^{\trans}]_{ij}$. Here $|\cdot|$, $\langle\cdot,\cdot\rangle$ and $\tr[\,\cdot\,]$  are the Euclidean norm, the inner product, and  the  {matrix trace}, respectively.
\begin{rem}
	Taking $\Delta X^{\xi,\varsigma}_{\tmt}\eqdef X^{\xi,\varsigma}_{\tmt}-X^{\xi,\varsigma}_{\tmt-}$,  with $(\xi,\varsigma)\in\mathcal{U}\times\mathcal{S}$ , we observe  $X^{\xi,\varsigma}_{\tmt}=X^{\xi,\varsigma}_{\tmt-}-\mathbb{n}_{\tmt}\Delta\zeta_{\tmt}\in \set$ for $\tmt\in[\tilde{\tau}_{i},\tilde{\tau}_{i+1})$.
\end{rem}	
Given  the initial state  $(\tilde{x},\tilde{\ell})\in\overline{\set}\times\mathbb{I}$  and the  control   $(\xi,\varsigma)\in\mathcal{U}\times\mathcal{S}$,  the \textit{functional cost} of the controlled process $(X^{\xi,\varsigma},I^{\varsigma})$  is defined by \eqref{esd1.1.1}, 
where $\E_{\tilde{x},\tilde{\ell}}$ is the expected value associated with $\Pro_{\tilde{x},\tilde{\ell}}$,    the probability law of $(X^{\xi,\varsigma},I^{\varsigma})$ when it starts at $(\tilde{x},\tilde{\ell})$,  and
\begin{align}
r(\tmt)&\eqdef\int_{0}^{\tmt}c(X^{\xi,\varsigma}_{\tms},I^{\varsigma}_{\tms})\der\tms={\sum_{i\geq0}\int_{\tmt\wedge\tilde{\tau}_{i}}^{\tmt\wedge\tilde{\tau}_{i+1}}c_{\ell_{i}}(X^{\xi,\varsigma}_{\tms})\der\tms},\label{eq0.1}\\
\int_{[ \tilde{\tau}_{i},\tmt\wedge \tilde{\tau}_{i+1})}\expo^{-r( \tms)} g_{\ell_{i}}   (X^{\xi,\varsigma}_{ \tms-})\circ\der\zeta_{ \tms}&\eqdef\int_{{\tilde{\tau}}_{i}}^{\tmt\wedge{\tilde{\tau}}_{i+1}}\expo^{-r(\tms)} g_{\ell_{i}}   (X^{\xi,\varsigma}_{ \tms})\der\zeta_{ \tms}^{\comp}\notag\\
&\quad+\sum_{ \tilde{\tau}_{i}\leq\tms<\tmt\wedge \tilde{\tau}_{i+1}}\expo^{-r(\tms)}\Delta\zeta_{ \tms}\int_{0}^{1} g_{\ell_{i}}    (X^{\xi,\varsigma}_{ \tms-}-\lambda \mathbb{n}_{ \tms}\Delta\zeta_{ \tms})\der\lambda, \label{eq0.1.1}
\end{align}
where $\zeta^{\comp}$ denotes the continuous part of $\zeta$, $c_{\ell}\eqdef c(\cdot,\ell)$ is a  positive continuous function from $\overline{\set}$ to $\R$,  and $ h_{\ell}\eqdef h(\cdot,\ell)$, $ g_{\ell}\eqdef g(\cdot,\ell)$ are non-negative continuous functions from $\overline{\set}$ to $\R$.   Notice that within the regime $\ell\in\mathbb{I}$  the singular control $\xi=(\mathbb{n},\zeta)$ generates two types of costs. One of them is when $\xi$ continuously controls   the process $X^{\xi,\varsigma}$ by $\zeta^{\comp}$, the other  is when $\xi$ controls $X^{\xi,\varsigma}$ by jumps of $\zeta$.  While $X^{\xi,\varsigma}$ is in the regime $\ell$, the term $\int_{0}^{1} g_{\ell}   (X^{\xi,\varsigma}_{\tms-}-\lambda \mathbb{n}_{\tms}\Delta\zeta_{\tms})\der\lambda$ represents the cost for using the   jump  $\Delta\zeta_{\tms}\neq0$ with direction $-\mathbb{n}_{\tms}$ on $X^{n,\zeta}_{\tms-}$ at time $ \tms$.  
\begin{rem}
	If $g_{\ell}\equiv a$, with $a$ a positive constant, Equation  \eqref{eq0.1.1} is reduced to the following form 
	\begin{equation*}
	\int_{[ \tilde{\tau}_{i},\tmt\wedge \tilde{\tau}_{i+1})}\expo^{-r( \tms)} g_{\ell_{i}}   (X^{\xi,\varsigma}_{ \tms-})\circ\der\zeta_{ \tms}=a\int_{[ \tilde{\tau}_{i},\tmt\wedge \tilde{\tau}_{i+1})}\expo^{-r( \tms)}\der\zeta_{ \tms}.
	\end{equation*}
	It means that the costs for controlling $X^{\xi,\varsigma}$ by $\zeta^{\comp}$ or $\Delta \zeta\neq0$ are the same.  
\end{rem}

The cost for switching from the regime $\ell$ to $\kappa$ is given by a constant $\vartheta_{\ell,\kappa}\geq0$ and we assume that 
\begin{equation}\label{eq1}
\vartheta_{\ell_{1},\ell_{3}}\leq\vartheta_{\ell_{1},\ell_{2}}+\vartheta_{\ell_{2},\ell_{3}},\ \text{for}\ \ell_{3}\neq\ell_{1}, \ell_{2}, 
\end{equation}
which means that it is cheaper to switch directly from regime $\ell_{1}$ to $\ell_{3}$ than  using the intermediate regime $\ell_{2}$.  Additionally, we assume that there is  no {\it loop of zero cost}, i.e., 
	\begin{equation}\label{l1}
	\text{no family of regimes}\  \{\ell_{0},\ell_{1}, \dots,\ell_{n},\ell_{0}\}\ \text{such that}\  \vartheta_{\ell_{0},\ell_{1}}=\vartheta_{\ell_{1},\ell_{2}}=\cdots=\vartheta_{\ell_{n},\ell_{0}}=0. 
	\end{equation} 
The \textit{value function} is defined by 
\begin{equation}\label{vf1}
V_{\tilde{\ell}}(\tilde{x}) \eqdef\inf_{(\xi,\varsigma)\in\mathcal{U}\times\mathcal{S}}V_{\xi,\varsigma}(\tilde{x},\tilde{\ell}),\ \text{for}\ (\tilde{x},\tilde{\ell})\in\overline{\set}\times \mathbb{I}.
\end{equation}
From \eqref{l1} and by dynamic programming principle,  we   identify heuristically   that the value function $V$ is associated with  the following HJB equation  
	\begin{equation}\label{esd5}
	\max\{[c_{\ell}-\dif_{\ell}]u_{\ell}-  h_{\ell},|\deri^{1}u_{\ell}|- g_{\ell},  u_{\ell}-\mathcal{M}_{\ell}u\}= 0,\ \text{ in}\ \set,\quad
	\text{s.t.}\ 
	u_{\ell}=0 \ \text{on}\  \partial\set,
	\end{equation}
	where $u=(u_{1},\dots,u_{m}):\overline{\set}\longrightarrow\R^{m}$ and, for $\ell\in\mathbb{I}$,  the operators $\dif_{\ell}$ is as in \eqref{eq2} and $\mathcal{M}_{\ell}$ is defined by
	\begin{align}
	\mathcal{M}_{\ell}u&=\min_{\kappa\in \mathbb{I}\setminus\{\ell\}}\{u_{\kappa}+\vartheta_{\ell,\kappa}\}.\label{p6.0}
	\end{align}

\subsection{Assumptions and main results}\label{as1}
First, let us give the necessary conditions to guarantee the existence and uniqueness of the solution $u$ to the HJB equation \eqref{esd5} on the space  $\hol^{0,1}(\overline{\set})\cap\sob^{2,\infty}_{\loc}(\set)$.
{\it
\begin{enumerate}[label=(H\arabic*),ref=H\arabic*]

\item  \label{h3} The switching costs sequence  $\{\vartheta_{\ell,\kappa}\}_{\ell,\kappa\in\mathbb{I}}$ is such that  $\vartheta_{\ell,\kappa}\geq 0$ and \eqref{eq1} and \eqref{l1} hold.  
\item\label{h0}  The domain set $\set$ is an open and bounded set such that its boundary $\partial\set$ is of class $\hol^{4,\alpha'}$, with $\alpha'\in(0,1)$ fixed.

\end{enumerate}
Let $\ell$ be in $\mathbb{I}$. Then:
\begin{enumerate}[label=(H\arabic*),ref=H\arabic*]
\setcounter{enumi}{2}	
\item\label{a4}The functions $h_{\ell},g_{\ell}\in\hol^{2,\alpha'}(\overline{\set})$ are non-negative and  $||h_{\ell}||_{\hol^{2,\alpha'}(\overline{\set})}$, $||g_{\ell}||_{\hol^{2,\alpha'}(\overline{\set})}$  are bounded by some finite positive constant $\Lambda$. 
\item\label{h2} Let $\mathcal{S}(d)$ be the set of $d\times d$ symmetric matrices. The coefficients of the differential part of $ \dif_{\ell}$, $a_{\ell}=(a_{\ell\,ij})_{d\times d}:\overline{\set}\longrightarrow\mathcal{S}(d)$, $b_{\ell}=(b_{\ell\,1},\dots,b_{\ell\,d}):\overline{\set}\longrightarrow\R^{d}$ and $c_{\ell}:\overline{\set}\longrightarrow\R$, are such that  $a_{\ell\,ij},b_{\ell\,i},c_{\ell}\in\hol^{2,\alpha'}(\overline{\set})$, $c_{\ell}>0$ on $\set$ and $||a_{\ell\,ij}||_{\hol^{2,\alpha'}(\overline{\set})},||b_{\ell\,i}||_{\hol^{2,\alpha'}(\overline{\set})},||c_{\ell}||_{\hol^{2,\alpha'}(\overline{\set})}$ are bounded by some finite positive constant $\Lambda$. We assume that there exists a real number $\theta>0$  such that  
		\begin{equation}\label{H2}
		\langle a_{\ell}(x) \zeta,\zeta\rangle\geq \theta|\zeta|^{2},\ \text{for all}\  x\in\overline{\set},\ \zeta\in \R^{d}.
		\end{equation} 
\end{enumerate}
}

Under assumptions \eqref{h3}--\eqref{h2},  the first main {goal} obtained in this document is as follows.
\begin{teor}\label{M1}
The HJB equation \eqref{esd5} has a unique  non-negative strong solution  {(in the almost everywhere sense)} $u=(u_{1},\dots,u_{m})$ where  $ u_{\ell}\in\hol^{0,1}(\overline{\set})\cap\sob^{2,\infty}_{\loc}(\set)$ for each $\ell\in\mathbb{I}$.
\end{teor}

\begin{rem}
	Previous to this work, the HJB equation \eqref{esd5} was studied by Lenhart and Belbas \cite{lenh1983} in 1983, in absence of  $|\deri^{1}u_{\ell}|-g_{\ell}$,  proving that the unique solution to their HJB equation belongs to $\sob^{2,\infty}(\set)$. This HJB equation is related to optimal switching stochastic control problems.  Afterwards, Yamada \cite{Y1988} analyzed the HJB equation for \eqref{vf1}  when  $\vartheta_{\ell,\kappa}=0$ and $g_{\ell}=g$ for $\ell,\kappa\in\mathbb{I}$. This case is an example of a system with loop zero costs  whose  HJB equation  has the form
	\begin{equation}\label{HJB1}
	\max\Big\{\max_{\ell\in\mathbb{I}}\{[c_{\ell}-\dif_{\ell}]\tilde{u}-h_{\ell}\},|\deri^{1}\tilde{u}|-g\Big\}=0 \ \text{in}\ \set,\quad\text{s.t.}\ \tilde{u}=0\ \text{on}\ \partial\set.
	\end{equation}
	Yamada showed that there exists a solution $\tilde{u}$ to \eqref{HJB1} in $\hol^{0,1}(\overline{\set})\cap\sob^{2,\infty}_{\loc}(\set)$ that does not  dependent on the elements of $\mathbb{I}$,  and assuming $g>0$, he guaranteed that $\tilde{u}$ belongs to $\hol^{1}(\set)\cap\hol(\overline{\set})$ and   is the unique viscosity solution to \eqref{HJB1}. Comparing those papers with ours, we see that the results presented here are more general than in \cite{lenh1983} and, under the assumption of  absence of loop zero cost  in the system, we guarantee the unique solution $u$ to \eqref{esd5} in the almost everywhere sense. 
\end{rem} 

 In addition to the statement in \eqref{h0}, we need to assume that the domain set is {convex, }  which will permit to verify that the value function $V$ and the solution $u$ to  \eqref{esd5} agree on $\set$.
{\it\begin{enumerate}[label=(H\arabic*),ref=H\arabic*]
		\setcounter{enumi}{4}
		\item\label{h5} The domain set $\set$ is an open, convex and bounded set such that its boundary $\partial\set$ is of class $\hol^{4,\alpha'}$, with $\alpha'\in(0,1)$ fixed.
\end{enumerate} }
Under assumptions \eqref{h3} and \eqref{a4}--\eqref{h5},  the second main {goal}  obtained in this document is as follows.
\begin{teor}\label{verf2}
	Let $V$ be the value function given by \eqref{vf1}. Then $V_{\tilde{\ell}}(\tilde{x})=u_{\tilde{\ell}}(\tilde{x})$ for $(\tilde{x},\tilde{\ell})\in\overline{\set}\times\mathbb{I}$. 
\end{teor}

To finalize, let us comment about the assumptions mentioned  above. Under \eqref{h3}--\eqref{h2} and  using the Schaefer fixed point theorem (see, i.e., \cite[Thm. 4 p. 539]{evans2}), for each $\varepsilon,\delta\in(0,1)$ fixed, we guarantee the existence and uniqueness of the classical non-negative solution $u^{\varepsilon,\delta}$ to the NPDS \eqref{NPD.1}; see Proposition \ref{princ1.0} and the appendix for its proof. Also, those assumptions  are required to show some \`a priori estimates of $u^{\varepsilon,\delta}$, which are independent of $\varepsilon,\delta$; see Lemma  \ref{Lb1}. Since \eqref{h3} holds, $c_{\ell}>0$ on $\overline{\set}$ and using Lemma \ref{Lb1}, we verified that $u^{\varepsilon}$ is the unique non-negative solution to the HJB equation \eqref{pc1}; see Subsection \ref{prop1}. Once again making use of \eqref{h3} and $c_{\ell}>0$ on $\overline{\set}$, it is  {proven} that the solution $u$ to the HJB equation \eqref{esd5} is unique; see Subsection \ref{pro2}. Finally, Assumption \eqref{h5} helps to check that the value function $V$, given in  \eqref{Vfp1}, and the solution $u$ to the HJB equation \eqref{esd5} agree on $\overline{\set}$.  

\subsection{$\varepsilon$-penalized absolutely continuous/switching control problem}\label{penal1}
To prove the statements above; Theorems  \ref{M1}--\ref{verf2}, firstly we study an $\varepsilon$-penalized absolutely continuous/switching control problem that is closely related to the value function problem seen previously.  At the same time, the solution to the HJB equation related to this stochastic control problem (see Proposition \ref{princ1.1})  {helps}  to guarantee the existence and regularity  of the solution to the variational inequalities system \eqref{esd5}. The verification  lemma for this part, which is divided into two lemmas (Lemmas \ref{lv1} and \ref{optim1})   will be presented below to provide  the proofs of Theorem \ref{M1} and Proposition \ref{princ1.1}.

Define the penalized controls set $\mathcal{U}^{\varepsilon}$ in the following way
\begin{equation}\label{p1}
\mathcal{U}^{\varepsilon}\eqdef\{ \xi=(\mathbb{n},\zeta)\in\mathcal{U}: \zeta_{\tmt}\ \text{is absolutely continuous,}\ 0\leq\dot{\zeta}_{\tmt}\leq 2C/\varepsilon\}, 
\end{equation} 
with  $\varepsilon\in(0,1)$ fixed, where $C$ is some  fixed positive constant independent of $\varepsilon$. For each $(\tilde{x},\tilde{\ell})\in\overline{\set}\times\mathbb{I}$ and $(\xi,\varsigma)\in\mathcal{U}^{\varepsilon}\times\mathcal{S}$,  the process $X^{\xi, \varsigma}_{\tmt}=\{X^{\xi, \varsigma}_{\tmt}:\tmt\geq0\}$ evolves as 
\begin{align}
&X^{\xi, \varsigma}_{\tmt}=X^{\xi, \varsigma}_{\tilde{\tau}_{i}} -\displaystyle\int_{\tilde{\tau}_{i}}^{\tmt}[b(X^{\xi, \varsigma}_{\tms},\ell_{i})+\mathbb{n}_{\tmt}\dot{\zeta}_{\tmt}]\der \tms+\displaystyle\int_{\tilde{\tau}_{i}}^{\tmt}\sigma(X^{\xi, \varsigma}_{\tms},\ell_{i})\der W_{\tms},\label{esd3.1.0.0}\\
&I_{\tmt}=\ell_{i}\quad\text{for}\ \tmt\in[\tilde{\tau}_{i},\tilde{\tau}_{i+1})\ \text{ and }\ i\geq0, \label{esd3.1.0.1}
\end{align} 
where $\tilde{\tau}_{i}=\tau_{i}\wedge\tau$ and $\tau=\inf\{ \tmt>0:X^{\xi, \varsigma}_{\tmt}\notin\set\}$. Before introducing the corresponding functional cost $\mathcal{V}_{\xi,\varsigma}$  of $(\xi,\varsigma)\in\mathcal{U}^{\varepsilon}\times\mathcal{S}$, let us define the penalization function $\psi_{\varepsilon}$. Consider $\varphi$ as a function from $\R$ to itself that is in $\hol^{\infty}(\R)$ and 
\begin{equation}\label{p12.1}
\begin{split}
&\varphi(t)=0,\  t\leq0,\quad
\varphi(t)>0,\ t>0,\\
&\varphi(t)=t-1,\  t\geq2,\quad\varphi'(t)\geq0,\ \varphi''(t)\geq0.
\end{split}
\end{equation} 
Then, $\psi_{\varepsilon}$ is taken as $\psi_{\varepsilon}(t)=\varphi(t/\varepsilon)$, for each $\varepsilon\in(0,1)$. Also, for each $(x,\ell)\in\overline{\set}\times\mathbb{I}$ fixed, we  define the Legendre transform of $H_{\ell}^{\varepsilon}(\gamma,x)\eqdef H^{\varepsilon}(\gamma,x,\ell)\eqdef\psi_{\varepsilon}(|\gamma|^{2}-g_{\ell}(x)^{2})$ by 
\begin{equation*}%\label{p12.2}
l_{\ell}^{\varepsilon}(y,x)\eqdef l^{\varepsilon}(y,x,\ell)\eqdef \sup_{\gamma\in\R^{d}}\left\{\langle \gamma,y \rangle-H_{\ell}^{\varepsilon}(\gamma,x)\right\},\quad \text{for}\ y\in\R^{d}.
\end{equation*} 
Notice that, for each $(x,\ell)\in\overline{\set}\times\mathbb{I}$ fixed,  $H_{\ell}^{\varepsilon}(\gamma,x)$ is a $\hol^{2}$ and convex function with respect to the variable $\gamma\in\R^{d}$, since $\psi_{\varepsilon}\in\hol^{\infty}(\R)$ is convex. The penalized cost of $(\xi,\varsigma)\in\mathcal{U}^{\varepsilon}\times\mathcal{S}$ is defined by 
\begin{align}\label{pen1}
\mathcal{V}_{\xi,\varsigma}(\tilde{x},\tilde{\ell})&=\sum_{i\geq0}\E_{\tilde{x},\tilde{\ell}}\bigg[\expo^{-r(\tau_{i+1})}\vartheta_{\ell_{i},\ell_{i+1}}\uno_{\{\tau_{i+1}<\tau\}}\bigg]\notag\\
&\quad+\sum_{i\geq0}\E_{\tilde{x},\tilde{\ell}}\bigg[\int_{{\tilde{\tau}}_{i}}^{{\tilde{\tau}}_{i+1}}\expo^{-r(\tms)}[h_{\ell_{i}}(X^{\xi, \varsigma}_{\tms})+l^{\varepsilon}_{\ell_{i}}(\dot{\zeta}_{\tms}\mathbb{n}_{\tms},X^{\xi, \varsigma}_{\tms})]\der \tms\bigg],
\end{align}
The value function for this problem is given by
\begin{equation}\label{Vfp1}
V_{\tilde{\ell}}^{\varepsilon}(\tilde{x})\eqdef\inf_{(\xi,\varsigma)\in\mathcal{U}^{\varepsilon}}\mathcal{V}_{\xi,\varsigma}(\tilde{x},\tilde{\ell}),
\end{equation}
whose  corresponding  HJB equation is
\begin{equation}\label{p13.0.1.0}
\begin{split}
\max\bigg\{[c_{\ell}-\dif_{\ell}]  u_{\ell}^{\varepsilon}+ \sup_{y\in\R^{d}}\{\langle \deri^{1}u_{\ell}^{\varepsilon},y\rangle-l_{\ell}^{\varepsilon}(y,\cdot)\}-h_{\ell},u_{\ell}^{\varepsilon}-\mathcal{M}_{\ell}u^{\varepsilon}\bigg\}&= 0 ,\ \text{in}\ \set,\\  
\text{s.t.}\ u_{\ell}^{\varepsilon}&=0,\ \text{on}\ \partial\set ,
\end{split}
\end{equation}
where $\dif_{\ell}, \mathcal{M}_{\ell}$ are as in \eqref{eq2}, \eqref{p6.0}, respectively. Observe that \eqref{p13.0.1.0} can be rewritten as
\begin{equation}\label{pc1}
\begin{split}
\max\left\{[c_{\ell}-\dif_{\ell}]  u_{\ell}^{\varepsilon}+\psi_{\varepsilon}(|\deri^{1}u_{\ell}^{\varepsilon}|^{2}- g_{\ell}^{2})-h_{\ell},u_{\ell}^{\varepsilon}-\mathcal{M}_{\ell}u^{\varepsilon}\right\}&= 0 ,\ \text{in}\ \set,\\ 
\text{s.t.}\ u_{\ell}^{\varepsilon}&=0,\ \text{on}\ \partial\set,
\end{split}
\end{equation} 
because of 
$H^{\varepsilon}_{\ell}(\gamma,x)=\sup_{y\in\R^{d}}\{\langle \gamma,y\rangle-l^{\varepsilon}_{\ell}(y,x)\}$.
%The equality above is true, due to $H^{\varepsilon}_{\ell}(\cdot,x)\in\hol^{2}(\R^{d})$.
Under assumptions \eqref{h3}--\eqref{h2},   the following result  is obtained,  {whose proof is given in the next section; see Subsection \ref{prop1}.}
\begin{prop}\label{princ1.1}
	For each $\varepsilon\in(0,1)$ fixed, there exists a unique non-negative strong solution $u^{\varepsilon}=(u^{\varepsilon}_{1},\dots,u^{\varepsilon}_{m})$ to the HJB equation  \eqref{pc1} where  $u^{\varepsilon}_{\ell}\in\hol^{0,1}(\overline{\set})\cap\sob^{2,\infty}_{\loc}(\set)$ for each $\ell\in\mathbb{I}$.
\end{prop}

\section{Existence and uniqueness to the HJB equations }\label{euHJB1}

This section is devoted to guarantee existence and uniqueness to the  HJB equations \eqref{esd5} and \eqref{pc1}. The solution $u^{\varepsilon}$ to  \eqref{pc1}, with $\varepsilon\in(0,1)$ fixed,  will be constructed  as the limit of a sequence of functions $\{u^{\varepsilon,\delta}\}_{\delta\in(0,1)}$ which are solutions to the following NPDS  
\begin{equation}\label{NPD.1}
\begin{split}
[c_{\ell}-\dif_{\ell}]  u_{\ell}^{\varepsilon,\delta}+ \psi_{\varepsilon}(|\deri^{1} u_{\ell}^{\varepsilon,\delta}|^{2}- g_{\ell}^{2})+\displaystyle\sum_{\kappa\in\mathbb{I}\setminus\{\ell\}}\psi_{\delta}(u_{\ell}^{\varepsilon,\delta}-u_{\kappa}^{\varepsilon,\delta} -\vartheta_{\ell,\kappa})&=  h_{\ell}  ,\ \text{in}\ \set,\\
\text{s.t.}\ 
u^{\varepsilon,\delta}_{\ell}&=0,\ \text{on}\ \partial\set.
\end{split}
\end{equation}
The Schaefer fixed point theorem is employed to guarantee the existence  of the classic solution $u^{\varepsilon,\delta}$ to \eqref{NPD.1}. For that aim we need some  \textit{\`a priory} estimates of $u^{\varepsilon,\delta}$, which will also help to verify the theorems seen in the section above. 
\begin{rem}\label{R1}
	From now on, we  consider cut-off functions $\omega\in\hol^{\infty}_{\comp}( \set)$  which satisfy $0\leq\omega\leq1$, $\omega=1$ on the open ball $B_{\beta r}\subset B_{\beta'r}\subset\set$ and $\omega=0$ on $\set  \setminus B_{\beta' r}$, with $r>0$, $\beta'= \frac{\beta+1}{2}$ and $\beta\in(0,1]$.  It is  {also} assumed that $||\omega||_{\hol^{2}(\overline{B_{\beta r}})}\leq K_{1}$, where $K_{1}>0$ is a constant independent of $\varepsilon$ and $\delta$.
\end{rem}

 Under assumptions \eqref{h3}--\eqref{h2},   the following results  are obtained. The proofs of them can be seen in the appendix.  
\begin{lema}\label{Lb1}
	Let $ u^{\varepsilon,\delta}=(u_{1}^{\varepsilon,\delta},\dots,u_{m}^{\varepsilon,\delta})$ be a vector solution to the NPDS \eqref{NPD.1}, whose components are in $\hol^{4}(\overline{\set}) $. Then, for each $\ell\in\mathbb{I}$, there exist  positive constants $C_{1},C_{2},C_{3}$ independent of $\varepsilon,\delta$ such that  if $x\in\overline{\set}$, then
	\begin{align}
	&0\leq u_{\ell}^{\varepsilon,\delta}(x)\leq C_{1},\label{ap1}\\
	&|\deri^{1}u_{\ell}^{\varepsilon,\delta}(x)|\leq  C_{2},\label{ap2}\\
	&\omega(x)|\deri^{2}u_{\ell}^{\varepsilon,\delta}(x)|\leq C_{3}.\label{ap3}
	\end{align}
\end{lema}
\begin{prop}\label{princ1.0}
	Let $\varepsilon,\delta\in(0,1)$ be fixed. There exists a unique non-negative solution $u^{\varepsilon,\delta}=(u^{\varepsilon,\delta}_{1},\dots,u^{\varepsilon,\delta}_{m})$ to the NPDS  \eqref{NPD.1} where  $u^{\varepsilon,\delta}_{\ell}\in\hol^{4,\alpha'}(\overline{\set})$ for each $\ell\in\mathbb{I}$.
\end{prop}

\begin{rem}
The NPDS \eqref{NPD.1} has been studied in a number of similar problems.  {One of them is when the second term on the  left-hand {side } of \eqref{NPD.1}  does not appear. This equation}  was considered by Lenhart and Belbas in 1983 \cite{lenh1983} to study a HJB equation related to a switching stochastic  control problem when there is no loop of  zero cost. Later, Yamada \cite{Y1988} used {a NPDS similar to} \eqref{NPD.1}, to study the HJB equation \eqref{HJB1}.  {This type of equations  {appears also} in some stochastic singular control problems; see \cite{KM2019} and the references therein.}  
\end{rem}

\subsection{Proof of proposition \ref{princ1.1}}\label{prop1}
Let  $\varepsilon\in(0,1)$ and $\ell\in\mathbb{I}$ be  fixed. Since $u^{\varepsilon,\delta}_{\ell}, \ \partial_{ij}u^{\varepsilon,\delta}_{\ell}$ are bounded, uniformly in $\delta$, on the spaces $(\hol^{1}(\overline{\set}),||\cdot||_{\hol^{1}(\set)})$ and $(\hol(\overline{B_{ r}}),||\cdot||_{\hol(B_{  r})})$, respectively, where $B_{r}\subset\set$; see Lemma \ref{Lb1}, and using Arzel\`a-Ascoli compactness criterion (see \cite[p. 718]{evans2}) and that for each $p\in(1,\infty)$, $(\Lp^{p}(B_{\beta r}),||\cdot||_{\Lp^{p}(B_{r})})$, with $B_{r}\subset\set$, is a reflexive space (see \cite[Thm. 2.46, p. 49]{adams}), it can be proven that there exists a sub-sequence $\{u^{\varepsilon,\delta_{\hat{n}}}_{\ell}\}_{\hat{n}\geq1}$ of $\{u^{\varepsilon,\delta}_{\ell}\}_{\delta\in(0,1)}$  and $u^{\varepsilon}_{\ell}\in\hol^{0,1}(\overline{\set})\cap\sob^{2,\infty}_{\loc}(\set)$ such that
\begin{equation}
\begin{split} \label{conv1}
&\text{$u^{\varepsilon,\delta_{\hat{n}}}_{\ell}\underset{\delta_{\hat{n}}\rightarrow0}{\longrightarrow}u^{\varepsilon}_{\ell}$ in $\hol(\overline{\set})$, $\partial_{i}u^{\varepsilon,\delta_{\hat{n}}}_{\ell}\underset{\delta_{\hat{n}}\rightarrow0}{\longrightarrow}\partial_{i} u^{\varepsilon}_{\ell}$ $\text{in}\ \hol_{\loc}(\set)$,}\\ &\text{$\partial_{ij}u^{\varepsilon,\delta_{\hat{n}}}_{\ell}\underset{\delta_{\hat{n}}\rightarrow0}{\longrightarrow}\partial_{ij}u^{\varepsilon}_{\ell}$,  weakly $ \Lp^{p}_{\loc}(\set)$, for each $p\in(1,\infty)$.}
\end{split}
\end{equation}

We proceed to prove Proposition \ref{princ1.1}.

\begin{proof}[Proof of Proposition \ref{princ1.1}. Existence.]  Taking $\kappa\in\mathbb{I}\setminus\{\ell\}$,  Using \eqref{eq2}, \eqref{NPD.1} and Lemma \ref{Lb1}, we have that $\psi_{\delta}(u^{\varepsilon,\delta}_{\ell}-u^{\varepsilon,\delta}_{\kappa} -\vartheta_{\ell,\kappa})$ is locally bounded, uniformly in $\delta$. From here and \eqref{conv1}, it yields that  $u^{\varepsilon}_{\ell}-u^{\varepsilon}_{\kappa} -\vartheta_{\ell,\kappa}\leq0$ in $\set$. Then, 
\begin{equation}\label{ineq2}
u^{\varepsilon}_{\ell}-\mathcal{M}_{\ell}u^{\varepsilon}\leq0,\quad \text{in}\ \set. 
\end{equation}
Note that the previous inequality is true on the boundary set $\partial\set$, since $u^{\varepsilon,\delta}_{\ell}=u^{\varepsilon,\delta}_{\kappa}=0$ on $\partial\set$ and $\vartheta_{\ell,\kappa}\geq0$. Recall that the operator $\mathcal{M}_{\ell}$ is defined in \eqref{p6.0}. On the other hand, since $u^{\varepsilon,\delta_{\hat{n}}}_{\ell}$ is the unique solution to \eqref{NPD.1},  when $\delta=\delta_{\hat{n}}$, it follows that 
\begin{equation}\label{ineq1}
\int_{B_{r}}\Big\{[c_{\ell}-\dif_{\ell}]  u_{\ell}^{\varepsilon,\delta_{\hat{n}}}+ \psi_{\varepsilon}(|\deri^{1} u_{\ell}^{\varepsilon,\delta_{\hat{n}}}|^{2}- g_{\ell}^{2})\Big\}\varpi\der x\leq  \int_{B_{r}}h_{\ell}\varpi\der x  ,\quad\text{for}\ \varpi\in\mathcal{B}(B_{r}), 
\end{equation}
where
\begin{equation}\label{c1}
\mathcal{B}(A)\eqdef\{\varpi\in \hol^{\infty}_{\comp}(A): \varpi\geq0\ \text{and}\ \sop[\varpi]\subset A\subset\set \}.
\end{equation}
By \eqref{conv1} and  letting $\delta_{\hat{n}}\rightarrow0$ in \eqref{ineq1}, we obtain that 
\begin{equation}\label{ineq3}
[c_{\ell}-\dif_{\ell}]  u_{\ell}^{\varepsilon}+ \psi_{\varepsilon}(|\deri^{1} u_{\ell}^{\varepsilon}|^{2}- g_{\ell}^{2})\leq h_{\ell}\quad  \text{a.e. in $\set$}. 
\end{equation}
From \eqref{ineq2} and \eqref{ineq3},  $\max\left\{[c_{\ell}-\dif_{\ell}]  u_{\ell}^{\varepsilon}+\psi_{\varepsilon}(|\deri^{1} u_{\ell}^{\varepsilon}|^{2}- g_{\ell}^{2})-h_{\ell},u^{\varepsilon}_{\ell}-\mathcal{M}_{\ell}u^{\varepsilon}\right\}\leq 0$  a.e. in $\set$. We shall prove that if 
\begin{equation}\label{ineq5}
u^{\varepsilon}_{\ell}(x^{*})-\mathcal{M}_{\ell}u^{\varepsilon}(x^{*})<0,\quad \text{for some}\ x^{*}\in\set,
\end{equation}
then, there exists a  neighborhood $\mathcal{N}_{x^{*}}\subset\set$  of  $x^{*}$ such that
\begin{equation}\label{ineq6}
[c_{\ell}-\dif_{\ell}]  u_{\ell}^{\varepsilon}+\psi_{\varepsilon}(|\deri^{1} u_{\ell}^{\varepsilon}|^{2}- g_{\ell}^{2})= h_{\ell},\quad  \text{a.e. in $\mathcal{N}_{x^{*}}$.}
\end{equation}
{Assume} \eqref{ineq5} holds. Then, taking  $\kappa\in\mathbb{I}\setminus\{\ell\}$, we see that $u^{\varepsilon}_{\ell}-u^{\varepsilon}_{\kappa} -\vartheta_{\ell,\kappa}\leq u^{\varepsilon}_{\ell}-\mathcal{M}_{\ell}u^{\varepsilon}<0$ at $x^{*}$. Since  $u^{\varepsilon}_{\ell}-u^{\varepsilon}_{\kappa} $ is a continuous function, there exists a ball $B_{r_{\kappa}}\subset\set$ such that  $x^{*}\in B_{r_{\kappa}}$ and $u^{\varepsilon}_{\ell}-u^{\varepsilon}_{\kappa} -\vartheta_{\ell,\kappa}<0$ in $B_{r_{\kappa}}$. From here and defining $\mathcal{N}_{x^{*}}$ as $\bigcap_{\kappa\in\mathbb{I}\setminus\{\ell\}}B_{r_{\kappa}}$, %and $\hat{\vartheta}=\max_{\kappa\in\mathbb{I}\setminus\{\ell\}}\{\vartheta_{\ell,\kappa}\}$, 
we have that $\mathcal{N}_{x^{*}}\subset\set$ is a neighborhood of $x^{*}$ and 
\begin{equation}\label{ineq7}
%u^{\varepsilon}_{\ell}-u^{\varepsilon}_{\kappa}-\hat{\vartheta}\leq 
u^{\varepsilon}_{\ell}-u^{\varepsilon}_{\kappa} -\vartheta_{\ell,\kappa}<0,\quad \text{ in}\ \mathcal{N}_{x^{*}},\ \text{for}\ \kappa\in\mathbb{I}\setminus\{\ell\}.
\end{equation}
Meanwhile, observe that
\begin{equation}\label{ineq8}
||u^{\varepsilon,\delta_{\hat{n}}}_{\ell}-u^{\varepsilon,\delta_{\hat{n}}}_{\kappa} -(u^{\varepsilon}_{\ell}-u^{\varepsilon}_{\kappa} )||_{\hol(\set)}\underset{\delta_{\hat{n}}\rightarrow0}{\longrightarrow}0,\quad\text{for}\ \kappa\in\mathbb{I}\setminus\{\ell\},
\end{equation}
since \eqref{conv1} holds. Then, by \eqref{ineq7}--\eqref{ineq8}, it yields that for each $\kappa\in\mathbb{I}\setminus\{\ell\}$, there exists a $\delta^{(\kappa)}\in(0,1)$ such that if $\delta_{\hat{n}}\leq\delta^{(\kappa)}$, 
$u^{\varepsilon,\delta_{\hat{n}}}_{\ell}-u^{\varepsilon,\delta_{\hat{n}}}_{\kappa} -\vartheta_{\ell,\kappa}<0$ in $\mathcal{N}_{x^{*}}$. Taking $\delta'\eqdef\min_{\kappa\in\mathbb{I}\setminus\{\ell\}}\{\delta^{(\kappa)}\}$, it follows that $u^{\varepsilon,\delta_{\hat{n}}}_{\ell}-u^{\varepsilon,\delta_{\hat{n}}}_{\kappa} -\vartheta_{\ell,\kappa}<0$ in $\mathcal{N}_{x^{*}}$, for all $\delta_{\hat{n}}\leq\delta'$ and $\kappa\in\mathbb{I}\setminus\{\ell\}$. From here and since for each $\delta_{\hat{n}}\leq\delta'$, $u^{\varepsilon,\delta_{\hat{n}}}_{\ell}$ is the unique solution  to \eqref{NPD.1}, when $\delta=\delta_{\hat{n}}$, it implies that 
\begin{equation*}
\int_{\mathcal{N}_{x^{*}}}\Big\{[c_{\ell}-\dif_{\ell}]  u_{\ell}^{\varepsilon,\delta_{\hat{n}}}+ \psi_{\varepsilon}(|\deri^{1} u_{\ell}^{\varepsilon,\delta_{\hat{n}}}|^{2}- g_{\ell}^{2})\Big\}\varpi\der x=  \int_{\mathcal{N}_{x^{*}}}h_{\ell}\varpi\der x  , ,\quad\text{for}\ \varpi\in\mathcal{B}(\mathcal{N}_{x^{*}}).
\end{equation*}
Therefore,  \eqref{ineq6} holds. Hence, we get that for each $\varepsilon\in(0,1)$, $u^{\varepsilon}=(u^{\varepsilon}_{1},\dots,u^{\varepsilon}_{m})$ is a solution the HJB equation \eqref{p13.0.1.0}.
\end{proof}

\begin{proof}[Proof of Proposition \ref{princ1.1}. Uniqueness.] Let $\varepsilon\in(0,1)$ be fixed. Suppose that  $u^{\varepsilon}=(u^{\varepsilon}_{1},\dots,u^{\varepsilon}_{m})$ and $v^{\varepsilon}=(v^{\varepsilon}_{1},\dots,v^{\varepsilon}_{m})$ are two solutions to the HJB equation \eqref{pc1} whose components belong to  $\hol^{0,1}(\overline{\set})\cap\sob^{2,\infty}_{\loc}(\set)$. Take $(x_{0},\ell_{0})\in\overline{\set}\times\mathbb{I}$ such that
\begin{equation}\label{ineq9}
u^{\varepsilon}_{\ell_{0}}(x_{0})-v^{\varepsilon}_{\ell_{0}}(x_{0})=\sup_{(x,\ell)\in\overline{\set}\times\mathbb{I}}\{u^{\varepsilon}_{\ell}(x)-v_{\ell}^{\varepsilon}(x)\}.
\end{equation}
Notice that by \eqref{ineq9}, we only need to verify that 
\begin{equation}\label{ineq10}
u^{\varepsilon}_{\ell_{0}}(x_{0})-v^{\varepsilon}_{\ell_{0}}(x_{0})\leq0,
\end{equation}
which is trivially true,  if $x_{0}\in\partial\set$, since $u^{\varepsilon}_{\ell_{0}}-v^{\varepsilon}_{\ell_{0}}=0$ on $\partial\set$. Let us assume $x_{0}\in\set$. We shall verify \eqref{ineq10} by contradiction. Suppose that $u^{\varepsilon}_{\ell_{0}}-v^{\varepsilon}_{\ell_{0}}>0$ at $x_{0}$. Then, by continuity of $u^{\varepsilon}_{\ell_{0}}-v^{\varepsilon}_{\ell_{0}}$, there exists a  ball $B_{r_{1}}(x_{0})\subset\set$ such that 
\begin{equation}\label{ineq10.1}
c_{\ell_{0}}[u^{\varepsilon}_{\ell_{0}}-v^{\varepsilon}_{\ell_{0}}]\geq\min_{x\in B_{r_{1}}(x_{0})}\{c_{\ell_{0}}(x)[u^{\varepsilon}_{\ell_{0}}(x)-v^{\varepsilon}_{\ell_{0}}(x)]\}>0,\quad  \text{in}\ B_{r_{1}}(x_{0}).
\end{equation}
The last inequality is true because of $c_{\ell_{0}}>0$ in $\set$. Taking $\ell_{1}\in\mathbb{I}$ such that 
\begin{equation}\label{ineq10.0}
\mathcal{M}_{\ell_{0}}v^{\varepsilon}(x_{0})=v^{\varepsilon}_{\ell_{1}}(x_{0} )+\vartheta_{\ell_{0},\ell_{1}},
\end{equation}
by \eqref{pc1} and \eqref{ineq9}, we get that
$v^{\varepsilon}_{\ell_{0}}-(v^{\varepsilon}_{\ell_{1}}+\vartheta_{\ell_{0},\ell_{1}})
=v^{\varepsilon}_{\ell_{0}}-\mathcal{M}_{\ell_{0}}v^{\varepsilon}\leq u^{\varepsilon}_{\ell_{0}}-\mathcal{M}_{\ell_{0}}u^{\varepsilon}\leq 0$ at $x_{0}$.
If $v^{\varepsilon}_{\ell_{0}}(x_{0})-\mathcal{M}_{\ell_{0}}v^{\varepsilon}(x_{0})<0$, there exists a ball $B_{r_{2}}(x_{0})\subset\set$ such that $v^{\varepsilon}_{\ell_{0}}-\mathcal{M}_{\ell_{0}}v^{\varepsilon}<0$ in $B_{r_{2}}(x_{0})$. Moreover, from \eqref{pc1},
\begin{equation}\label{ineq11}
\begin{split}
[c_{\ell_{0}}-\dif_{\ell_{0}}]  v_{\ell_{0}}^{\varepsilon}+\psi_{\varepsilon}(|\deri^{1} v_{\ell_{0}}^{\varepsilon}|^{2}- g_{\ell_{0}}^{2})-h_{\ell_{0}}&=0,\\
[c_{\ell_{0}}-\dif_{\ell_{0}}]  u_{\ell_{0}}^{\varepsilon}+\psi_{\varepsilon}(|\deri^{1} u_{\ell_{0}}^{\varepsilon}|^{2}- g_{\ell_{0}}^{2})-h_{\ell_{0}}&\leq0,
\end{split}
\quad\text{in}\ B_{r_{2}}(x_{0}).
\end{equation}
Notice that $\psi_{\varepsilon}(|\deri^{1} u_{\ell_{0}}^{\varepsilon}|^{2}- g_{\ell_{0}}^{2})-\psi_{\varepsilon}(|\deri^{1} v_{\ell_{0}}^{\varepsilon}|^{2}- g_{\ell_{0}}^{2})$ is a continuous function in $\set$  due to $\partial_{i}u^{\varepsilon}_{\ell_{0}},\partial_{i}v^{\varepsilon}_{\ell_{0}}\in\hol^{0}(\set)$, which satisfies 
$\psi_{\varepsilon}(|\deri^{1} u_{\ell_{0}}^{\varepsilon}|^{2}- g_{\ell_{0}}^{2})-\psi_{\varepsilon}(|\deri^{1} v_{\ell_{0}}^{\varepsilon}|^{2}- g_{\ell_{0}}^{2})=0$ at $x_{0}$, since $x_{0}$ is the point where $u^{\varepsilon}_{\ell_{0}}-v^{\varepsilon}_{\ell_{0}}$ attains its maximum. Meanwhile, by Bony's maximum principle (see \cite{lions}), it is known that for every $r\leq r_{3}$, with $r_{3}>0$ small enough,
\begin{equation}\label{ineq11.0}
\tr[a_{\ell_{0}}\deri^{2}[u^{\varepsilon}_{\ell_{0}}-v^{\varepsilon}_{\ell_{0}}]]\leq0, \quad\text{a.e. in}\ B_{r}(x_{0}). 
\end{equation}
So, from \eqref{ineq10.1}, \eqref{ineq11} and \eqref{ineq11.0}, it yields that for every $r\leq \hat{r}\eqdef\min\{r_{1},r_{2},r_{3}\}$,
\begin{align*}
0&\geq \tr[a_{\ell_{0}}\deri^{2}[u^{\varepsilon}_{\ell_{0}}-v^{\varepsilon}_{\ell_{0}}]]\notag\\
&\geq  c_{\ell_{0}}[u^{\varepsilon}_{\ell_{0}}-v^{\varepsilon}_{\ell_{0}}]+\langle b_{\ell_{0}},\deri^{1}[u^{\varepsilon}_{\ell_{0}}-v^{\varepsilon}_{\ell_{0}}]\rangle+\psi_{\varepsilon}(|\deri^{1} u_{\ell_{0}}^{\varepsilon}|^{2}- g_{\ell_{0}}^{2})-\psi_{\varepsilon}(|\deri^{1} v_{\ell_{0}}^{\varepsilon}|^{2}- g_{\ell_{0}}^{2})\notag\\
&\geq \min_{x\in B_{r_{1}}(x_{0})}\{c_{\ell_{0}}(x)[u^{\varepsilon}_{\ell_{0}}(x)-v^{\varepsilon}_{\ell_{0}}(x)]\} +\langle b_{\ell_{0}},\deri^{1}[u^{\varepsilon}_{\ell_{0}}-v^{\varepsilon}_{\ell_{0}}]\rangle\notag\\
&\quad+\psi_{\varepsilon}(|\deri^{1} u_{\ell_{0}}^{\varepsilon}|^{2}- g_{\ell_{0}}^{2})-\psi_{\varepsilon}(|\deri^{1} v_{\ell_{0}}^{\varepsilon}|^{2}- g_{\ell_{0}}^{2}),\quad\text{a.e. in $B_{r}(x_{0})$.}
\end{align*}
 Then, 
\begin{multline}
\lim_{r\rightarrow0}\bigg\{\infess_{B_{r}(x_{0})}[\psi_{\varepsilon}(|\deri^{1} u_{\ell_{0}}^{\varepsilon}|^{2}- g_{\ell_{0}}^{2})-\psi_{\varepsilon}(|\deri^{1} v_{\ell_{0}}^{\varepsilon}|^{2}- g_{\ell_{0}}^{2})]\bigg\}\\
<-\min_{x\in B_{r_{1}}(x_{0})}\{c_{\ell_{0}}(x)[u^{\varepsilon}_{\ell_{0}}(x)-v^{\varepsilon}_{\ell_{0}}(x)]\}<0.
\end{multline}
That means $\psi_{\varepsilon}(|\deri^{1} u_{\ell_{0}}^{\varepsilon}|^{2}- g_{\ell_{0}}^{2})-\psi_{\varepsilon}(|\deri^{1} v_{\ell_{0}}^{\varepsilon}|^{2}- g_{\ell_{0}}^{2})$ is not continuous at $x_{0}$ which is a contradiction. Thus, 
\begin{equation}\label{ineq12.01}
0=v^{\varepsilon}_{\ell_{0}}-(v^{\varepsilon}_{\ell_{1}}+\vartheta_{\ell_{0},\ell_{1}})
=v^{\varepsilon}_{\ell_{0}}-\mathcal{M}_{\ell_{0}}v^{\varepsilon}\leq u^{\varepsilon}_{\ell_{0}}-\mathcal{M}_{\ell_{0}}u^{\varepsilon}\leq 0\quad\text{at}\ x_{0}.
\end{equation}
It implies that 
\begin{align}
&u^{\varepsilon}_{\ell_{1}}({x_{0} })-v^{\varepsilon}_{\ell_{1}}({x_{0} })\geq u^{\varepsilon}_{\ell_{0}}(x_{0})-v^{\varepsilon}_{\ell_{0}}(x_{0})>0,\label{ineq13}\\
&v^{\varepsilon}_{\ell_{0}}(x_{0})=v^{\varepsilon}_{\ell_{1}}({x_{0} })+\vartheta_{\ell_{0},\ell_{1}}.\notag%\label{ineq14}
\end{align}
By \eqref{ineq9} and \eqref{ineq13}, we have that $u^{\varepsilon}_{\ell_{1}}-v^{\varepsilon}_{\ell_{1}}$ attains its maximum  at $ x_{0} \in\set$, whose value agrees with $u^{\varepsilon}_{\ell_{0}}(x_{0})-v^{\varepsilon}_{\ell_{0}}(x_{0})$.  Then, {replacing} $u^{\varepsilon}_{\ell_{0}}-v^{\varepsilon}_{\ell_{0}}$ by $u^{\varepsilon}_{\ell_{1}}-v^{\varepsilon}_{\ell_{1}}$ above and   repeating the same arguments seen in  \eqref{ineq10.0}--\eqref{ineq12.01}, we get that there is a regime $\ell_{2}\in\mathbb{I}$ such that 
\begin{align*}
&u^{\varepsilon}_{\ell_{2}}(x_{0})-v^{\varepsilon}_{\ell_{2}}(x_{0})=u^{\varepsilon}_{\ell_{1}}(x_{0})-v^{\varepsilon}_{\ell_{1}}(x_{0})=u^{\varepsilon}_{\ell_{0}}(x_{0})-v^{\varepsilon}_{\ell_{0}}(x_{0})>0,\notag\\
&v^{\varepsilon}_{\ell_{1}}(x_{0})=v^{\varepsilon}_{\ell_{2}}(x_{0})+\vartheta_{\ell_{1},\ell_{2}}.%\label{ineq15}
\end{align*}
Recursively, we obtain a sequence of regimes $\{\ell_{i}\}_{i\geq0}$ such that
\begin{align}
&u^{\varepsilon}_{\ell_{i}}(x_{0})-v^{\varepsilon}_{\ell_{i}}(x_{0})=u^{\varepsilon}_{\ell_{i-1}}(x_{0})-v^{\varepsilon}_{\ell_{i-1}}(x_{0})=\cdots=u^{\varepsilon}_{\ell_{0}}(x_{0})-v^{\varepsilon}_{\ell_{0}}(x_{0})>0,\notag\\
&v^{\varepsilon}_{\ell_{i}}(x_{0})=v^{\varepsilon}_{\ell_{i+1}}(x_{0})+\vartheta_{\ell_{i},\ell_{i+1}}.\label{ineq16}
\end{align}
Since  $\mathbb{I}$ is finite,  {there is} a regime $\ell'$  that will appear infinitely often in  $\{\ell_{i}\}_{i\geq0}$. Let  $\ell_{n}=\ell'$, for some $n>1$. After $\hat{n}$ steps, the regime $\ell'$ reappears, i.e. $\ell_{n+\hat{n}}=\ell'$. Then, by \eqref{ineq16}, we get
\begin{equation}\label{ineq17}
v^{\varepsilon}_{\ell'}(x_{0})=v^{\varepsilon}_{\ell'}(x_{0})+\vartheta_{\ell',\ell_{n+1}}+\vartheta_{\ell_{n+1},\ell_{n+2}}+\cdots+\vartheta_{\ell_{n+\hat{n}-1},\ell'}.
\end{equation} 
Notice that \eqref{ineq17} contradicts the assumption that there is no loop of zero cost (see Eq. \eqref{l1}). From here we conclude that \eqref{ineq10} must occur. Taking $v^{\varepsilon}-u^{\varepsilon}$ and proceeding in the same way as before, it follows that for each $\ell\in\mathbb{I}$, $v^{\varepsilon}_{\ell}-u^{\varepsilon}_{\ell}\leq 0$  in $ \set $, and hence  we conclude that the solution $u^{\varepsilon}$ to the HJB equation \eqref{pc1} is unique.
\end{proof}

\subsection{Proof of Theorem \ref{M1}}\label{pro2}
In view of  Lemma \ref{Lb1} and by \eqref{conv1},  the following inequalities hold  for each $\ell\in\mathbb{I}$,
\begin{equation}\label{ineqe1}
0\leq u^{\varepsilon}_{\ell}\leq C_{1}\quad\text{and}\quad |\deri^{1}u^{\varepsilon}_{\ell}|\leq C_{4},\quad\text{in}\ \overline{\set},
\end{equation}	
for some positive constant $C_{4}=C_{4}(d,\Lambda,\alpha')$. The constant $C_{1}$ is as in \eqref{ap1}. Besides, for each $B_{\beta r}\subset\set$, there exists a positive constant $C_{5}=C_{5}(d,\Lambda,\alpha')$ such that
\begin{equation}\label{ineqe2}
||\deri^{2}u^{\varepsilon}_{\ell}||_{\Lp^{p}(B_{\beta r})}\leq C_{5},\quad \text{for each}\ p\in(1,\infty). 
\end{equation}
Then, from \eqref{ineqe1}--\eqref{ineqe2} and using again Arzel\`a-Ascoli compactness criterion and that $(\Lp^{p}(B_{\beta r}),||\cdot||_{\Lp^{p}(B_{\beta r})})$ is a reflexive space, we have that there exists a sub-sequence $\{u^{\varepsilon_{n}}_{\ell}\}_{n\geq1}$ of $\{u^{\varepsilon}_{\ell}\}_{\varepsilon\in(0,1)}$  and $u_{\ell}\in\hol^{0,1}(\overline{\set})\cap\sob^{2,\infty}_{\loc}(\set)$ such that
\begin{equation}
\begin{split} \label{econv1}
&\text{$u^{\varepsilon_{n}}_{\ell}\underset{\varepsilon_{n}\rightarrow0}{\longrightarrow}u_{\ell}$ in $\hol(\overline{\set})$, $\partial_{i}u^{\varepsilon_{n}}_{\ell}\underset{\varepsilon_{n}\rightarrow0}{\longrightarrow}\partial_{i} u_{\ell}$ $\text{in}\ \hol_{\loc}(\set)$,}\\ &\text{$\partial_{ij}u^{\varepsilon_{n}}_{\ell}\underset{\varepsilon_{n}\rightarrow0}{\longrightarrow}\partial_{ij}u_{\ell}$,  weakly $ \Lp^{p}_{\loc}(\set)$, for each $p\in(1,\infty)$.}
\end{split}
\end{equation}
\begin{rem}
The {notations} $C=C(*,\dots,*)$ and $K=K(*,\dots,*)$ represent positive constants that depend only on the quantities appearing in parenthesis.
\end{rem}
 We proceed to prove Theorem \ref{M1}.

\begin{proof}[Proof of Theorem \ref{M1}. Existence.]

	Now, let $\ell\in\mathbb{I}$ be fixed. Since $u^{\varepsilon_{n}}_{\ell}$ is the unique  {strong} solution to the HJB equation \eqref{pc1} when $\varepsilon=\varepsilon_{n}$,  {which belongs to $\hol^{0,1}(\overline{\set})$},  it follows that  for each $\kappa\in\mathbb{I}\setminus\{\ell\}$,  $u^{\varepsilon_{n}}_{\ell}-(u^{\varepsilon_{n}}_{\kappa} +\vartheta_{\ell,\kappa})\leq u^{\varepsilon_{n}}_{\ell}-\mathcal{M}_{\ell}u^{\varepsilon_{n}}\leq0$ in $\set$. From here and \eqref{econv1}, it yields that  $u_{\ell}-u_{\kappa} -\vartheta_{\ell,\kappa}\leq0$ in $\set$. Then, $u_{\ell}-\mathcal{M}_{\ell}u\leq0$, in $\set$. Also, we know that $[c_{\ell}-\dif_{\ell}]  u_{\ell}^{\varepsilon_{n}}+ \psi_{\varepsilon_{n}}(|\deri^{1} u_{\ell}^{\varepsilon_{n}}|^{2}- g_{\ell}^{2})\leq h_{\ell}$ a.e. in $\set$. Then,
 \begin{align}
 &0\leq\psi_{\varepsilon_{n}}(|\deri^{1} u_{\ell}^{\varepsilon_{n}}|^{2}- g_{\ell}^{2})\leq h_{\ell}-[c_{\ell}-\dif_{\ell}]  u_{\ell}^{\varepsilon_{n}},\quad \text{a.e.}\ \text{in}\  \set,\label{eneq2.1}%\\
 %&[c_{\ell}-\dif_{\ell}]  u_{\ell}^{\varepsilon_{n}}-h_{\ell}\leq 0,\quad \text{a.e.}\ \text{in}\  \set.\label{eneq2.2} 
 \end{align}
 Consequently, by \eqref{a4}, \eqref{ineqe1}, \eqref{ineqe2} and \eqref{eneq2.1}, there exists a positive constant $C_{6}=C_{6}(d,\Lambda,\alpha')$ such that $0\leq\int_{B_{r}}\psi_{\varepsilon_{n}}(|\deri^{1} u_{\ell}^{\varepsilon_{n}}|^{2}- g_{\ell}^{2})\varpi\der x\leq\int_{B_{r}} \{h_{\ell}-[c_{\ell}-\dif_{\ell}]  u_{\ell}^{\varepsilon_{n}}\}\varpi\der x\leq C_{6}$ for each  $\varpi\in\mathcal{B}(B_{r})$, with $\mathcal{B}(\cdot)$ as in \eqref{c1}. Thus, using definition of $\psi_{\varepsilon}$ (see  \eqref{p12.1})  {and since $|\deri^{1}u_{\ell}^{\varepsilon_{n}}|^{2}-g^{2}_{\ell}$ is continuous in $\set$}, we have that for each $B_{r}\subset \set$, there exists $\varepsilon'\in(0,1)$ small enough, such that for all $\varepsilon_{n}\leq \varepsilon'$,  $|\deri^{1}u^{\varepsilon_{n}}_{\ell}|-g_{\ell}\leq0$ in $B_{r}$. Then, since \eqref{econv1} holds, it follows that $|\deri^{1} u_{\ell}|\leq g_{\ell}$ in $\set$. From \eqref{eneq2.1}, we get $\int_{B_{r}}\big\{[c_{\ell}-\dif_{\ell}]  u_{\ell}^{\varepsilon_{n}}-h_{\ell}\big\}\varpi\der x\leq 0$, for each $\varpi\in \mathcal{B}(B_{r})$. From here and \eqref{econv1}, we obtain that $[c_{\ell}-\dif_{\ell}]  u_{\ell}-h_{\ell}\leq 0$ a.e. in $\set$. Therefore, by the seen previously, 
	\begin{equation}\label{eineq4}
	\max\left\{[c_{\ell}-\dif_{\ell}]  u_{\ell}-h_{\ell},|\deri^{1} u_{\ell}|- g_{\ell},u_{\ell}-\mathcal{M}_{\ell}u\right\}\leq 0 ,\quad \text{a.e. in}\ \set.
	\end{equation}
	Without loss of generality we assume that $
	u_{\ell}(x^{*})-\mathcal{M}_{\ell}u(x^{*})<0$, for some $x^{*}\in\set$. Otherwise, the equality is satisfied in \eqref{eineq4}. Then, for each $\kappa\in\mathbb{I}$ such that $\kappa\neq\ell$,  $u_{\ell}-(u_{\kappa} +\vartheta_{\ell,\kappa})\leq u_{\ell}-\mathcal{M}_{\ell}u<0$ at $x^{*}$.  There exists a ball $B_{r_{1}}(x^{*})\subset\set$ such that 
	\begin{equation}\label{eineq5}
	u_{\ell}-(u_{\kappa} +\vartheta_{\ell,\kappa})\leq u_{\ell}-\mathcal{M}_{\ell}u<0,\quad\text{in}\ B_{r_{1}}(x^{*})
	\end{equation}
	due to the continuity of $u_{\ell}-u_{\kappa}$ in $\overline{\set}$. Now, consider that $|\deri^{1}u_{\ell}|-g_{\ell}<0$ for some $x^{*}_{1}\in B_{r_{1}}(x^{*})$. Otherwise, the equality is also satisfied in \eqref{eineq4}. By continuity of $|\deri^{1}u_{\ell}|-g_{\ell}$, it yields that for some $B_{r_{2}}(x^{*}_{1})\subset\set$, $|\deri^{1}u_{\ell}|-g_{\ell}<0$ in $B_{r_{2}}(x^{*}_{1})$. From here, using \eqref{econv1}, \eqref{eineq5} and  taking $\mathcal{N}\eqdef B_{r_{1}}(x^{*})\cap B_{r_{2}}(x^{*}_{1})$,  it can be verified that there exists an $\varepsilon'\in(0,1)$ small enough, such that for each $\varepsilon_{n}\leq\varepsilon'$, $|\deri^{1}u^{\varepsilon_{n}}_{\ell}|-g_{\ell}<0$ and $ u^{\varepsilon_{n}}_{\ell}-\mathcal{M}_{\ell}u^{\varepsilon_{n}}<0$ in $\mathcal{N}$. Thus, $[c_{\ell}-\dif_{\ell}]  u_{\ell}^{\varepsilon_{n}}= h_{\ell}$ a.e. in $\mathcal{N}$, since $u^{\varepsilon_{n}}$ is the unique solution to the HJB equation \eqref{pc1}, when $\varepsilon=\varepsilon_{n}$. Then, $\int_{\mathcal{N}}\big\{[c_{\ell}-\dif_{\ell}]  u_{\ell}^{\varepsilon_{n}}-h_{\ell}\big\}\varpi\der x= 0$, for each $\varpi\in\mathcal{B}(\mathcal{N})$.  Hence, letting $\varepsilon_{n}\rightarrow0$ and using again \eqref{econv1}, we get that  $u=(u_{1},\dots,u_{m})$ is a solution to the HJB equation  \eqref{esd5}.
\end{proof}

\begin{proof}[Proof of Theorem \ref{M1}. Uniqueness.] Suppose that  $u=(u_{1},\dots,u_{m})$ and $v=(v_{1},\dots,v_{m})$ are two solutions to the HJB equation  \eqref{esd5} whose components belong to  $\hol^{0,1}(\overline{\set})\cap\sob^{2,\infty}_{\loc}(\set)$. Take $(x_{0},\ell_{0})\in\overline{\set}\times\mathbb{I}$ such that
	\begin{equation}\label{eineq9}
	u_{\ell_{0}}(x_{0})-v_{\ell_{0}}(x_{0})=\sup_{(x,\ell)\in\overline{\set}\times\mathbb{I}}\{u_{\ell}(x)-v_{\ell}(x)\}.
	\end{equation}
	As before (see Subsection \ref{prop1}), we only need to verify that 
	\begin{equation}\label{eineq10}
	u_{\ell_{0}}-v_{\ell_{0}}\leq0,\quad \text{at}\  x_{0}\in\set,
	\end{equation}
  Assume that $u_{\ell_{0}}-v_{\ell_{0}}>0$ at $x_{0}$. Then, there exists a  ball $B_{r_{1}}(x_{0})\subset\set$ such that $c_{\ell_{0}}[u_{\ell_{0}}-v_{\ell_{0}}]\geq\min_{x\in B_{r_{1}}(x_{0})}\{c_{\ell_{0}}(x)[u_{\ell_{0}}(x)-v_{\ell_{0}}(x)]\}>0$ in $B_{r_{1}}(x_{0})$ due to the continuity of  $u_{\ell_{0}}-v_{\ell_{0}}$ in $\overline{\set}$ and that $c_{\ell_{0}}>0$ in $\set$. Meanwhile, from \eqref{eineq9}, $v_{\ell_{0}}-\mathcal{M}_{\ell_{0}}v\leq u_{\ell_{0}}-\mathcal{M}_{\ell_{0}}u\leq 0$ at $x_{0}$. If $v_{\ell_{0}}-\mathcal{M}_{\ell_{0}}v<0$ at $x_{0}$, there exists a ball $B_{r_{2}}(x_{0})\subset\set$ such that $v_{\ell_{0}}-\mathcal{M}_{\ell_{0}}v<0$ in $B_{r_{2}}(x_{0})$. Now, consider the auxiliary function $f_{\varrho}\eqdef u_{\ell_{0}}-v_{\ell_{0}}-\varrho u_{\ell_{0}}$, with $\varrho\in(0,1)$. Notice that $f_{\varrho}=0$ on $\partial\set$, for $\varrho\in(0,1)$, and  
	\begin{equation}\label{eineq11}
	f_{\varrho}\uparrow  u_{\ell_{0}}-v_{\ell_{0}}\ \text{uniformly in}\ \set,\ \text{when}\  \varrho\downarrow 0. 
	\end{equation}
Besides, there is a $\varrho'\in(0,1)$ small enough such that $\sup_{x\in  B_{r_{2}}(x_{0})}\{f_{\varrho}(x)\}>0$ for all $\varrho\in(0,\varrho')$ because of $u_{\ell_{0}}-v_{\ell_{0}}>0$ at $x_{0}$. By \eqref{eineq9} and \eqref{eineq11}, there exists $\hat{\varrho}\in(0,\varrho')$ small enough such that $f_{\hat{\varrho}}$ has a local maximum  at  $x_{\hat{\varrho}}\in B_{r_{1}}(x_{0})\cap B_{r_{2}}(x_{0})$. It follows that $|\deri^{1}v_{\ell_{0}}(x_{\hat{\varrho}})|=[1-\hat{\rho}]|\deri^{1}u_{\ell_{0}}(x_{\hat{\varrho}})|<|\deri^{1}u_{\ell_{0}}(x_{\hat{\varrho}})|\leq g(x_{\hat{\varrho}}).$ Thus, there exists a ball $B_{r_{3}}(x_{\hat{\varrho}})\subset B_{r_{1}}(x_{0})\cap B_{r_{2}}(x_{0})$ such that $[c_{\ell_{0}}-\dif_{\ell_{0}}]  v_{\ell_{0}}-h_{\ell_{0}}=0$ and  $[c_{\ell_{0}}-\dif_{\ell_{0}}]  u_{\ell_{0}}-h_{\ell_{0}}\leq0$ in $B_{r_{3}}(x_{\hat{{\varrho}}})$.  Then, by Bony's maximum principle, we have that $0\geq\lim_{r\rightarrow0}\big\{\infess_{B_{r}(x_{\hat{\varrho}})} \tr[a_{\ell_{0}}\deri^{2}f_{\hat{\varrho}}]\big\}\geq c_{\ell_{0}}f_{\hat{\varrho}}+\hat{\varrho} h_{\ell_{0}}$ at $x_{\hat{\varrho}}$,  which is a contradiction because of $\hat{\varrho} h_{\ell}\geq0$, $f_{\hat{\varrho}}>0$ and $c_{\ell_{0}}>0$ at $x_{\hat{\varrho}}$. We conclude that, $0=v_{\ell_{0}}-\mathcal{M}_{\ell_{0}}v\leq u_{\ell_{0}}-\mathcal{M}_{\ell_{0}}u\leq 0$ at $x_{0}$. Using the same arguments seen in  the proof of uniqueness of the solution to the  HJB equation \eqref{esd5}  (see Subsection \ref{prop1}), it can be verified that there is a contradiction with the assumption that there is no loop of  zero cost  (see Eq. \eqref{l1}). From here we conclude that \eqref{eineq10} must occur. Taking $v-u$ and proceeding in the same way as  before, we see $u$ is the unique solution to the HJB equation \eqref{esd5}.
\end{proof}	  

\section{Probabilistic part}\label{Pp1}
 In this section, under assumptions \eqref{h3} and {\eqref{a4}--\eqref{h5}}, we shall proceed to verify that the value functions $V$  given in \eqref{vf1} agrees with the solution $u$ to the HJB equation \eqref{esd5}  {on $\overline{\set}$}. For that aim, let us start showing that the value function $V^{\varepsilon}$ given in \eqref{Vfp1}, with $\varepsilon\in(0,1)$ fixed, agrees with the solution $u^{\varepsilon}$ to \eqref{pc1}. The proof of this result is presented in two parts; see Lemmas \ref{lv1} and \ref{optim1}.   

Let us consider $ (X^{\xi,\varsigma},I^{\varsigma})$ evolves as \eqref{esd3.1.0.0}--\eqref{esd3.1.0.1}, with $(\xi,\varsigma)\in\mathcal{U}^{\varepsilon}\times\mathcal{S}$ and initial state $(\tilde{x},\tilde{\ell})\in {\set}\times\mathbb{I}$. Recall that $\mathcal{U}^{\varepsilon}$ is defined in \eqref{p1}  and $\mathcal{S}$ is the set of elements $\varsigma=(\tau_{i},\ell_{i})_{i\geq0}$ that satisfy \eqref{cont.2}. The functional cost $\mathcal{V}_{\xi,\varsigma}$  is as in \eqref{pen1}. 

Under assumptions \eqref{h3}--\eqref{h2}, Lemmas \ref{lv1} and \ref{convexu1.0} shall be proven.
\begin{lema}[Verification Lemma for penalized absolutely continuous/switching control problem. First part]\label{lv1} Let $\varepsilon\in(0,1)$ be fixed. Then, $u^{\varepsilon}_{\tilde{\ell}}(\tilde{x})\leq\mathcal{V}_{\zeta,\varsigma}(\tilde{x},\tilde{\ell})$ for each $(\tilde{x},\tilde{\ell})\in\overline{\set}\times\mathbb{I}$  and $(\xi,\varsigma)\in\mathcal{U}^{\varepsilon}\times\mathcal{S}$.
\end{lema}
 From now on, for simplicity of notation, we replace $X^{\xi,\varsigma}$ by $X$ in the  proofs of the results.
\begin{proof}[Proof of Lemma \ref{lv1}]
 Let $\{u^{\varepsilon,\delta_{\hat{n}}}\}_{\hat{n}\geq1}$ be the sequence of unique solutions to the NPDS \eqref{NPD.1}, when $\delta=\delta_{\hat{n}}$, which  satisfy \eqref{conv1}. By Proposition \ref{princ1.0}, it is known that $u^{\varepsilon,\delta_{\hat{n}}}_{\ell}\in\hol^{4,\alpha'}(\overline{\set})$ for $\ell\in\mathbb{I}$. Then, using integration by parts and It\^o's formula  (see \cite[Cor. 2 and Thm. 33, pp. 68 and 81, respectively]{pro}) in $\expo^{-r(\tmt)}u^{\varepsilon,\delta_{\hat{n}}}_{\ell_{i}}(X_{\tmt})$ on $ [\tilde{\tau}_{i},\tilde{\tau}_{i+1})$, $i\geq0$, where $\tilde{\tau}_{i}=\tau_{i}\wedge\tau$, $\tau=\inf\{t\geq0: X_{\tmt}\in\set\}$ and  $r(\tmt)$ is as in \eqref{eq0.1}, and taking expected value on it,  we have that 
\begin{align}\label{expand.1.0}
&\E_{\tilde{x},\tilde{\ell}}\Big[\expo^{-r(\tilde{\tau}_{i})}u^{\varepsilon,\delta_{\hat{n}}}_{\ell_{i}}(X_{\tilde{\tau}_{i}})\Big]=\E_{\tilde{x},\tilde{\ell}}\bigg[[u^{\varepsilon,\delta_{\hat{n}}}_{\ell_{1}}(X_{\tau_{1}})+\vartheta_{\ell_{0},\ell_{1}}]\uno_{\{\tau_{1}=0,\tau_{1}<\tau,i=0\}}\notag\\ 
&\quad+[u^{\varepsilon,\delta_{\hat{n}}}_{\ell_{0}} (X_{\tau_{1}})-u^{\varepsilon,\delta_{\hat{n}}}_{\ell_{1}}(X_{\tau_{1}} )-\vartheta_{\ell_{0},\ell_{1}}]\uno_{\{\tau_{1}=0,\tau_{1}<\tau,i=0\}}\notag\\
&\quad+\bigg[\expo^{-r(\tmt\wedge\tilde{\tau}_{1})}u^{\varepsilon,\delta_{\hat{n}}}_{\ell_{0}}(X_{\tmt\wedge\tilde{\tau}_{1}})-\widetilde{\mathcal{M}}_{\ell_{0}}[0,\tmt\wedge\tilde{\tau}_{1};X,u^{\varepsilon,\delta_{\hat{n}}}]\notag\\
&\quad+\int_{0}^{\tmt\wedge\tilde{\tau}_{1}}\expo^{-r(\tms)}[c_{\ell_{0}}(X_{\tms})u^{\varepsilon,\delta_{\hat{n}}}_{\ell_{0}}(X_{\tms})-\dif_{\ell_{0}}u^{\varepsilon,\delta_{\hat{n}}}_{\ell_{0}}(X_{\tms})+\langle\deri^{1} u^{\varepsilon,\delta_{\hat{n}}}_{\ell_{0}}(X_{\tms}),\dot{\zeta}_{\tms}\mathbb{n}_{\tms}\rangle]\der \tms\bigg]\uno_{\{\tau_{1}\neq0,i=0\}}\notag\\
&\quad+\bigg[\expo^{-r(\tmt\wedge\tilde{\tau}_{i+1})}u^{\varepsilon,\delta_{\hat{n}}}_{\ell_{i}}(X_{\tmt\wedge\tilde{\tau}_{i+1}})-\widetilde{\mathcal{M}}_{\ell_{i}}[\tilde{\tau}_{i},\tmt\wedge\tilde{\tau}_{i+1};X,u^{\varepsilon,\delta_{\hat{n}}}]\notag\\
&\quad+\int_{\tilde{\tau}_{i}}^{\tmt\wedge\tilde{\tau}_{i+1}}\expo^{-r(\tms)}[c_{\ell_{i}}(X_{\tms})u^{\varepsilon,\delta_{\hat{n}}}_{\ell_{i}}(X_{\tms})-\dif_{\ell_{i}}u^{\varepsilon,\delta_{\hat{n}}}_{\ell_{i}}(X_{\tms})+\langle\deri^{1} u^{\varepsilon,\delta_{\hat{n}}}_{\ell_{i}}(X_{\tms}),\dot{\zeta}_{\tms}\mathbb{n}_{\tms}\rangle]\der \tms\bigg]\uno_{\{i\neq0\}}\bigg],
\end{align}
where
\begin{align}\label{mar1}
\widetilde{\mathcal{M}}_{\ell_{i}}[\tilde{\tau}_{i},\tmt\wedge\tilde{\tau}_{i+1};X,u^{\varepsilon,\delta_{\hat{n}}}]\eqdef\int_{\tilde{\tau}_{i}}^{\tmt\wedge\tilde{\tau}_{i+1}}\expo^{-r(\tms)}\langle\deri^{1} u^{\varepsilon,\delta_{\hat{n}}}_{\ell_{i}}(X_{\tms}),\sigma_{\ell_{i}}(X_{\tms})\rangle\der  W_{\tms}.
\end{align}
Since $\widetilde{\mathcal{M}}_{\ell_{i}}[\tilde{\tau}_{i},\tmt\wedge\tilde{\tau}_{i+1};X,u^{\varepsilon,\delta_{\hat{n}}}]$ is a square integrable martingale, $\langle\gamma,y\rangle\leq\psi_{\varepsilon}(|\gamma|^{2}-g_{\ell}   (x)^{2})+l^{\varepsilon}_{\ell}(y,x)$ and $[c_{\ell}-\dif_{\ell}]u^{\varepsilon,\delta_{\hat{n}}}_{\ell}+\psi_{\varepsilon}(|\deri^{1}u^{\varepsilon,\delta_{\hat{n}}}_{\ell}|^{2}-g^{2}_{\ell})\leq h_{\ell}$, it follows that  
\begin{align}\label{ex1}
&\E_{\tilde{x},\tilde{\ell}}\Big[\expo^{-r(\tilde{\tau}_{i})}u^{\varepsilon,\delta_{\hat{n}}}_{\ell_{i}}(X_{\tilde{\tau}_{i}})\Big]\notag\\
&\leq\E_{\tilde{x},\tilde{\ell}}\bigg[[u^{\varepsilon,\delta_{\hat{n}}}_{\ell_{1}}(X_{\tau_{1}})+\vartheta_{\ell_{0},\ell_{1}}]\uno_{\{\tau_{1}=0,\tau_{1}<\tau,i=0\}}+\mathcal{D}[{\tau}_{1},\ell_{0},\ell_{1};X, u^{\varepsilon,\delta_{\hat{n}}}]\uno_{\{\tau_{1}=0,\tau_{1}<\tau,i=0\}}\notag\\
&\quad+\expo^{-r(\tmt\wedge\tilde{\tau}_{1})}[u_{\ell_{1}}^{\varepsilon,\delta_{\hat{n}}}(X_{\tmt\wedge\tilde{\tau}_{1}})+\vartheta_{\ell_{0},\ell_{1}}]\uno_{\{\tau_{1}\neq0,i=0\}}+\expo^{-r(\tmt\wedge\tilde{\tau}_{i+1})}[u_{\ell_{i+1}}^{\varepsilon,\delta_{\hat{n}}}(X_{\tmt\wedge\tilde{\tau}_{i+1}})+\vartheta_{\ell_{i},\ell_{i+1}}]\uno_{\{i\neq0\}}\notag\\
&\quad+\mathcal{D}[\tmt\wedge\tilde{\tau}_{1},\ell_{0},\ell_{1};X, u^{\varepsilon,\delta_{\hat{n}}}]\uno_{\{\tau_{1}\neq0,i=0\}}+\mathcal{D}[\tmt\wedge\tilde{\tau}_{i+1},\ell_{i},\ell_{i+1};X, u^{\varepsilon,\delta_{\hat{n}}}]\uno_{\{i\neq0\}}\notag\\
&\quad+\uno_{\{\tau_{1}\neq0,i=0\}}\int_{0}^{\tmt\wedge\tilde{\tau}_{1}}\expo^{-r(\tms)}[h_{\ell_{0}}(X_{\tms})+l^{\varepsilon}_{\ell_{0}}(\dot{\zeta}_{\tms}\mathbb{n}_{\tms},X_{\tms})]\der \tms\notag\\
&\quad+\uno_{\{i\neq0\}}\int_{\tilde{\tau}_{i}}^{\tmt\wedge\tilde{\tau}_{i+1}}\expo^{-r(\tms)}[h_{\ell_{i}}(X_{\tms})+l^{\varepsilon}_{\ell_{i}}(\dot{\zeta}_{\tms}\mathbb{n}_{\tms},X_{\tms})]\der \tms\bigg],
\end{align}
where
\begin{equation}\label{ex5}
\mathcal{D}[\tmt\wedge\tilde{\tau}_{i+1},\ell_{i},\ell_{i+1};X, u^{\varepsilon,\delta_{\hat{n}}}]\eqdef \expo^{-r({\tmt\wedge\tilde{\tau}}_{i+1})}[u^{\varepsilon,\delta_{\hat{n}}}_{\ell_{i}}(X_{\tmt\wedge\tilde{\tau}_{i+1}})-[u^{\varepsilon,\delta_{\hat{n}}}_{\ell_{i+1}}(X_{\tmt\wedge\tilde{\tau}_{i+1}})+\vartheta_{\ell_{i},\ell_{i+1}}]\big].
\end{equation}
We have $\max\{\expo^{-r(\tmt\wedge\tilde{\tau}_{i})}u^{\varepsilon,\delta_{\hat{n}}}_{\ell_{i}}(X_{{\tmt\wedge\tilde{\tau}}_{i}}), 
|\mathcal{D}[\tmt\wedge\tilde{\tau}_{i+1},\ell_{i},\ell_{i+1},X, u^{\varepsilon,\delta_{\hat{n}}}]|\}\leq 2 C_{1}+\max_{(x,\ell)\in\set\times\mathbb{I}}c_{\ell}(x)$ due to Lemma \ref{Lb1}. Meanwhile, we know that $\max_{(x,\ell)\in\set\times\mathbb{I}}|u^{\varepsilon,\delta_{\hat{n}}}_{\ell}(x)-u^{\varepsilon}_{\ell}(x)|\underset{\delta_{\hat{n}}\rightarrow0}{\longrightarrow}0$. Then, letting first $\delta_{\hat{n}}{\rightarrow}0$ and after $\tmt\rightarrow\infty$ in \eqref{ex1} and by the dominated convergence theorem, it follows that 
\begin{multline}\label{ex7}
\E_{\tilde{x},\tilde{\ell}}\Big[\expo^{-r(\tilde{\tau}_{i})}u^{\varepsilon}_{\ell_{i}}(X_{{\tilde{\tau}}_{i}})\Big]\leq\E_{\tilde{x},\tilde{\ell}}\bigg[\{\expo^{-r(\tau_{i+1})}u_{\ell_{i+1}}^{\varepsilon}(X_{\tau_{i+1}})+\mathcal{D}[\tau_{i+1},\ell_{i},\ell_{i+1};X, u^{\varepsilon}]\}\uno_{\{\tau_{i+1}<\tau\}}\\
+\int_{{\tilde{\tau}}_{i}}^{{\tilde{\tau}}_{i+1}}\expo^{-r(\tms)}[h_{\ell_{i}}(X_{\tms})+l^{\varepsilon}_{\ell_{i}}(\dot{\zeta}_{\tms}\mathbb{n}_{\tms},X_{\tms})]\der \tms+\expo^{-r(\tau_{i+1})}\vartheta_{\ell_{i},\ell_{i+1}}\uno_{\{\tau_{i+1}<\tau\}}\bigg],\quad \text{for}\ i\geq0, 
\end{multline}
Observe that $\mathcal{D}[\tau_{i+1},\ell_{i},\ell_{i+1},X, u^{\varepsilon}]\uno_{\{\tau_{i+1}<\tau\}}\leq0$ for $i\geq0$ because of  $u^{\varepsilon}_{\ell}-(u^{\varepsilon}_{\kappa} +\vartheta_{\ell,\kappa})\leq u^{\varepsilon}_{\ell}-\mathcal{M}_{\ell}u^{\varepsilon}\leq0$. With this remark and \eqref{ex7}, we conclude the statement of the lemma above.
\end{proof}

\subsubsection*{$\varepsilon$-penalized absolutely continuous/switching optimal control problem}
Before presenting the second part of the verification lemma and proving  it, let us first  construct the control $(\xi^{\varepsilon,*},\varsigma^{\varepsilon,*})$ which turns out to  be the optimal strategy for the $\varepsilon$-penalized absolutely continuous/switching  control problem.  For that aim, let us introduce the switching regions. 

For any $\ell\in\mathbb{I}$, let $\mathcal{S}^{\varepsilon}_{\ell}$ be the set defined by 
$$\mathcal{S}^{\varepsilon}_{\ell}=\{x\in\set:u^{\varepsilon}_{\ell}(x)-\mathcal{M}_{\ell}u^{\varepsilon}(x)=0\}.$$ 
Notice that $\mathcal{S}^{\varepsilon}_{\ell}$ is a closed subset of $\set$ and corresponds with the region where it is optimal to switch regimes. The complement $\mathcal{C}^{\varepsilon}_{\ell}$ of $\mathcal{S}^{\varepsilon}_{\ell}$ in $\set$, where is optimal to stay in  {the} regime $\ell$, is the so-called continuation region
 $$\mathcal{C}^{\varepsilon}_{\ell}=\{x\in\set:u^{\varepsilon}_{\ell}(x)-\mathcal{M}_{\ell}u^{\varepsilon}(x)<0\}.$$ 
\begin{rem}\label{reg1}
	Observe that $u^{\varepsilon}_{\ell}\in\sob^{2,\infty}_{\loc}(\mathcal{C}^{\varepsilon}_{\ell})=\hol_{\loc}^{1,1}(\mathcal{C}^{\varepsilon}_{\ell})$. It implies that $h_{\ell}-\psi_{\varepsilon}(|\deri^{1}u^{\varepsilon}_{\ell}|^{2}-g_{\ell}^{2})\in\hol^{0,\alpha'}_{\loc}(\mathcal{C}_{\ell}^{\varepsilon})$ due to $h_{\ell}\in\hol^{2,\alpha'}(\overline{\set}))$ and $\psi_{\varepsilon}(|\deri^{1}u^{\varepsilon}_{\ell}|^{2}-g_{\ell}^{2})\in\hol^{0,\alpha'}_{\loc}(\mathcal{C}_{\ell}^{\varepsilon})$. From here and using Theorem 9.19 of \cite{gilb}, it yields that  $u^{\varepsilon}_{\ell}\in\hol_{\loc}^{2,\alpha'}(\mathcal{C}^{\varepsilon}_{\ell})$.
\end{rem}

\begin{rem}
	Notice that there are no isolated points in a switching region $\mathcal{S}_{\ell}^{\varepsilon}$.
\end{rem}

Also, the set $\mathcal{S}^{\varepsilon}_{\ell}$ satisfies the following property.
\begin{lema}\label{optim1}
Let $\ell$ be in $\mathbb{I}$. Then, $\mathcal{S}^{\varepsilon}_{\ell}=\widetilde{\mathcal{S}}^{\varepsilon}_{\ell}\eqdef\bigcup_{\kappa\in\mathbb{I}\setminus\{\ell \}}\mathcal{S}^{\varepsilon}_{\ell,\kappa}$ where $\mathcal{S}^{\varepsilon}_{\ell,\kappa}\eqdef\{x\in\mathcal{C}^{\varepsilon}_{\kappa}:u^{\varepsilon}_{\ell}(x)=u^{\varepsilon}_{\kappa}(x )+\vartheta_{\ell,\kappa} \}$. 
\end{lema}
\begin{proof}
We obtain trivially that  $\widetilde{\mathcal{S}}^{\varepsilon}_{\ell}\subset\mathcal{S}^{\varepsilon}_{\ell}$ due to $u^{\varepsilon}_{\ell}-u^{\varepsilon}_{\kappa}-\vartheta_{\ell,\kappa}\leq u^{\varepsilon}_{\ell}-\mathcal{M}_{\ell}u^{\varepsilon}\leq0$ on $\set$ for $\kappa\in\mathbb{I}\setminus\{\ell\}$. If $x\in\mathcal{S}^{\varepsilon}_{\ell}$, there is an $\ell_{1}\neq\ell$ where $u^{\varepsilon}_{\ell}(x)=u^{\varepsilon}_{\ell_{1}}(x)+\vartheta_{\ell,\ell_{1}}$. 
Notice that $x$ must belong either $\mathcal{C}^{\varepsilon}_{\ell_{1}}$ or $\mathcal{S}^{\varepsilon}_{\ell_{1}}$. If $x\in\mathcal{C}^{\varepsilon}_{\ell_{1}}$, it yields that  $x\in\mathcal{S}^{\varepsilon}_{\ell,\ell_{1}}\subset\widetilde{\mathcal{S}}^{\varepsilon}_{\ell} $. Otherwise, there is an $\ell_{2}\neq\ell_{1}$ such that $u^{\varepsilon}_{\ell_{1}}(x)=u^{\varepsilon}_{\ell_{2}}(x)+\vartheta_{\ell_{1},\ell_{2}}$. It implies $u^{\varepsilon}_{\ell}(x)=u^{\varepsilon}_{\ell_{2}}(x)+\vartheta_{\ell,\ell_{1}}+\vartheta_{\ell_{1},\ell_{2}}\geq u^{\varepsilon}_{\ell_{2}}(x)+\vartheta_{\ell,\ell_{2}}$, since \eqref{eq1} holds. Then, $u^{\varepsilon}_{\ell}(x)= u^{\varepsilon}_{\ell_{2}}(x)+\vartheta_{\ell,\ell_{2}}$. Again $x$ must belong either $\mathcal{C}^{\varepsilon}_{\ell_{2}}$ or $\mathcal{S}^{\varepsilon}_{\ell_{2}}$. If $x\in\mathcal{C}^{\varepsilon}_{\ell_{2}}$, it yields that  $x\in\mathcal{S}^{\varepsilon}_{\ell,\ell_{2}}\subset\widetilde{\mathcal{S}}^{\varepsilon}_{\ell} $. Otherwise, arguing the same way than before  and since the number of regimes is finite, it must occur that there is some $\ell_{i}\neq\ell$ such that $x\in\mathcal{C}^{\varepsilon}_{\ell_{i}}$ and $u^{\varepsilon}_{\ell}(x)= u^{\varepsilon}_{\ell_{i}}(x)+\vartheta_{\ell,\ell_{i}}$. Therefore $x\in\mathcal{S}^{\varepsilon}_{\ell,\ell_{i}}\subset\widetilde{\mathcal{S}}^{\varepsilon}_{\ell} $.
\end{proof}

Now we construct the optimal control $(\xi^{\varepsilon,*},\varsigma^{\varepsilon,*})$  to the problem \eqref{Vfp1}.   Let $(\tilde{x},\tilde{\ell})$ be in $\overline{\set}\times\mathbb{I}$. The dynamics of the process $X^{\varepsilon,*}_{t}\eqdef\{X^{\varepsilon,*}_{\tmt}:\tmt\geq0\}$ and $(\xi^{\varepsilon,*},\varsigma^{\varepsilon,*})$  {is} given recursively in the following way:

\begin{enumerate}
	\item[(i)] Define $\tau^{*}_{0}=0$ and $\ell^{*}_{0-}=\tilde{\ell}$. If $\tilde{x}\notin\mathcal{C}^{\varepsilon}_{\tilde{\ell}}$, take $\tau^{*}_{1}\eqdef0$ and pass to item (ii) because of Lemma \ref{optim1}. Otherwise,  the process  $X^{\varepsilon,*}$ evolves as
\begin{multline}\label{opt1}
 {X}^{\varepsilon,*}_{\tmt\wedge\tilde{\tau}^{*}_{1}}=\tilde{x}-\int_{0}^{\tmt\wedge\tilde{\tau}^{*}_{1}}[ b ({X}^{\varepsilon,*}_{\tms},\ell^{*}_{0})+\mathbb{n}_{\tms}^{\varepsilon,*}\zeta^{\varepsilon,*}_{\tms}]\der \tms\\
 +\int_{0}^{\tmt\wedge\tilde{\tau}^{*}_{1}}\sigma({X}^{\ell,\varepsilon,*}_{\tms},\ell^{*}_{0})\der  W_{\tms},\quad\text{for}\ t>0,
\end{multline}
with $X^{\varepsilon,*}_{0}=\tilde{x}$, $\tau^{*}\eqdef\inf\{\tmt>0:X^{\varepsilon,*}_{\tmt}\notin\set\}$,
\begin{equation}\label{opt3}
\tilde{\tau}^{*}_{1}\eqdef\tau^{*}_{1}\wedge\tau^{*}\quad\text{and}\quad\tau^{*}_{1}\eqdef\inf\big\{\tmt\geq0:X^{\varepsilon,*}_{\tmt}\in\mathcal{S}^{\varepsilon}_{\ell^{*}_{0}}\big\}.
\end{equation}
The control process $\xi^{\varepsilon,*}=(\mathbb{n}^{\varepsilon,*},\zeta^{\varepsilon,*})$ is defined by
	\begin{equation}\label{opt4}
\mathbb{n}^{\varepsilon,*}_{\tmt}=
\begin{cases}
\frac{\deri^{1}u^{\varepsilon}_{\ell^{*}_{0}}(X^{\varepsilon,*}_{\tmt})}{|\deri^{1}u^{\varepsilon}_{\ell^{*}_{0}}( X^{\varepsilon,*}_{\tmt})|}, &\text{if}\ |\deri^{1}u^{\varepsilon}_{\ell^{*}_{0}}( X ^{\varepsilon,*}_{\tmt})|\neq0\ \text{and}\ \tmt\in[0,\tilde{\tau}^{*}_{1}),\\
\gamma_{0},& \text{if}\ |\deri^{1}u^{\varepsilon}_{\ell^{*}_{0}}(X^{\varepsilon,*}_{\tmt})|=0\  \text{and}\ \tmt\in[0,\tilde{\tau}^{*}_{1}), 
\end{cases}
\end{equation}
where $\gamma_{0}\in\R^{d}$ is a unit vector fixed, and $\zeta^{\varepsilon,*}_{\tmt}=\int_{0}^{\tmt}\dot{\zeta}^{\varepsilon,*}_{\tms}\der \tms$, with $\tmt\in[0,\tilde{\tau}^{*}_{1})$ and 
\begin{align}\label{opt5}
\dot{\zeta}^{\varepsilon,*}_{\tms}= 2\psi'_{\varepsilon}(|\deri^{1} u^{\varepsilon}_{\ell^{*}_{0}}(X^{\varepsilon,*}_{\tms})|^{2}- g_{\ell^{*}_{0}}(X^{\varepsilon,*}_{\tms})^{2})|\deri^{1} u^{\varepsilon}_{\ell^{*}_{0}}(X^{\varepsilon,*}_{\tms})|.
\end{align} 
\item[(ii)] Recursively, letting  $i\geq1$ and defining
\begin{equation}\label{opt6}
\begin{split}
&\ell^{*}_{i}\in\displaystyle\argmin_{\kappa\in\mathbb{I}\setminus\{\ell^{*}_{i-1}\}}\big\{u^{\varepsilon}_{\kappa}(X^{\varepsilon,*}_{\tau^{*}_{i}})+\vartheta_{\ell^{*}_{i-1},\kappa}\big\},\\
&\tilde{\tau}^{*}_{i+1}=\tau^{*}_{i+1}\wedge\tau^{*},\quad\tau^{*}_{i+1}=\inf\big\{\tmt>\tau^{*}_{i}:X^{\varepsilon,*}_{\tmt}\in\mathcal{S}^{\varepsilon}_{\ell^{*}_{i}}\big\},
\end{split}
\end{equation}  
if $\tau^{*}_{i}<\tau^{*}$, the process $X^{\varepsilon,*}$  evolves as
\begin{multline}\label{opt7}
{X}^{\varepsilon,*}_{\tmt\wedge\tilde{\tau}^{*}_{i+1}}=X^{\varepsilon,*}_{\tau_{i}^{*}}
-\int_{\tau^{*}_{i}}^{\tmt\wedge\tilde{\tau}^{*}_{i+1}}[ b ({X}^{\varepsilon,*}_{\tms},\ell^{*}_{i})+\mathbb{n}_{\tms}^{\varepsilon,*}\zeta^{\varepsilon,*}_{\tms}]\der \tms\\
+\int_{\tau^{*}_{i}}^{\tmt\wedge\tilde{\tau}^{*}_{i+1}}\sigma({X}^{\varepsilon,*}_{\tms},\ell^{*}_{i})\der  W_{\tms},\quad\text{for}\ \tmt\geq\tau^{*}_{i},
\end{multline}
where  
\begin{equation}\label{opt9}
\mathbb{n}^{\varepsilon,*}_{\tmt}=
\begin{cases}
\frac{\deri^{1}u^{\varepsilon}_{\ell^{*}_{i}}(X^{\varepsilon,*}_{\tmt})}{|\deri^{1}u^{\varepsilon}_{\ell^{*}_{i}}( X^{\varepsilon,*}_{\tmt})|}, &\text{if}\ |\deri^{1}u^{\varepsilon}_{\ell^{*}_{i}}( X ^{\varepsilon,*}_{\tmt})|\neq0\ \text{and}\ \tmt\in[\tau^{*}_{i},\tilde{\tau}^{*}_{i+1}),\\
\gamma_{0},& \text{if}\ |\deri^{1}u^{\varepsilon}_{\ell^{*}_{i}}(X^{\varepsilon,*}_{\tmt})|=0\  \text{and}\ \tmt\in[\tau^{*}_{i},\tilde{\tau}^{*}_{i+1}), 
\end{cases}
\end{equation}
with $\gamma_{0}\in\R^{d}$ is a unit vector fixed, and $\zeta^{\varepsilon,*}_{\tmt}=\int_{\tau^{\varepsilon,*}_{i}}^{\tmt}\dot{\zeta}^{\varepsilon,*}_{\tms}\der \tms$, with $\tmt\in[\tau^{*}_{i},\tilde{\tau}^{*}_{i+1})$ and 
\begin{align}\label{opt10}
\dot{\zeta}^{\varepsilon,*}_{\tms}= 2\psi'_{\varepsilon}(|\deri^{1} u^{\varepsilon}_{\ell^{*}_{i}}(X^{\varepsilon,*}_{\tms})|^{2}- g_{\ell^{*}_{i}}(X^{\varepsilon,*}_{\tms})^{2})|\deri^{1} u^{\varepsilon}_{\ell^{*}_{i}}(X^{\varepsilon,*}_{\tms})|.
\end{align} 
\end{enumerate}

\begin{rem}\label{ins1}
Suppose that $\tau^{*}_{i}<\tau^{*}$ for some $i>0$. We notice that for $\tmt\in[\tau^{*}_{i},\tau^{*}_{i+1})$,  $\mathbb{n}^{\varepsilon,*}_{\tmt}\dot{\zeta}^{\varepsilon,*}_{\tmt}=2\psi'_{\varepsilon}(|\deri^{1} u^{\varepsilon}_{\ell^{*}_{i}}(X^{\varepsilon,*}_{\tmt})|^{2}- g   (X^{\varepsilon,*}_{\tmt})^{2})\deri^{1}u^{\varepsilon}_{\ell^{*}_{i}}(X^{\varepsilon,*}_{\tmt})$, $\Delta\zeta^{\varepsilon,*}_{\tmt}=0$, $|\mathbb{n}^{\varepsilon,*}_{\tmt}|=1$ and, by \eqref{p12.1} and \eqref{ineqe1}, it yields that  $\dot{\zeta}^{\varepsilon,*}_{\tmt}\leq\frac{2C_{4}}{\varepsilon}$.  Also {we} see that $X_{\tmt}\in\mathcal{C}^{\varepsilon}_{\ell^{*}_{i}}$ if  $\tmt\in[\tau^{*}_{i},\tau^{*}_{i+1})$ due to Lemma \ref{optim1}.
\end{rem}
\begin{rem}
On the event $\{\tau^{*}=\infty\}$, $\tilde{\tau}^{*}_{i}=\tau_{i}^{*}$ for $i\geq0$.  From here and by \eqref{opt4}--\eqref{opt5} and \eqref{opt9}--\eqref{opt10},  it yields that  the control process $(\xi^{\varepsilon,*},\varsigma^{\varepsilon,*})$ belongs to $\mathcal{U}^{\varepsilon}\times\mathcal{S}$. On the event $\{\tau^{*}<\infty \}$,   let $\hat{\iota}$ be defined 
 as $\hat{\iota}=\max\{i\in\mathbb{N}:\tau^{*}_{i}\leq\tau^{*}\}$. Then, taking $\tau^{*}_{i}\eqdef\tau^{*}+i$ and $\ell^{*}_{i}=\hat{\ell}$ for $i>\hat{\iota}$, where $\hat{\ell}\in\mathbb{I}$ is fixed,  it follows that $\varsigma^{\varepsilon,*}=(\tau^{*}_{i},\ell^{*}_{i})_{i\geq1}\in\mathcal{S}$. Meanwhile,  since \eqref{ineqe1} holds on $\overline{\set}$  and $u^{\varepsilon}_{\ell}=0$ on $\overline{\set}$,  we take $\dot{\zeta}^{\varepsilon,*}_{\tmt}\equiv 0$ and  $\mathbb{n}^{\varepsilon,*}_{\tmt}\eqdef\gamma_{0}$, for $\tmt>\tau^{*}$. In this way, we have that $(\mathbb{n}^{\varepsilon,*},\zeta^{\varepsilon,*})\in\mathcal{U}^{\varepsilon}$.
\end{rem}

\begin{rem}
	Taking $I^{\varepsilon,*}_{\tmt}=\tilde{\ell}\uno_{[0,\tau^{*}_{1})}(\tmt)+\sum_{i\geq1}\ell^{*}_{i}\uno_{[\tau^{*}_{i},\tau^{*}_{i+1})}(\tmt)$, we see that it is a c\`adl\`ag process.
\end{rem}

\begin{lema}[Verification Lemma for penalized absolutely continuous/switching control problem. Second part]\label{convexu1.0}
	Let  $\varepsilon\in(0,1)$ be fixed and let  {$(X^{\varepsilon,*},I^{\varepsilon,*})$ be the process that is governed by} \eqref{opt1}--\eqref{opt10}. Then,  $u^{\varepsilon}_{\tilde{\ell}}(\tilde{x})=\mathcal{V}_{\xi^{\varepsilon,*},\varsigma^{\varepsilon,*}}(\tilde{x},\tilde{\ell})=V^{\varepsilon}_{\tilde{\ell}}(\tilde{x})$ for each $(\tilde{x},\tilde{\ell})\in\overline{\set}\times\mathbb{I}$. 
\end{lema}
\begin{proof}
 Take    $\hat{\tau}_{i}^{*,q}\eqdef\tau_{i}^{*}\wedge\inf\{ \tmt>\tau_{i-1}^{*}:X^{\varepsilon,*}_{\tmt}\notin\set_{q}\}$, with $\set_{q}\eqdef\{x\in\set:\dist(x,\partial\set )>1/q\}$ and $q$ a positive integer  {large enough}. By Remarks \ref{reg1} and \ref{ins1}, it is known that $u^{\varepsilon}_{\ell}$ is a $\hol^{2}$-function on $\mathcal{C}^{\varepsilon}_{\ell}$ and that $X^{\varepsilon,*}_{\tmt}\in\mathcal{C}^{\varepsilon}_{\ell^{*}_{i}}$ if  $\tmt\in[\hat{\tau}^{*,q}_{i},\hat{\tau}^{*,q}_{i+1})$. Then, using   integration by parts and It\^o's formula in $\expo^{-r(\tmt)}u^{\varepsilon}_{\ell^{*}_{i}}(X^{\varepsilon,*}_{\tmt})$ on the interval $[\hat{\tau}^{*,q}_{i},\hat{\tau}^{*,q}_{i+1})$, an taking expected value on  it, we obtain a similar expression as in \eqref{expand.1.0}. Now, considering that $\widetilde{\mathcal{M}}_{\ell^{*}_{i}}[\hat{\tau}_{i}^{*,q},\tmt\wedge\hat{\tau}^{*,q}_{i+1};X^{\varepsilon,*},u^{\varepsilon}]$, which was defined in \eqref{mar1}, is a square integrable martingale,  and since $[c_{\ell^{*}_{i}}-\dif_{\ell^{*}_{i}}]  u_{\ell^{*}_{i}}^{\varepsilon}=h_{\ell^{*}_{i}}-\psi_{\varepsilon}(|\deri^{1} u_{\ell^{*}_{i}}^{\varepsilon}|^{2}- g_{\ell^{*}_{i}}^{2})$ on $\mathcal{C}^{\varepsilon}_{\ell^{*}_{i}}$ and  the supremum  of $l^{\varepsilon}_{\ell}(\eta,x)$ is attained if $\gamma$ is related to $\eta$ by $\eta=2\psi'_{\varepsilon}(|\gamma|^{2}-g_{\ell}(x)^{2})\gamma$, i.e., $l^{\varepsilon}_{\ell}(2\psi'_{\varepsilon}(|\gamma|^{2}- g_{\ell}(x)^{2})\gamma,x)=2\psi'_{\varepsilon}(|\gamma|^{2}- g_{\ell}(x)^{2})|\gamma|^{2}-\psi_{\varepsilon}(|\gamma|^{2}- g_{\ell}(x)^{2})$, it can be checked 
\begin{align}\label{ex1.1}
\E_{\tilde{x},\tilde{\ell}}\Big[\expo^{-r(\hat{\tau}^{*,q}_{i})}u^{\varepsilon}_{\ell^{*}_{i}}(X^{\varepsilon,*}_{{\hat{\tau}}^{*,q}_{i}})\Big]&=\E_{\tilde{x},\tilde{\ell}}\bigg[[u^{\varepsilon}_{\ell^{*}_{1}}(X^{\varepsilon,*}_{\tau^{*}_{1}})+\vartheta_{\ell^{*}_{0},\ell^{*}_{1}}]\uno_{\{\tau^{*}_{1}=0,\tau^{*}_{1}<\tau^{*},i=0\}}\notag\\ 
&\quad+[u^{\varepsilon}_{\ell^{*}_{0}} (X^{\varepsilon,*}_{\tau^{*}_{1}})-u^{\varepsilon}_{\ell^{*}_{1}}(X^{\varepsilon,*}_{\tau^{*}_{1}} )-\vartheta_{\ell^{*}_{0},\ell^{*}_{1}}]\uno_{\{\tau^{*}_{1}=0,\tau^{*}_{1}<\tau^{*},i=0\}}\notag\\
&\quad+\expo^{-r(\tmt\wedge{\hat{\tau}}^{*,q}_{1})}u^{\varepsilon}_{\ell^{*}_{0}}(X^{\varepsilon,*}_{\tmt\wedge{\hat{\tau}}^{*,q}_{1}})\uno_{\{\tau^{*}_{1}\neq0,i=0\}}+\expo^{-r(\tmt\wedge{\hat{\tau}}^{*,q}_{i+1})}u^{\varepsilon}_{\ell^{*}_{i}}(X^{\varepsilon,*}_{\tmt\wedge{\hat{\tau}}^{*,q}_{i+1}})\uno_{\{i\neq0\}}\notag\\
&\quad+\uno_{\{\tau^{*}_{1}\neq0,i=0\}}\int_{0}^{\tmt\wedge{\hat{\tau}}^{*,q}_{1}}\expo^{-r(\tms)}[h_{\ell^{*}_{0}}(X^{\varepsilon,*}_{\tms})+l^{\varepsilon}_{\ell^{*}_{0}}(\dot{\zeta}^{*}_{\tms}\mathbb{n}^{*}_{\tms},X^{\varepsilon,*}_{\tms})]\der \tms\notag\\
&\quad+\uno_{\{i\neq0\}}\int_{{\hat{\tau}}^{*,q}_{i}}^{\tmt\wedge{\hat{\tau}}^{*,q}_{i+1}}\expo^{-r(\tms)}[h_{\ell^{*}_{i}}(X_{\tms})+l^{\varepsilon}_{\ell^{*}_{i}}(\dot{\zeta}^{*}_{\tms}\mathbb{n}^{*}_{\tms},X^{\varepsilon,*}_{\tms})]\der \tms\bigg],
\end{align}
Notice that $\hat{\tau}^ {*,q}_{i}\uparrow\tilde{\tau}^{*}_{i}$ as $q\rightarrow\infty$, $\Pro_{\tilde{x}}$-a.s.. Consequently,  letting first $q\rightarrow\infty$ and after $\tmt\rightarrow\infty$ in \eqref{ex1.1}, we see that
\begin{multline}\label{ex4.1}
\E_{\tilde{x},\tilde{\ell}}\Big[\expo^{-r(\tilde{\tau}^{*}_{i})}u^{\varepsilon}_{\ell^{*}_{i}}(X^{\varepsilon,*}_{{\tilde{\tau}^{*}}_{i}})\Big]=\E_{\tilde{x},\tilde{\ell}}\bigg[\{\expo^{-r(\tau^{*}_{i+1})}u_{\ell^{*}_{i+1}}^{\varepsilon}(X^{\varepsilon,*}_{\tau^{*}_{i+1}})+\mathcal{D}[\tau^{*}_{i+1},\ell^{*}_{i},\ell^{*}_{i+1},X^{\varepsilon,*}, u^{\varepsilon}]\}\uno_{\{\tau^{*}_{i+1}<\tau^{*}\}}\\
+\int_{{\tilde{\tau}^{*}}_{i}}^{{\tilde{\tau}^{*}}_{i+1}}\expo^{-r(\tms)}[h_{\ell^{*}_{i}}(X^{\varepsilon,*}_{\tms})+l^{\varepsilon}_{\ell^{*}_{i}}(\dot{\zeta}^{*}_{\tms}\mathbb{n}^{*}_{\tms},X^{\varepsilon,*}_{\tms})]\der \tms+\expo^{-r(\tau^{*}_{i+1})}\vartheta_{\ell^{*}_{i},\ell^{*}_{i+1}}\uno_{\{\tau^{*}_{i+1}<\tau^{*}\}}\bigg]\quad\text{for}\ i\geq0,
\end{multline}
with $\mathcal{D}[\tau^{*}_{i+1},\ell^{*}_{i},\ell^{*}_{i+1},X^{\varepsilon,*}, u^{\varepsilon}]$ as in \eqref{ex5}. By \eqref{opt3} and \eqref{opt6},  
$$\mathcal{D}[\tau^{*}_{i+1},\ell^{*}_{i},\ell^{*}_{i+1},X^{\varepsilon,*}, u^{\varepsilon}]\uno_{\{\tau^{*}_{i+1}<\tau^{*}\}}=0.$$
Therefore, from here and \eqref{ex4.1}, we obtain the desired result that was given in the lemma above.
\end{proof}

\subsubsection*{Proof of Theorem \ref{verf2}}
To finalize this section, we present the proof of Theorem \ref{verf2}{, which shall be proven under assumptions \eqref{h3} and {\eqref{a4}--\eqref{h5}}}.  Recall that $\mathcal{U}$ and   $\mathcal{S}$  are the families of controls $\xi=(\mathbb{n},\zeta)$ and $\varsigma=(\tau_{i},\ell_{i})_{i\geq0}$ that satisfy  \eqref{cont.1} and \eqref{cont.2}, respectively. The functional cost $V_{\xi,\varsigma}$ is given by \eqref{esd1.1.1} for $(\xi,\varsigma)\in\mathcal{U}\times\mathcal{S}$ . 

\begin{proof}[Proof of Theorem \ref{verf2}]
 Let $\{u^{\varepsilon_{n}}\}_{n\geq1}$ be the sequence of unique strong solutions to the HJB equation \eqref{pc1}, when $\varepsilon=\varepsilon_{n}$,  which  satisfy \eqref{econv1}. From Lemma \ref{convexu1.0}, we know that  
$$u^{\varepsilon_{n}}_{\tilde{\ell}}(\tilde{x})=\mathcal{V}_{\xi^{\varepsilon_{n},*},\varsigma^{\varepsilon_{n},*}}(\tilde{x},\tilde{\ell})= V^{\varepsilon_{n}}(\tilde{x},\tilde{\ell})\quad\text{for $(\tilde{x},\tilde{\ell})\in\overline{\set}\times\mathbb{I}$},$$
with $(\xi^{\varepsilon_{n},*},\varsigma^{\varepsilon_{n},*})$ as in \eqref{opt3}--\eqref{opt6}, when $\varepsilon=\varepsilon_{n}$. Notice that $l_{\varepsilon_{n}}(x,\beta\gamma)\geq\langle\beta\gamma, g_{\ell^{*}_{i}}   (x)\gamma\rangle-\psi_{\varepsilon_{n}}(| g_{\ell^{*}_{i}}(x)\gamma|^{2}- g_{\ell^{*}_{i}}(x)^{2})=\beta g _{\ell^{*}_{i}}  (x)$, with $\beta\in\R$ and $\gamma\in\R^{d}$ a unit vector. Then, from here and considering  {$(X^{\varepsilon_{n},*},I^{\varepsilon_{n},*})$ governed by} \eqref{opt1} and \eqref{opt7}, it follows that 
\begin{align}\label{NPIDD1.0}
V_{\tilde{\ell}}(\tilde{x})&\leq V_{\xi^{\varepsilon_{n},*},\varsigma^{\varepsilon_{n},*}}(\tilde{x},\tilde{\ell})\notag\\
&=\sum_{i\geq0}\E_{\tilde{x},\tilde{\ell}}\bigg[\int_{\tilde{\tau}^{*}_{i}}^{\tilde{\tau}^{*}_{i+1}}\expo^{-r(\tms)}[h_{\ell_{i}^{*}}(X^{\varepsilon_{n},*}_{\tms})\der \tms+\dot{\zeta}^{*}_{\tms}g_{\ell^{*}_{i}}(X^{\varepsilon_{n},*}_{\tms-})]\der\tms+\expo^{-r(\tau^{*}_{i+1})}\vartheta_{\ell^{*}_{i},\ell^{*}_{i+1}}\uno_{\{\tau^{*}_{i+1}<\tau^{*}\}}\bigg]\notag\\
&\leq \sum_{i\geq0}\E_{\tilde{x},\tilde{\ell}}\bigg[\int_{{\tilde{\tau}}^{*}_{i}}^{{\tilde{\tau}}^{*}_{i+1}}\expo^{-r(\tms)}[h_{\ell_{i}^{*}}(X^{\varepsilon_{n},*}_{\tms})+l^{\varepsilon_{n}}_{\ell^{*}_{i}}(\dot{\zeta}^{*}_{\tms}\mathbb{n}^{*}_{\tms},X^{\varepsilon_{n},*}_{\tms})]\der \tms+\expo^{-r(\tau^{*}_{i+1})}\vartheta_{\ell^{*}_{i},\ell^{*}_{i+1}}\uno_{\{\tau^{*}_{i+1}<\tau^{*}\}}\bigg]\notag\\
&=u^{\varepsilon_{n}}_{\tilde{\ell}}(\tilde{x}).
\end{align}
Notice that $V_{\xi^{\varepsilon_{n},*},\varsigma^{\varepsilon_{n},*}}$ is the cost function given in \eqref{esd1.1.1} corresponding to the control $(\xi^{\varepsilon_{n},*},\varsigma^{\varepsilon_{n},*})$, where the second term in the RHS  of \eqref{eq0.1.1} is zero, since $\zeta^{\varepsilon_{n},*}$  has the continuous part only.  Letting $\varepsilon_{n}\rightarrow0$ in \eqref{NPIDD1.0}, it yields $V\leq u$ on $\overline{\set}$. 

Let $\{u^{\varepsilon_{n},\delta_{\hat{n}}}\}_{n,\hat{n}\geq1}$ be the sequence of unique solutions to the NPDS \eqref{NPD.1}, when $\varepsilon=\varepsilon_{n}$ and $\delta=\delta_{\hat{n}}$, which  satisfy \eqref{conv1} and \eqref{econv1}.  Let us consider $(X,I)$ evolves as in \eqref{es1}  with initial state $(\tilde{x},\tilde{\ell})\in\overline{\set}\times\mathbb{I}$ and the control process $(\xi,\varsigma)$ belongs to $\mathcal{U}\times\mathcal{S}$.   Take    $\hat{\tau}_{i}^{q}\eqdef\tau_{i}\wedge\inf\{ \tmt>\tau_{i-1}:X_{\tmt}\notin\set_{q}\}$, with $i\geq1$, $\set_{q}\eqdef\{x\in\set:\dist(x,\partial\set )>1/q\}$ and $q$ a positive integer  {large enough}.  Using integration by parts and It\^o's formula in $\expo^{-r(\tmt)}u^{\varepsilon_{n},\delta_{\hat{n}}}_{\ell_{i}}(X_{\tmt})$ on $ [\hat{\tau}^{q}_{i},\hat{\tau}^{q}_{i+1})$, $i\geq0$,  we get that 
\begin{align}\label{s.1}
&\expo^{-r(\hat{\tau}^{q}_{i})}u^{\varepsilon_{n},\delta_{\hat{n}}}_{\ell_{i}}(X_{{\hat{\tau}}^{q}_{i}})=[u^{\varepsilon_{n},\delta_{\hat{n}}}_{\ell_{1}}(X_{\tau_{1}})+\vartheta_{\ell_{0},\ell_{1}}]\uno_{\{\tau_{1}=0,\tau_{1}<\tau,i=0\}}\notag\\ 
&\quad+[u^{\varepsilon_{n},\delta_{\hat{n}}}_{\ell_{0}} (X_{\tau_{1}})-u^{\varepsilon_{n},\delta_{\hat{n}}}_{\ell_{1}}(X_{\tau_{1}} )-\vartheta_{\ell_{0},\ell_{1}}]\uno_{\{\tau_{1}=0,\tau_{1}<\tau,i=0\}}\notag\\
&\quad+\expo^{-r(\tmt\wedge{\hat{\tau}}^{q}_{1})}u^{\varepsilon_{n},\delta_{\hat{n}}}_{\ell_{0}}(X_{\tmt\wedge{\hat{\tau}}^{q}_{1}})\uno_{\{\tau_{1}\neq0,i=0\}}+\expo^{-r(\tmt\wedge{\hat{\tau}}^{q}_{i+1})}u^{\varepsilon_{n},\delta_{\hat{n}}}_{\ell_{i}}(X_{\tmt\wedge{\hat{\tau}}^{q}_{i+1}})\uno_{\{i\neq0\}}\notag\\
&\quad+\uno_{\{\tau_{1}\neq0,i=0\}}\int_{0}^{\tmt\wedge{\hat{\tau}}^{q}_{1}}\expo^{-r(\tms)}[[c_{\ell_{0}}(X_{\tms})u^{\varepsilon_{n},\delta_{\hat{n}}}_{\ell_{0}}(X_{\tms})-\dif_{\ell_{0}}u^{\varepsilon_{n},\delta_{\hat{n}}}_{\ell_{0}}(X_{\tms})]\der\tms+\langle\deri^{1} u^{\varepsilon_{n},\delta_{\hat{n}}}_{\ell_{0}}(X_{\tms}),\mathbb{n}_{\tms}\rangle\der\zeta^{\comp}_{\tms}]\notag\\
&\quad+\uno_{\{i\neq0\}}\int_{{\hat{\tau}}^{q}_{i}}^{\tmt\wedge{\hat{\tau}}^{q}_{i+1}}\expo^{-r(\tms)}[[c_{\ell_{i}}(X_{\tms})u^{\varepsilon_{n},\delta_{\hat{n}}}_{\ell_{i}}(X_{\tms})-\dif_{\ell_{i}}u^{\varepsilon_{n},\delta_{\hat{n}}}_{\ell_{i}}(X_{\tms})]\der\tms+\langle\deri^{1} u^{\varepsilon_{n},\delta_{\hat{n}}}_{\ell_{i}}(X_{\tms}),\mathbb{n}_{\tms}\rangle\der\zeta^{\comp}_{\tms}]\notag\\
&\quad-\sum_{0\leq\tms<\tmt\wedge\hat{\tau}^{q}_{1}}\expo^{-r(\tms)}\mathcal{J}[\tms;\ell_{0},X, u^{\varepsilon_{n},\delta_{\hat{n}}}]\uno_{\{\tau_{1}\neq0,i=0\}}-\sum_{\hat{\tau}^{q}_{i}\leq\tms<\tmt\wedge\hat{\tau}^{q}_{i+1}}\expo^{-r(\tms)}\mathcal{J}[\tms;\ell_{i},X, u^{\varepsilon_{n},\delta_{\hat{n}}}]\uno_{\{i\neq0\}}\notag\\
&\quad-\widetilde{\mathcal{M}}_{\ell_{0}}[0,\tmt\wedge\hat{\tau}^{q}_{1};X,u^{\varepsilon_{n},\delta_{\hat{n}}}]\uno_{\{\tau_{1}\neq0,i=0\}}-\widetilde{\mathcal{M}}_{\ell_{i}}[\hat{\tau}^{q}_{i},\tmt\wedge\hat{\tau}^{q}_{i+1};X,u^{\varepsilon_{n},\delta_{\hat{n}}}]\uno_{\{i\neq0\}},
\end{align}
where $\widetilde{\mathcal{M}}_{\ell_{i}}[\hat{\tau}^{q}_{i},\tmt\wedge\hat{\tau}^{q}_{i+1};X,u^{\varepsilon_{n},\delta_{\hat{n}}}]$ is as in \eqref{mar1} and
\begin{equation*}
\mathcal{J}[\tms;\ell_{i},X, u^{\varepsilon_{n},\delta_{\hat{n}}}]\eqdef u^{\varepsilon_{n},\delta_{\hat{n}}}_{\ell_{i}}(X_{\tms-}-\mathbb{n}_{\tms}\Delta \zeta_{\tms})-u^{\varepsilon_{n},\delta_{\hat{n}}}_{\ell_{i}}(X_{\tms-}),\quad \text{for}\ i\geq0.
\end{equation*}
 Since $X_{\tms-}-\mathbb{n}_{\tms}\Delta\xi_{\tmt}\in \set$ for $\tms\in[\hat{\tau}^{q}_{i},\tmt\wedge\hat{\tau}^{q}_{i+1})$, $i\geq0$, and $\set$ is a convex set, by Mean Value Theorem,  it implies
\begin{align*}%\label{esd14}
-\mathcal{J}[\tms;\ell_{i},X, u^{\varepsilon_{n},\delta_{\hat{n}}}]&\leq |u^{\varepsilon_{n},\delta_{\hat{n}}}_{\ell_{i}}(X_{\tms-}-\mathbb{n}_{\tms}\Delta \zeta_{\tms})-u^{\varepsilon_{n},\delta_{\hat{n}}}_{\ell_{i}}(X_{\tms-})|\notag\\
&\leq\Delta\zeta_{\tms}\int_{0}^{1}|\deri^{1} u^{\varepsilon_{n},\delta_{\hat{n}}}_{\ell_{i}}( X_{\tms-}-\lambda \mathbb{n}_{\tms}\Delta\zeta_{\tms})|\der\lambda.
\end{align*}
Taking expected value in \eqref{s.1} and since $\widetilde{\mathcal{M}}_{\ell_{i}}[\hat{\tau}_{i}^{q},\tmt\wedge\hat{\tau}^{q}_{i+1};X,u^{\varepsilon_{n},\delta_{\hat{n}}}]$ is a square integrable martingale  and $[c_{\ell_{i}}-\dif_{\ell_{i}}]u^{\varepsilon_{n},\delta_{\hat{n}}}_{\ell_{i}}\leq h_{\ell_{i}}$, it follows that  
\begin{align}\label{s.2}
\E_{\tilde{x},\tilde{\ell}}\big[\expo^{-r(\hat{\tau}^{q}_{i})}u^{\varepsilon_{n},\delta_{\hat{n}}}_{\ell_{i}}&(X_{{\hat{\tau}}^{q}_{i}})\big]\leq\E_{\tilde{x},\tilde{\ell}}\bigg[[u^{\varepsilon_{n},\delta_{\hat{n}}}_{\ell_{1}}(X_{\tau_{1}})+\vartheta_{\ell_{0},\ell_{1}}]\uno_{\{\tau_{1}=0,\tau_{1}<\tau,i=0\}}\notag\\ 
&\quad+[u^{\varepsilon_{n},\delta_{\hat{n}}}_{\ell_{0}} (X_{\tau_{1}})-u^{\varepsilon_{n},\delta_{\hat{n}}}_{\ell_{1}}(X_{\tau_{1}} )-\vartheta_{\ell_{0},\ell_{1}}]\uno_{\{\tau_{1}=0,\tau_{1}<\tau,i=0\}}\notag\\
&\quad+\expo^{-r(\tmt\wedge{\hat{\tau}}^{q}_{1})}u^{\varepsilon_{n},\delta_{\hat{n}}}_{\ell_{0}}(X_{\tmt\wedge{\hat{\tau}}^{q}_{1}})\uno_{\{\tau_{1}\neq0,i=0\}}+\expo^{-r(\tmt\wedge{\hat{\tau}}^{q}_{i+1})}u^{\varepsilon_{n},\delta_{\hat{n}}}_{\ell_{i}}(X_{\tmt\wedge{\hat{\tau}}^{q}_{i+1}})\uno_{\{i\neq0\}}\notag\\
&\quad+\uno_{\{\tau_{1}\neq0,i=0\}}\int_{0}^{\tmt\wedge{\hat{\tau}}^{q}_{1}}\expo^{-r(\tms)}[h_{\ell_{0}}(X_{\tms})\der\tms+|\deri^{1} u^{\varepsilon_{n},\delta_{\hat{n}}}_{\ell_{0}}(X_{\tms})|\der\zeta^{\comp}_{\tms}]\notag\\
&\quad+\uno_{\{i\neq0\}}\int_{{\hat{\tau}}^{q}_{i}}^{\tmt\wedge{\hat{\tau}}^{q}_{i+1}}\expo^{-r(\tms)}[h_{\ell_{i}}(X_{\tms})\der\tms+|\deri^{1} u^{\varepsilon_{n},\delta_{\hat{n}}}_{\ell_{i}}(X_{\tms})|\der\zeta^{\comp}_{\tms}]\notag\\
&\quad+\sum_{0\leq\tms<\tmt\wedge\hat{\tau}^{q}}\expo^{-r(\tms)}\Delta\zeta_{\tms}\int_{0}^{1}|\deri^{1} u^{\varepsilon_{n},\delta_{\hat{n}}}_{\ell_{0}}( X_{\tms-}-\lambda \mathbb{n}_{\tms}\Delta\zeta_{\tms})|\der\lambda\uno_{\{\tau_{1}\neq0,i=0\}}\notag\\
&\quad+\sum_{\hat{\tau}^{q}_{i}\leq\tms<\tmt\wedge\hat{\tau}^{q}_{i+1}}\expo^{-r(\tms)}\Delta\zeta_{\tms}\int_{0}^{1}|\deri^{1} u^{\varepsilon_{n},\delta_{\hat{n}}}_{\ell_{i}}( X_{\tms-}-\lambda \mathbb{n}_{\tms}\Delta\zeta_{\tms})|\der\lambda\uno_{\{i\neq0\}}\bigg],
\end{align}
 Notice that  for  $\tms\in[\hat{\tau}^{q}_{i},\tmt\wedge\hat{\tau}^{q}_{i+1})$, 
 \begin{align*}
 &|u^{\varepsilon_{n},\delta_{\hat{n}}}_{\ell_{i}}(X_{\tms})-u_{\ell_{i}}(X_{\tms})|\leq\max_{(x,\ell)\in\set_{q}\times\mathbb{I}}|u^{\varepsilon_{n},\delta_{\hat{n}}}_{\ell}(x)-u_{\ell}(x)|\underset{\varepsilon_{n},\delta_{\hat{n}}\rightarrow0}{\longrightarrow}0,\\ &|\partial_{j}u^{\varepsilon_{n},\delta_{\hat{n}}}_{\ell_{i}}(X_{\tms})-\partial_{j}u_{\ell_{i}}(X_{\tms})|\leq\max_{(x,\ell)\in\set_{q}\times\mathbb{I}}|\partial_{j}u^{\varepsilon_{n},\delta_{\hat{n}}}_{\ell}(x)-\partial_{j}u_{\ell}(x)|\underset{\varepsilon_{n},\delta_{\hat{n}}\rightarrow0}{\longrightarrow}0,\\
 &|\partial_{j}u^{\varepsilon_{n},\delta_{\hat{n}}}_{\ell_{i}}(X_{\tms}-\lambda \mathbb{n}_{\tms}\Delta\zeta_{\tms})-\partial_{j}u_{\ell_{i}}(X_{\tms}-\lambda \mathbb{n}_{\tms}\Delta\zeta_{\tms})|\leq\max_{(x,\ell)\in\set_{q}\times\mathbb{I}}|\partial_{j}u^{\varepsilon_{n},\delta_{\hat{n}}}_{\ell}(x)-\partial_{j}u_{\ell}(x)|\underset{\varepsilon_{n},\delta_{\hat{n}}\rightarrow0}{\longrightarrow}0,
 \end{align*} 
 with $\lambda\in[0,1]$. Then, letting $\varepsilon_{n},\delta_{\hat{n}}\rightarrow0$ in \eqref{s.2}, by  Dominated Convergence Theorem and using $|\deri^{1} u_{\ell_{i}}|\leq g_{\ell_{i}}$,  
\begin{align}\label{s.3}
\E_{\tilde{x},\tilde{\ell}}\big[\expo^{-r(\hat{\tau}^{q}_{i})}u_{\ell_{i}}&(X_{{\hat{\tau}}^{q}_{i}})\big]\leq\E_{\tilde{x},\tilde{\ell}}\bigg[[u _{\ell_{1}}(X_{\tau_{1}})+\vartheta_{\ell_{0},\ell_{1}}]\uno_{\{\tau_{1}=0,\tau_{1}<\tau,i=0\}}\notag\\ 
&\quad+[u _{\ell_{0}} (X_{\tau_{1}})-u _{\ell_{1}}(X_{\tau_{1}} )-\vartheta_{\ell_{0},\ell_{1}}]\uno_{\{\tau_{1}=0,\tau_{1}<\tau,i=0\}}\notag\\
&\quad+\expo^{-r(\tmt\wedge{\hat{\tau}}^{q}_{1})}u _{\ell_{0}}(X_{\tmt\wedge{\hat{\tau}}^{q}_{1}})\uno_{\{\tau_{1}\neq0,i=0\}}+\expo^{-r(\tmt\wedge{\hat{\tau}}^{q}_{i+1})}u _{\ell_{i}}(X_{\tmt\wedge{\hat{\tau}}^{q}_{i+1}})\uno_{\{i\neq0\}}\notag\\
&\quad+\uno_{\{\tau_{1}\neq0,i=0\}}\int_{0}^{\tmt\wedge{\hat{\tau}}^{q}_{1}}\expo^{-r(\tms)}[h_{\ell_{0}}(X_{\tms})\der\tms+g_{\ell_{0}}(X_{\tms-})\circ\der\zeta^{\comp}_{\tms}]\notag\\
&\quad+\uno_{\{i\neq0\}}\int_{{\hat{\tau}}^{q}_{i}}^{\tmt\wedge{\hat{\tau}}^{q}_{i+1}}\expo^{-r(\tms)}[h_{\ell_{i}}(X_{\tms})\der\tms+g_{\ell_{i}}(X_{\tms-})\circ\der\zeta_{\tms}]\bigg],
\end{align}
with $g_{\ell_{i}}(X_{\tms-})\circ\der\zeta_{\tms-}$ as in \eqref{eq0.1.1}. Again, letting $q\rightarrow\infty$ and $\tmt\rightarrow\infty$ in \eqref{s.3}, it implies that
\begin{multline}\label{s.4}
\E_{\tilde{x},\tilde{\ell}}\Big[\expo^{-r(\tilde{\tau}_{i})}u _{\ell_{i}}(X_{{\tilde{\tau}}_{i}})\Big]\leq\E_{\tilde{x},\tilde{\ell}}\bigg[\{\expo^{-r(\tau_{i+1})}u_{\ell_{i+1}} (X_{\tau_{i+1}})+\mathcal{D}[\tau_{i+1},\ell_{i},\ell_{i+1},X, u ]\}\uno_{\{\tau_{i+1}<\tau\}}\\
+\int_{{\tilde{\tau}}_{i}}^{{\tilde{\tau}}_{i+1}}\expo^{-r(\tms)}[h_{\ell_{i}}(X_{\tms})\der \tms+g_{\ell_{i}}(X_{\tms-})\circ\der\zeta_{\tms}]+\expo^{-r(\tau_{i+1})}\vartheta_{\ell_{i},\ell_{i+1}}\uno_{\{\tau_{i+1}<\tau\}}\bigg]\quad\text{for}\ i\geq0,
\end{multline}
with $\tilde{\tau}_{i}=\tau_{i}\wedge\tau$ and $\mathcal{D}[\tau_{i+1},\ell_{i},\ell_{i+1},X, u ]$ as in \eqref{ex5}. Noticing that 
$$\mathcal{D}[\tau_{i+1},\ell_{i},\ell_{i+1},X, u ]\uno_{\{\tau_{i+1}<\tau\}}\leq0$$ 
and using \eqref{esd1.1.1}, \eqref{s.4}, it yields that  $u_{\tilde{\ell}}(\tilde{x})\leq V_{\xi,\varsigma}(\tilde{x},\tilde{\ell})$ for each $(\tilde{x},\tilde{\ell})\in\overline{\set}\times\mathbb{I}$. Since the previous property is true for each control $(\xi,\varsigma)\in\mathcal{U}\times\mathcal{S}$, we conclude that $u_{\tilde{\ell}}(\tilde{x})\leq V_{\tilde{\ell}}(\tilde{x})$ for each $(\tilde{x},\tilde{\ell})\in\overline{\set}\times\mathbb{I}$.
\end{proof}

\subsection{About penalized absolutely continuous /switching  optimal controls}
Previously, we checked that the value function $V_{ \ell}$, defined in \eqref{vf1}, satisfies the HJB \eqref{esd5}. This means that for each $\ell\in\mathbb{I}$ fixed, the domain set $\set$ is divided into three parts.  Consider $\mathcal{C}_{\ell}\eqdef\set\setminus(\mathcal{N}_{\ell}\cup \mathcal{S}_{\ell})$, where 
\begin{equation*}
\mathcal{N}_{\ell}\eqdef\{x\in\set:|\deri^{1}V_{\ell}|= g_{\ell}\}\quad \text{and} \quad \mathcal{S}_{\ell}=\{x\in\set:V_{\ell}=\mathcal{M}_{\ell}V\}.
\end{equation*}
The open set $\mathcal{C}_{\ell}$ is where $V_{\ell}$  satisfies the elliptic partial differential equation $[c_{\ell}-\dif_{\ell}] V_{\ell}=h_{\ell}$, which suggests that the `optimal  control' $(\xi^{*},\varsigma^{*})$ corresponding to this problem will not be exercised when the process  $X^{\xi^{*},\varsigma^{*}}$ is in $\mathcal{C}_{\ell}$. Otherwise, either  $\xi^{*}$ or $\varsigma^{*}$ will be exercised on $X^{\xi^{*},\varsigma^{*}}$ in the following way:
\begin{enumerate}
\item[(i)] if $X^{\xi^{*},\varsigma^{*}}\in\mathcal{N}_{\ell}\setminus\mathcal{S}_{\ell}$, the singular control $\xi^{*}$ will act on $X^{\xi^{*},\varsigma^{*}}$  in such a way that $ X^{\xi^{*},\varsigma^{*}} $ will be pushed back to some point $y \in\partial\mathcal{C}_{\ell}$;
\item[(ii)] if $X^{\xi^{*},\varsigma^{*}}\in \mathcal{S}_{\ell}$, the switching control $\varsigma^{*}$ will be executed  in such a way that the process $ X^{\xi^{*},\varsigma^{*}} $ will switch   to some regime $\kappa\neq\ell$ at time $\tau_{\kappa}\leq\tau$.
\end{enumerate}
The construction of  an optimal strategy to the problem \eqref{vf1} still remains an open problem of interest. A way for doing this construction  is verifying firstly that $\partial\mathcal{C}_{\ell}$ is at least of class $\hol^{1}$, which is not easy to get.

 In the literature, we can see that the existence of an optimal dividend/switching strategy, which is an example of a mixed singular/switching control problem, has been solved when the payoff expected value is given by \eqref{py1} and the cash reserve process switches between $m$-regimes governed by different Poisson processes with drifts, see \cite{AM2015}.  The authors proved that their solution  is stationary with a band structure.

Another way to address the problem \eqref{vf1} is by means of $\varepsilon$-penalized absolutely continuous/switching optimal controls which have been  constructed in \eqref{opt1}--\eqref{opt10}. 

%Assuming without loss of generality that  $\psi_{\varepsilon}(t)=\varphi(t/\varepsilon)$,  with $\varphi$  as in \eqref{p12.1},  is nondecreasing as $\varepsilon  \downarrow  0$, 

By Lemma \ref{convexu1.0} and the proof of Theorem \ref{verf2}, it is known  that for each $\ell\in\mathbb{I}$, $V^{\varepsilon_{n}}_{\ell}\rightarrow V_{\ell}$ as $\varepsilon_{n}\downarrow0$  and
$V_{\ell}\leq V_{\xi^{\varepsilon_{n},*},\varsigma^{\varepsilon_{n},*}}(\cdot, \ell)\leq V^{\varepsilon_{n}}_{\ell}\ \text{on $\overline{\set}$}$,
with $(\xi^{\varepsilon_{n},*},\varsigma^{\varepsilon_{n},*})$ as in \eqref{opt1}--\eqref{opt10}, and  $V_{\xi^{\varepsilon_{n},*},\varsigma^{\varepsilon_{n},*}}(\cdot,\ell),V_{\ell}, V^{\varepsilon_{n}}_{\ell}$ given by \eqref{esd1.1.1}, \eqref{vf1} and \eqref{Vfp1}, respectively.

 Taking $\varepsilon_{n}$ small enough and considering that the process $X^{ \varepsilon_{n},*}_{\tmt}$ is on the regime $\ell^{*}_{i}$ at time $\tmt\in[\tau_{i}^{{*}},\tilde{\tau}_{i+1}^{*})$, we have that $X^{ \varepsilon_{n},*}_{\tmt}\in\mathcal{C}_{\ell_{i}^{*}}^{\varepsilon_{n}}$ and the control $(\xi^{\varepsilon_{n},*},\varsigma^{\varepsilon_{n},*})$ will be  exercised as follows:  
 \begin{enumerate}
 \item[(i)] if the controlled process $X^{ \varepsilon_{n},*}$ satisfies $|\deri^{1}V^{\varepsilon_{n}}_{\ell_{i}^{*}}(X^{\varepsilon_{n},*}_{\tmt})|\leq g_{\ell_{i}^{*}}(X^{\varepsilon_{n},*}_{\tmt})$,  then $\zeta^{\varepsilon_{n},*}_{\tmt}\equiv0$ and $X^{\varepsilon_{n},*}_{\tmt}$ will stay in $\mathcal{C}_{\ell_{i}^{*}}$;
 
 \item[(ii)] if $0<|\deri^{1}V^{\varepsilon_{n}}_{\ell_{i}^{*}}(X^{\varepsilon_{n},*}_{\tmt})|^{2}-g_{\ell_{i}^{*}}(X^{\varepsilon_{n},*}_{\tmt})^{2}<2\varepsilon_{n}$, the process $X^{\varepsilon_{n},*}_{\tmt}$ will be crossing  $\partial\mathcal{C}_{\ell_{i}^{*}}$ persistently; 
 \item[(iii)]otherwise, $\xi^{\varepsilon_{n},*}_{\tmt}=(\mathbb{n}^{\varepsilon_{n},*}_{\tmt},\zeta^{\varepsilon_{n},*}_{\tmt})$   will exercise  a force $\frac{2}{\varepsilon_{n}}|\deri^{1}V^{\varepsilon_{n}}_{\ell_{i}^{*}}(X^{\varepsilon_{n},*}_{\tmt})|$ and in the direction $-\frac{\deri^{1}V^{\varepsilon_{n}}_{\ell_{i}^{*}}(X^{\varepsilon_{n},*}_{\tmt})}{|\deri^{1}V^{\varepsilon_{n}}_{\ell_{i}^{*}}(X^{\varepsilon_{n},*}_{\tmt})|}$ at $X^{\varepsilon_{n},*}_{\tmt}$ in such a way that it will be pushed back to  $\partial\mathcal{C}_{\ell_{i}^{*}}$. 
 
 \item[(iv)] At time $\tmt=\tau_{i+1}^{*}<\tau^{*}$,   $X^{\varepsilon_{n},*}_{\tau_{i+1}^{*}}$ will be for the first time in $\mathcal{S}^{\varepsilon_{n}}_{\ell_{i}^{*}}$ and will switch to the regime $\ell_{i+1}^{*}$. 
 \end{enumerate}
 This procedure will be repeated until the time $\tau^{*}$ which represents the first exit time of the process $X^{\varepsilon_{n},*}$ from the set $\set$.

\section{Conclusions  and some further work}\label{conc}
Under Assumptions \eqref{h3}--\eqref{h2},  the existence and uniqueness for the strong  solutions to the HJB  equations \eqref{esd5} and \eqref{pc1} were guaranteed on the space $\hol^{0,1}(\overline{\set})\cap\sob^{2,\infty}_{\loc}(\set)$. After that, it was proven that the value function $V^{\varepsilon}$  for the penalized absolutely continuous/switching control problem seen in Subsection \ref{penal1} satisfies \eqref{pc1}, showing  that for  $\ell\in\mathbb{I}$,   $V^{\varepsilon}_{\ell}\in\hol^{0,1}(\overline{\set})\cap\sob^{2,\infty}_{\loc}(\set)$. Finally, assuming also that the domain set $\set$ is an open convex set and by probabilistic arguments, it was verified that the value function $V$, given in \eqref{vf1}, is characterized as a    limit of $V^{\varepsilon}$ as $\varepsilon\downarrow0$ which permitted  to conclude that  $V$ satisfies \eqref{esd5} and from here, to see that for $\ell\in\mathbb{I}$,  $V_{\ell}\in\hol^{0,1}(\overline{\set})\cap\sob^{2,\infty}_{\loc}(\set)$.

 Although the optimal control process for the mixed singular/switching  stochastic control  problem   \eqref{vf1} was not given explicitly, and this is still an open problem, we constructed a family of $\varepsilon$-penalized absolutely continuous/switching optimal control processes $\{(\xi^{\varepsilon_{n},*},\varsigma^{\varepsilon_{n},*})\}_{n\geq1}$; see \eqref{opt1}--\eqref{opt10}, such that the limit of their value functions $V^{\varepsilon_{n}}$ (as $\varepsilon_{n}\rightarrow0$) agrees with the value function $V$.

There are some extensions to be considered and directions for future research. Some of them could be:

\begin{enumerate}
	\item  To study the value function $V$ given in \eqref{vf1} when the infinitesimal generator of the process $ X^{\xi,\varsigma}$, without the influence of the singular control $\xi$ in $X^{\xi,\varsigma}$, within the regime $\ell$, is given by 
	\begin{align*}
	\dif_{\ell} u_{\ell} &=\tr[a_{\ell} \deri^{2}u_{\ell} ]-\langle b _{\ell} ,\deri^{1}u_{\ell} \rangle\\
	&\quad+\int_{\R^{d}_{*}}[u_{\ell}(x+z)-u_{\ell}(x)-\langle\deri^{1}u_{\ell},z\rangle\uno_{\{|z|\in(0,1)\}}]s_{\ell}(x,z)\nu(\der z),
	\end{align*} 
	 where $\nu$ is a Radon measure on $\R^{d}_{*}\eqdef \R^{d}\setminus\{0\}$  {satisfying} $\int_{\R^{d}_{*}}[|z|^{2}\wedge 1]\nu(\der z)< \infty$, and  $s_{\ell}:\overline{\set}\times\R^{d}\longrightarrow[0,1]$ is such that $\int_{\R^{d}_{*}}s_{\ell}(x,z)\uno_{\{x+z\notin\set\}}\nu(\der z)<\infty$  for $x\in\set$. In this case, the main difficulty lies in obtaining some  \`a priori estimates  of $\int_{\{|z|\in(0,1)\}}\big[\int_{0}^{1} |\deri^{2}u^{\varepsilon}(\cdot+tz)|\der t\big]|z|^{2}s_{\ell}(\cdot,z)\nu(\der z)$ that are independent of $\varepsilon$.
	
	\item To analyze the problem given in \eqref{es1}--\eqref{vf1.0} on the whole space $\R^{d}$. In this direction, a technically involved problem could be studied when the controlled process can switch between $m$-regimes which are governed  by Brownian motions with different drifts.  
\end{enumerate}

\appendix
\section{Appendix}

 {Recall that Lemma \ref{Lb1} and Proposition  \ref{princ1.0} are under assumptions \eqref{h3}--\eqref{h2}.} Let us first show \eqref{ap1} and \eqref{ap2} of Lemma \ref{Lb1}, which helps to verify the existence of the classic solution $u^{\varepsilon, \delta}$  to  the NPDS \eqref{NPD.1}.  Afterwards,  to  finalize the proof of Proposition  \ref{princ1.0}, we shall prove \eqref{ap3} of Lemma \ref{Lb1}.  {From now on, for simplicity of notation, we replace $u^{\varepsilon,\delta}$ by $u$ in the proofs.} 

\subsection{Verification of \eqref{ap1} and \eqref{ap2} }\label{a1}
\begin{proof}[Proof of Lemma \ref{Lb1}. Eq. \eqref{ap1}]
	Let $(x_{0},\ell_{0})\in\overline{\set}\times\mathbb{I}$ be such that $u_{\ell_{0}}(x_{0})=\min_{x\in \overline{\set},\ell\in\mathbb{I}}u_{\ell}(x)$. If $x_{0}\in\partial\set$, it follows easily that $u_{\ell}(x)\geq u_{\ell_{0}}(x_{0})=0$ for all $(x,\ell)\in\overline{\set}\times\mathbb{I}$. Suppose that $x_{0}\in\set$. Since $u_{\ell_{0}}=0$ in $\partial\set$, we have that $u_{\ell_{0}}(x_{0})\leq0$. On the other hand, we know that 
	\begin{equation}\label{eq.1}
	\begin{split}
	&\deri^{1} u_{\ell_{0}}(x_{0})=0,\quad  \tr[a_{\ell_{0}}(x_{0})\deri^{2}u_{\ell_{0}}(x_{0})]\geq0,\\ &u_{\ell_{0}}(x_{0})-u_{\kappa}(x_{0})\leq 0,\ \text{for}\ \kappa\in\mathbb{I}\setminus\{\ell_{0}\}. 
	\end{split}
	\end{equation}
	Then, using \eqref{a4}, \eqref{NPD.1} and \eqref{eq.1}, $0\leq\tr[a_{\ell_{0}}\deri^{2}u_{\ell_{0}}]= c_{\ell_{0}}u_{\ell_{0}}-h_{\ell_{0}}\leq c_{\ell_{0}}u_{\ell_{0}},\ \text{at}\ x_{0}$.
	Since $c_{\ell_{0}}>0$ on $\overline{\set}$, it follows that $u_{\ell_{0}}(x_{0})\geq0$. Therefore $u_{\ell}(x)\geq u_{\ell_{0}}(x_{0})=0$, for all $x\in\overline{\set}$ and $\ell\in\mathbb{I}$. 
	
	For each $\ell\in\mathbb{I}$, consider $\tilde{v}_{\ell}$  as the unique solution to the  Dirichlet problem $[c_{\ell}-\dif_{\ell}] \tilde{v}_{\ell}=h_{\ell}$  in $\set$, s.t $\tilde{v}_{\ell}=0$, on $\partial\set$. By Theorem 1.2.10 of \cite{gilb}, it is well know that   $\tilde{v}_{\ell}\in\hol^{2,\alpha'}(\overline{\set})$  and $||\tilde{v}_{\ell}||_{\hol^{2,\alpha'}(\overline{\set})}\leq K_{2} ||h_{\ell}||_{\hol^{0,\alpha'}(\overline{\set})}\leq K_{2}\Lambda\defeq C_{1}, $ 
	where $K_{2}=K_{2}(d,\Lambda ,\alpha')$,  since \eqref{h0}--\eqref{h2} hold.  Meanwhile, from Eq. \eqref{NPD.1}, it can been seen  that for each $\ell\in\mathbb{I}$, $ [c_{\ell}-\dif_{\ell}]u_{\ell}\leq h_{\ell}$ in $\set$. Then, taking $\eta_{\ell}\eqdef u_{\ell}-\tilde{v}_{\ell}$, we get for each $\ell\in\mathbb{I}$,
	\begin{equation}\label{sup1.1}
	[c_{\ell}-\dif_{\ell}]\eta_{\ell}\leq0,\ \text{in}\ \set,\quad\text{s.t.}\ 
	\eta_{\ell}=0,\ \text{on}\ \partial\set.
	\end{equation}
	Let $x^{*}_{\ell}\in\overline{\set}$ be a maximum point of $\eta_{\ell}$. If $x^{*}_{\ell}\in\partial{\set}$, trivially we have $u_{\ell}-\tilde{v}_{\ell}\leq0$ in $\overline{\set}$. Suppose that $x^{*}_{\ell}\in \set$. Note $\eta_{\ell}(x^{*}_{\ell})\geq0$ and 
	\begin{equation}\label{sup2}
	\deri^{1}\eta_{\ell}(x^{*}_{\ell})=0,\quad \tr[a_{\ell}(x^{*}_{\ell})\deri^{2}\eta_{\ell}(x^{*}_{\ell})]\leq 0.
	\end{equation}
	Then, using \eqref{sup1.1}--\eqref{sup2},
	$0\geq\tr[a_{\ell}\deri^{2}\eta_{\ell}]\geq c_{\ell}\eta_{\ell}$, at $x^{*}_{\ell}$. Hence,  {$\eta_{\ell}\leq0$ at $x^{*}_{\ell}$, since $c_{\ell}>0$ on $\overline{\set}$}. We conclude $[u_{\ell}-v_{\ell}](x)\leq[u_{\ell}-v_{\ell}](x^{*}_{\ell})=0$ for $x\in\overline{\set}$. Therefore, for each $\ell\in\mathbb{I}$,   $0\leq u_{\ell}\leq C_{1}$ on $\overline{\set}$. 
\end{proof}

\begin{proof}[Proof of Lemma \ref{Lb1}. Eq. \eqref{ap2}]
	For each $\ell\in\mathbb{I}$, consider the auxiliary function $w_{\ell}\eqdef|\deri^{1}u_{\ell}|^{2}-\lambda  A_{\varepsilon,\delta}u_{\ell}$, on $\overline{\set}$, where   $A_{\varepsilon,\delta}\eqdef\max_{(x,\ell)\in\overline{\set}\times\mathbb{I}}|\deri^{1}u_{\ell}(x)|$ and $\lambda\geq1$ a constant that shall be selected later on. Observe that if $  A_{\varepsilon,\delta}\leq1$, we obtain a bound for $  A_{\varepsilon,\delta}$ that is independent of  $\varepsilon,\delta$. Hence, we obtain the statement given in the lemma above.  We assume henceforth that $  A_{\varepsilon,\delta}>1$.  Taking first and second derivatives to $w_{\ell}$, it can be checked that
	\begin{align}\label{dercot1}
	-\tr[a_{\ell}\deri^{2}w_{\ell}]&=-2\sum_{i}\langle a_{\ell}\deri^{1}\partial_{i}u_{\ell},\deri^{1}\partial_{i}u_{\ell}\rangle\notag\\
	&\quad-2\sum_{i}\tr [a_{\ell}\deri^{2}\partial_{i}u_{\ell}]\partial_{i}u_{\ell}+\lambda  A_{\varepsilon,\delta}\tr[a_{\ell}\deri^{2}u_{\ell}].
	\end{align}	 
	Meanwhile, from \eqref{eq2} and \eqref{NPD.1},
	\begin{equation}
	\lambda  A_{\varepsilon,\delta}\tr[a_{\ell}\deri^{2} u_{\ell}]=  \lambda  A_{\varepsilon,\delta}\bigg[\widetilde{D}_{1}u_{\ell}+\psi_{\varepsilon,\ell}(\cdot)+\displaystyle\sum_{\kappa\in\mathbb{I}\setminus\{\ell\}}\psi_{\delta,\ell,\kappa}(\cdot) \bigg],
	\end{equation}
	where $\psi_{\varepsilon,\ell}(\cdot)$, $\psi_{\delta,\ell,\kappa}(\cdot)$  denote $\psi_{\varepsilon}(|\deri^{1}u_{\ell}|^{2}-g_{\ell}^{2})$,  $\psi_{\delta}(u_{\ell}-u_{\kappa} -\vartheta_{\ell,\kappa})$, respectively, and $\widetilde{D}_{1}u_{\ell}\eqdef\langle b_{\ell},\deri^{1}u_{\ell}\rangle+c_{\ell}u_{\ell}-h_{\ell}$. Now, differentiating \eqref{NPD.1}, multiplying by $2\partial_{i}u$ and taking summation over all $i$'s, we see that
	\begin{align}\label{dercot2}
	-2\sum_{i}\tr[a_{\ell}\deri^{2}\partial_{i} u_{\ell}]\partial_{i} u_{\ell}
	&=\widetilde{D}_{2} u_{\ell} -2\psi'_{\varepsilon,\ell}(\cdot)\langle\deri^{1} u_{\ell},\deri^{1}[|\deri^{1} u_{\ell}|^{2}-g_{\ell}^{2}]\rangle\notag\\
	&\quad-2\sum_{\kappa\in\mathbb{I}\setminus\{\ell\}}\psi'_{\delta,\ell,\kappa}(\cdot)[|\deri^{1}u_{\ell}|^{2}-\langle\deri^{1}u_{\ell},\deri^{1}u_{\kappa} \rangle],
	\end{align}
	where  
	\begin{equation}\label{dercot2.1}
	\widetilde{D}_{2}u_{\ell}\eqdef 2\sum_{k}\partial_{k}u_{\ell}\tr[[\partial_{k}a_{\ell}]\deri^{2}u_{\ell}]
	-2\langle\deri^{1}u_{\ell},\deri^{1}[\langle b_{\ell},\deri^{1}u_{\ell}\rangle+c_{\ell}u_{\ell}-h_{\ell}]\rangle.
	\end{equation}
	Then, from \eqref{dercot1}--\eqref{dercot2}, it can be shown that 
	\begin{align*}
	-\tr[a_{\ell}\deri^{2}w_{\ell}]&=-2\sum_{i}\langle a_{\ell}\deri^{1}\partial_{i}u_{\ell},\deri^{1}\partial_{i}u_{\ell}\rangle+\widetilde{D}_{2}u_{\ell} +\lambda   A_{\varepsilon,\delta} \widetilde{D}_{1}u_{\ell}\notag\\
	&\quad-2\psi'_{\varepsilon,\ell}(\cdot)\langle\deri^{1}u_{\ell},\deri^{1}[|\deri^{1}u_{\ell}|^{2}-g_{\ell}^{2}]\rangle+\lambda   A_{\varepsilon,\delta}\psi_{\varepsilon,\ell}(\cdot)\notag\\
	&\quad-\sum_{\kappa\in\mathbb{I}\setminus\{\ell\}}\big\{2\psi'_{\delta,\ell,\kappa}(\cdot)[|\deri^{1}u_{\ell}|^{2}-\langle\deri^{1}u_{\ell},\deri^{1}u_{\kappa} \rangle]-\lambda   A_{\varepsilon,\delta}\psi_{\delta,\ell,\kappa}(\cdot)\big\},\quad \notag \text{on}\ \set.
	\end{align*}
	By \eqref{a4}, \eqref{h2} and \eqref{ap1}, notice that
	\begin{align}\label{dercot1.1}
	-2\sum_{i}&\langle a_{\ell}\deri^{1}\partial_{i}u_{\ell},\deri^{1}\partial_{i}u_{\ell}\rangle+\widetilde{D}_{2}u_{\ell} +\lambda   A_{\varepsilon,\delta} \widetilde{D}_{1}u_{\ell}\notag\\
	&\leq2\Lambda d|\deri^{1}u_{\ell}|^{2}-\theta|\deri^{2}u_{\ell}|^{2}+2\Lambda[1+d^{3}]|\deri^{1}u_{\ell}|\,|\deri^{2}u_{\ell}|\notag\\
	&\quad+2\Lambda[1+C_{1}]|\deri^{1}u_{\ell}|+\lambda A_{\varepsilon,\delta}\Lambda[|\deri^{1}u_{\ell}|+C_{1}]\notag\\
	&\leq\bigg[2\Lambda d+\frac{\Lambda^{2}[1+d^{3}]^{2}}{\theta}\bigg]|\deri^{1}u_{\ell}|^{2}+2\Lambda[1+C_{1}]|\deri^{1}u_{\ell}|+\lambda A_{\varepsilon,\delta}\Lambda[|\deri^{1}u_{\ell}|+C_{1}].
	\end{align} 
	Then, by \eqref{dercot1.1} and using  convexity property of   $\psi_{\cdot}$, i.e. $\psi_{\cdot}(r)\leq \psi'_{\cdot}(r)r$, for all $r\in\R$,  we see that  
	\begin{align}\label{partu4}
	&-\tr[a_{\ell}\deri^{2}w_{\ell}]\leq K_{3}|\deri^{1}u_{\ell}|^{2}+K_{3} [1+  \lambda A_{\varepsilon,\delta}]|\deri^{1}u_{\ell}|+\lambda K_{3}   A_{\varepsilon,\delta}\notag\\
	&\quad-\psi'_{\varepsilon,\ell}(\cdot)\big\{2\langle\deri^{1}u_{\ell},\deri^{1}|\deri^{1}u_{\ell}|^{2}\rangle-K_{3}|\deri^{1}u_{\ell}|-\lambda   A_{\varepsilon,\delta} [|\deri^{1}u_{\ell}|^{2}- g_{\ell}^{2}]\big\}\notag\\
	&\quad-\sum_{\kappa\in\mathbb{I}\setminus\{\ell\}}\psi'_{\delta,\ell,\kappa}(\cdot)\big\{2[|\deri^{1}u_{\ell}|^{2}-\langle\deri^{1}u_{\ell},\deri^{1}u_{\kappa} \rangle]-\lambda   A_{\varepsilon,\delta}[u_{\ell}-u_{\kappa} -\vartheta_{\ell,\kappa}]\big\},\quad\text{on $\set$,}
	\end{align}
	for some $K_{3}=K_{3}(d,\Lambda ,\alpha')$. Let $(x_{\lambda},\ell_{\lambda})\in\overline{\set}\times\mathbb{I}$  {(depending on $\lambda$)} be such that $w_{\ell_{\lambda}}(x_{\lambda})=\max_{(x,\ell)\in\overline{\set}\times\mathbb{I}}w_{\ell}(x)$.  From here, \eqref{ap1} and by definition of $w_{\ell}$, we see that 
	\begin{align}\label{partu5}
	|\deri^{1}u_{\ell}(x)|^{2}%&\leq |\deri^{1}u_{\ell_{\lambda}}(x_{\lambda})|^{2}+\lambda   A_{\varepsilon,\delta}[u_{\ell}(x)-u_{\ell_{\lambda}}(x_{\lambda})]\notag\\
	&\leq|\deri^{1}u_{\ell_{\lambda}}(x_{\lambda})|^{2}+\lambda   A_{\varepsilon,\delta}C_{1},\quad  {\text{for}\  (x,\ell)\in \overline{\set}\times\mathbb{I}.}
	\end{align}
	So, it  suffices to bound $|\deri^{1}u_{\ell_{\lambda}}(x_{\lambda})|^{2}$ by a positive constant $C=C(d,\Lambda,\alpha')$.  If $x_{\lambda}\in\partial \set$,  by \eqref{ap1}, it can be verified that $|\deri^{1}u_{\ell_{\lambda}}|\leq d^{\frac{1}{2}}C_{1}$ on $\partial{\set}$. Then, from (\ref{partu5}), 
	\begin{equation}\label{e1}
	|\deri^{1}u_{\ell}(x)|^{2}\leq C_{1}[dC_{1}+\lambda   A_{\varepsilon,\delta}] ,\quad  {\text{for}\  (x,\ell)\in\overline{\set}\times\mathbb{I}.}
	\end{equation}
	On the other hand, observe that for each $\varrho>0$, there exists $x_{\varrho}\in\overline{\set}$ such that 
	\begin{equation}\label{e2}
	[  A_{\varepsilon,\delta}-\varrho]^{2}\leq|\deri^{1} u_{\ell}(x_{\varrho})|^{2}.
	\end{equation}
	Using \eqref{e1} in \eqref{e2} and letting $\varrho\rightarrow0$,  
	\begin{equation}\label{e3}
	A_{\varepsilon,\delta}^{2}\leq C_{1}[dC_{1}+\lambda   A_{\varepsilon,\delta}] .
	\end{equation}
	Multiplying by $1/A_{\varepsilon,\delta}$ in \eqref{e3} and since $A_{\varepsilon,\delta}>1$, we get 
	\begin{equation}\label{e4}
	|\deri^{1}u_{\ell}(x)|\leq   A_{\varepsilon,\delta}\leq dC_{1}^{2}+\lambda  C_{1},\quad  \text{for}\ (x,\ell)\in\overline{\set}\times\mathbb{I}.
	\end{equation}
	In this case, considering  $\lambda>1$ fixed and $C_{2}\eqdef C_{1}[dC_{1}+\lambda]$, we obtain the  result that is proposed in the lemma.  Let $x_{\lambda}$ be in $\set$. It is known that $\partial_{i}|\deri^{1}u_{\ell_{\lambda}}|^{2}-\lambda   A_{\varepsilon,\delta}\partial_{i}u_{\ell_{\lambda}}=0$ at $x_{\lambda}$. Then, 
	\begin{equation}\label{max.1}
	2\langle\deri^{1}u_{\ell_{\lambda}},\deri^{1}|\deri^{1}u_{\ell_{\lambda}}|^{2}\rangle=2\lambda   A_{\varepsilon,\delta}|\deri^{1}u_{\ell_{\lambda}}|^{2},\ \text{at}\   x_{\lambda}. 
	\end{equation}
	Also, since  $\vartheta_{\ell_{\lambda},\kappa}\geq0$ and   $w_{\ell_{\lambda}}(x_{\lambda})-w_{\kappa}(x_{\lambda} )\geq0$, for each $\kappa\in\mathbb{I}$, and  $|y_{1}|^{2}-|y_{2}|^{2}=2[|y_{1}|^{2}-\langle y_{1},y_{2}\rangle]-|y_{1}-y_{2}|^{2}\leq2[|y_{1}|^{2}-\langle y_{1},y_{2}\rangle]$, for $y_{1},y_{2}\in\R^{d},$ 
	it implies
	\begin{equation}\label{max.2}
	2[|\deri^{1}u_{\ell_{\lambda}}|^{2}-\langle \deri^{1}u_{\ell_{\lambda}},\deri^{1}u_{\kappa} \rangle]-\lambda   A_{\varepsilon,\delta}[u_{\ell_{\lambda}}-u_{\kappa} -\vartheta_{\ell_{\lambda},\kappa}]\geq0,\ \text{at}\ x_{\lambda}. 
	\end{equation}
	By \eqref{partu4}, \eqref{max.1},\eqref{max.2} and since $\tr[a_{\ell_{\lambda}}(x_{\lambda})\deri^{2}w_{\ell_{\lambda}}(x_{\lambda})]\leq0$, we have
	\begin{align}\label{partu4.1}
	0\leq-\tr[a_{\ell_{\lambda}}\deri^{2}w_{\ell_{\lambda}}]&\leq K_{3}|\deri^{1}u_{\ell_{\lambda}}|^{2}+K_{3} [1+  \lambda A_{\varepsilon,\delta}]|\deri^{1}u_{\ell_{\lambda}}|+\lambda K_{3}   A_{\varepsilon,\delta}\notag\\
	&\quad-\psi'_{\varepsilon,\ell_{\lambda}}(\cdot)[\lambda   A_{\varepsilon,\delta}|\deri^{1}u_{\ell_{\lambda}}|^{2}-K_{3}|\deri^{1}u_{\ell_{\lambda}}|+\lambda   A_{\varepsilon,\delta} g_{\ell_{\lambda}}^{2}],\ \text{at}\ x_{\lambda}.
	\end{align}
	On the other hand, notice that either $\psi'_{\varepsilon,\ell_{\lambda}}(\cdot)<\frac{1}{\varepsilon}$ or $\psi'_{\varepsilon,\ell_{\lambda}}(\cdot)=\frac{1}{\varepsilon}$  at $x_{\lambda}$. If $\psi'_{\varepsilon,\ell_{\lambda}}(\cdot)<\frac{1}{\varepsilon}$ at $x_{\lambda}$, by definition of $\psi_{\varepsilon}$, given in (\ref{p12.1}), it follows that  $|\deri^{2}u^{\varepsilon}(x_{\lambda})|^{2}-g_{\ell}(x_{\lambda})^{2}\leq2\varepsilon$. 
	It implies that 
	$|\deri^{1}u_{\ell_{\lambda}}(x_{\lambda})|^{2}\leq \Lambda^{2}+2$.
	Then, by (\ref{partu5}) and arguing as  in \eqref{e4}, we obtain $ |\deri^{1}u_{\ell}(x)|\leq   A_{\varepsilon,\delta}\leq 2+\Lambda^{2}+\lambda C_{1}$, for each $(x,\ell)\in\overline{\set}\times\mathbb{I}$. For $\lambda>1$ fixed, it is obtained the result given in the lemma. Now, assume that  $\psi'_{\varepsilon,\ell_{\lambda}}(\cdot)=\frac{1}{\varepsilon}$.  Then, taking $\lambda>\max\{1,K_{3}\}$ fixed, using (\ref{partu4.1}) and proceeding of a similar way than \eqref{e4}, we get  
	$0\leq [K_{3}-\lambda]|\deri^{1}u_{\ell_{\lambda}}|^{2}+K_{3}[2+\lambda]|\deri^{1}u_{\ell_{\lambda}}|+\lambda K_{3},\ \text{at}\ x_{\lambda}$.  From here,  it yields that  $|\deri^{1}u_{\ell_{\lambda}}(x_{\lambda})|<K_{4}$, for some $K_{4}=K_{4}(d,\Lambda ,\alpha')$. Using (\ref{partu5}) and by similar arguments that (\ref{e3}), we conclude that there exists $C_{2}=C_{2}(d,\Lambda ,\alpha')$ such that   $|\deri^{1}u_{\ell}|\leq   A_{\varepsilon,\delta}\leq C_{2}$ on $\overline{\set}$.
\end{proof}

\subsection{Existence and uniqueness  of the solution to the NPDS}%\label{a2}

Let  $\mathcal{C}^{k}$, $\mathcal{C}^{k,\alpha'}$ be the sets given by $(\hol^{k}(\overline{\set}))^{m}$, $(\hol^{k,\alpha'}(\overline{\set}))^{m}$, respectively, with $k\in\N$ and $\alpha'\in(0,1)$. Defining $||\w||_{\mathcal{C}^{k,\alpha'}}=\max_{i\in\mathbb{I}}\{||\w_{i}||_{\hol^{k,\alpha'}(\overline{\set})}\}$ for each $\w=(\w_{1},\dots,\w_{m})\in\mathcal{C}^{k,\alpha'}$, the reader can verify that $||\cdot||_{\mathcal{C}^{k,\alpha'}}$ is a norm on $\mathcal{C}^{k,\alpha'}$ and  $(\mathcal{C}^{k,\alpha'},||\cdot||_{\mathcal{C}^{k,\alpha'}})$ is a Banach space. 

Notice that for each $\w\in\mathcal{C}^{1,\alpha'}$ fixed, there exists a unique solution $\sol\in\mathcal{C}^{2,\alpha'}$ to the following Linear Dirichlet partial system (LPDS)  
\begin{equation}\label{LPD.1}
\begin{split}
[c_{\ell}-\dif_{\ell}]  \sol_{\ell}= h_{\ell}- \psi_{\varepsilon}(|\deri^{1} \w_{\ell}|^{2}- g_{\ell}^{2})-\sum_{\kappa\in\mathbb{I}\setminus\{\ell\}}\psi_{\delta}(\w_{\ell}-\w_{\kappa} -\vartheta_{\ell,\kappa})  ,\ &\text{in}\ \set,\\
\text{s.t.}\  
\sol_{\ell}=0,\ &\text{on}\ \partial\set ,
\end{split}
\end{equation}
since \eqref{h0}--\eqref{h2} hold and $h_{\ell}- \psi_{\varepsilon}(|\deri^{1} \w_{\ell}|^{2}- g_{\ell}^{2})-\sum_{\kappa\in\mathbb{I}\setminus\{\ell\}}\psi_{\delta}(\w_{\ell}-\w_{\kappa} -\vartheta_{\ell,\kappa})\in\hol^{0,\alpha'}(\overline{\set})$ (see \cite[Thm. 6.14, p. 107 ]{gilb}). Additionally, by   \cite[Eq. (6.73), p. 126]{garroni} (or \cite[Thm. 1.2.10]{garroni}), it is known 
\begin{align}\label{hol1}
||\sol_{\ell}||_{\hol^{2,\alpha'}(\overline{\set})}&\leq K_{5}\bigg[||h_{\ell}||_{\hol^{0,\alpha'}(\overline{\set})}+||\psi_{\varepsilon}(|\deri^{1} \w_{\ell}|^{2}-g_{\ell}^{2})||_{\hol^{0,\alpha'}(\overline{\set})}\notag\\
&\quad+\sum_{\kappa\in\mathbb{I}\setminus\{\ell\}}||\psi_{\delta}(\w_{\ell}-\w_{\kappa} -\vartheta_{\ell,\kappa})||_{\hol^{0,\alpha'}(\overline{\set})}\bigg],\quad \text{for}\ \ell\in\mathbb{I},
\end{align}
for some  $K_{5}=K_{5}(d,\Lambda,\alpha')$. Also, using \cite[Thm. 4.12, p. 85]{adams} and \cite[Thm. 1.2.19]{garroni}, it follows that 
\begin{align}\label{elpd1}
||\sol_{\ell}||_{\hol^{1,\alpha'}(\overline{\set})}&\leq K_{6}\bigg[||h_{\ell}||_{\Lp^{p'}(\set)}+||\psi_{\varepsilon}(|\deri^{1} \w_{\ell}|^{2}-g_{\ell}^{2})||_{\Lp^{p'}(\set)}\notag\\
&\quad+\sum_{\kappa\in\mathbb{I}\setminus\{\ell\}}||\psi_{\delta}(\w_{\ell}-\w_{\kappa} -\vartheta_{\ell,\kappa})||_{\Lp^{p'}(\set)}\bigg],\quad \text{for}\ \ell\in\mathbb{I},
\end{align}
for some  $K_{6}=K_{6}(d,\Lambda,\alpha')$, where $p'\in(d,\infty)$ is such that $\alpha'=1-d/p'$.  We proceed to give the proof of Proposition \ref{princ1.0}. 
\begin{proof}[Proof of Proposition \ref{princ1.0}. Existence.]
	Define the mapping $T:(\mathcal{C}^{1,\alpha'},||\cdot||_{\mathcal{C}^{1,\alpha'}})\longrightarrow(\mathcal{C}^{1,\alpha'},||\cdot||_{\mathcal{C}^{1,\alpha'}})$ as $T[\w]=\sol$ for each $\w\in\mathcal{C}^{1,\alpha'}$, where  $\sol$ is the unique solution to the LDPS \eqref{LPD.1}. To use Schaefer's fixed point theorem (see i.e. \cite[Thm. 4 p. 539]{evans2}), we only need to verify that: (i) the mapping $T$ is continuous and compact; (ii) and the set $\mathcal{A}\eqdef\{\w\in\mathcal{C}^{1,\alpha'}:\w=\varrho T[\w],\ \text{for some }\ \varrho\in[0,1]\}$ is bounded uniformly, i.e. $||\w||_{\mathcal{C}^{1,\alpha'}}\leq C$, for each $\w\in\mathcal{A}$, where $C$ is some positive constant that is independent of $\w$ and $\varrho$.

	To verify the first item above, notice that, by \eqref{hol1}, $T$ maps bounded sets in $\mathcal{C}^{1,\alpha'}$ into bounded sets in $\mathcal{C}^{2,\alpha'}$ that are precompact in $\mathcal{C}^{2}$ and $\mathcal{C}^{1,\alpha'}$. With this remark and by the uniqueness of the solution to the LPDS \eqref{LPD.1}, the reader can verify  $T$ is a continuous and compact mapping from $\mathcal{C}^{1,\alpha'}$ to $\mathcal{C}^{2,\alpha'}\subset\mathcal{C}^{1,\alpha'}$. Let us prove the second item above. Consider now $\w\in \mathcal{A}$. Observe that if $\varrho=0$, it follows immediately that $\w\equiv\overline{0}\in\mathcal{C}^{1,\alpha'}$, where $\overline{0}$ is the null function. Assume that $w\in\mathcal{C}^{1,\alpha'}$ is such that $T[\w]=\frac{1}{\varrho}\w=(\frac{1}{\varrho}\w_{1},\dots,\frac{1}{\varrho}\w_{m})$ for some $\varrho\in(0,1]$, or, in other words, $\w\in\mathcal{C}^{2,\alpha'}$  and
	\begin{equation*}%\label{LPDw.1}
	[c_{\ell}-\dif_{\ell}]  \w_{\ell}= f_{\ell}  ,\ \text{in}\ \set,\quad
	\text{s.t.}\  
	\w_{\ell}=0,\  \text{on}\ \partial\set , \quad \text{for}\ \ell\in\mathbb{I},
	\end{equation*}  
	where $f_{\ell}\eqdef\varrho\big[h_{\ell}- \psi_{\varepsilon}(|\deri^{1} \w_{\ell}|^{2}- g_{\ell}^{2})-\sum_{\kappa\in\mathbb{I}\setminus\{\ell\}}\psi_{\delta}(\w_{\ell}-\w_{\kappa} -\vartheta_{\ell,\kappa})\big]$. By Theorems 6.14 and 9.19 of \cite{gilb}, it yields  that $\w\in\mathcal{C}^{3,\alpha'}$, since \eqref{h0}--\eqref{h2} hold and $f_{\ell}\in\hol^{1,\alpha'}(\overline{\set})$. Then, $\mathcal{A}\subset\mathcal{C}^{3,\alpha'}$. Applying the same arguments seen in proofs of \eqref{ap1} and \eqref{ap2} (Subsection \ref{a1}), it can be verified that the components of the function vector $\w=(\w_{1},\dots,\w_{m})\in\mathcal{A}$ satisfy
	\begin{equation}\label{enw1}
	0\leq\w_{\ell}\leq C_{1}\ \text{and}\ |\deri^{1}\w_{\ell}|\leq C_{2},\quad \text{on}\ \overline{\set},\ \text{for }\ \ell\in\mathbb{I},
	\end{equation}
	where $C_{1},C_{2}$ are positive constant as in Lemma  \ref{Lb1}. Notice that these constants are independent of $\w$ and $\varrho$.  Applying \eqref{enw1} in \eqref{elpd1}, it follows that $\mathcal{A}$, is bounded uniformly. with that, we see  that the items above, (i) and (ii), are true and by Schaefer's fixed point theorem, there exists a fixed point $u=(u_{1},\dots,u_{m})\in\mathcal{C}^{1,\alpha'}$ to the problem $T[u]=u$ which satisfies the NPDS \eqref{NPD.1}.   {In addition}, we have $u=T[u]\in\mathcal{C}^{2,\alpha'}$ and by  similar arguments seen previously, it can be shown that  $u$ is non-negative and belongs to $\mathcal{C}^{3,\alpha'}(\overline{\set})$. Again, repeating the same arguments seen above, we can  conclude that $u\in\mathcal{C}^{4,\alpha'}(\overline{\set})$.  
\end{proof} 
\begin{proof}[Proof of Proposition \ref{princ1.0}. Uniqueness.] The uniqueness of the solution $u$ to the NPDS \eqref{NPD.1} is obtained by contradiction. Assume that there are two solutions $u,v\in\mathcal{C}^{4,\alpha'}$ to the NPDS \eqref{NPD.1}. Let $\nu=(\nu_{1},\dots,\nu_{m})\in\mathcal{C}^{4,\alpha'}$ such that $\nu_{\ell}\eqdef u_{\ell}-v_{\ell}$ for $\ell\in\mathbb{I}$. Let $(x_{\circ},\ell_{\circ})$ be in  $\overline{\set}\times\mathbb{I}$ such that $\nu_{\ell_{\circ}}(x_{\circ})=\max_{(x,\ell)\in\overline{\set}\times\mathbb{I}}\nu_{\ell}(x)$.  If $x_{ \circ}\in\partial{\set}$, trivially we have $u_{\ell}-v_{\ell}\leq0$ in $\overline{\set}$, for $\ell\in\mathbb{I}$. Suppose that $x_{\circ}\in \set$. Then,
	\begin{equation}\label{sup2.1}
	\begin{split}
	&\deri^{1}\nu_{\ell_{\circ}}(x_{\circ})=0,\quad\tr[a_{\ell_{\circ}}(x_{\circ})\deri^{2}\nu_{\ell}(x_{\circ})]\leq 0,\\
	&u_{\ell_{\circ}}(x_{\circ})-u_{\kappa}(x_{\circ} )\geq v_{\ell_{\circ}}(x_{\circ})-v_{\kappa}(x_{\circ} )\quad \text{for}\ \kappa\neq\ell_{\circ} .
	\end{split} 
	\end{equation}
	Then, from \eqref{NPD.1} and \eqref{sup2.1},
	\begin{align}\label{NPDu.1}
	0&\geq\tr[a_{\ell_{\circ}}\deri^{2}\nu_{\ell_{\circ}}]\notag\\
	&=c_{\ell_{\circ}}\nu_{\ell_{\circ}}+\sum_{\kappa\in\mathbb{I}\setminus\{\ell_{\circ}\}}\big\{\psi_{\delta}(u_{\ell_{\circ}}-u_{\kappa} -\vartheta_{\ell_{\circ},\kappa})-\psi_{\delta}(v_{\ell_{\circ},\kappa}-v_{\kappa} -\vartheta_{\ell_{\circ},\kappa})\big\}\geq c_{\ell_{\circ}}\nu_{\ell_{\circ}}\quad \text{at}\ x_{\circ}
	\end{align}
	because of
	$0\leq\psi_{\delta}(u_{\ell_{\circ}}-u_{\kappa}-\vartheta_{\ell_{\circ},\kappa})
	-\psi_{\delta}(v_{\ell_{\circ}}-v_{\kappa}-\vartheta_{\ell_{\circ},\kappa})$ at $x_{\circ}$, for $\kappa\in\mathbb{I}$. 
	From \eqref{NPDu.1} and since $c_{\ell_{\circ}}>0$, we have that $u_{\ell}(x)-v_{\ell}(x)\leq u_{\ell_{\circ}}(x_{\circ})-v_{\ell_{\circ}}(x_{\circ})\leq0$ for $(x,\ell)\in\overline{\set}\times\mathbb{I}$. Taking now $\nu\eqdef v-u$ and proceeding the same way than before, it follows immediately that $v -u\geq0$.  Therefore $u =v $ and from here we conclude that the NPDS \eqref{NPD.1} has a unique solution $u$, whose components belong to $\hol^{4,\alpha'}(\overline{\set})$.
\end{proof}

\subsection{Verification of \eqref{ap3}}%\label{a3}
Let us define the auxiliary function $\phi_{\ell}$  as
	\begin{equation}\label{D2.1}
	\phi_{\ell}\eqdef\omega^{2}|\deri^{2}u_{\ell}|^{2}+\lambda A^{1}_{\varepsilon,\delta}\omega\tr[\alpha_{\ell_{0}}\deri^{2}u_{\ell}]+\mu|\deri^{1}u_{\ell}|^{2}\ \text{ on}\ \set,
	\end{equation}
	with $A^{1}_{\varepsilon,\delta}\eqdef\max_{(x,\ell)\in\overline{\set}\times\mathbb{I}}\omega(x)|\deri^{2}u_{\ell}(x)|$, $\lambda\geq\max\{1,2/\theta\}$, $\mu\geq1$ fixed, and $\alpha_{\ell_{0}}=(\alpha_{\ell_{0}\,ij})_{d\times d}$ be such that $\alpha_{\ell_{0}\,ij}\eqdef  a_{\ell_{0}\,ij}(x_{0})$, where $(x_{0},\ell_{0})\in\overline{\set}\times\mathbb{I}$ fixed. Recall that $\omega$ is a cut-off function as in Remark \ref{R1}.  We shall show that $\phi_{\ell}$ satisfies \eqref{d7}. In particular, \eqref{d7} holds when $\phi_{\ell}$ is evaluated at its maximum  $x_{\mu}\in\set$, which helps to see that \eqref{ap3} is true. 
\begin{lema}\label{D2.0.0}
	Let $\phi_{\ell}$ be the auxiliary function given by \eqref{D2.1}. 
	Then, there exists a positive constant $C_{7}=C_{7}(d,\Lambda ,\alpha',K_{1})$ such that on $(x,\ell)\in\overline{\set}\times\mathbb{I}$,
	\begin{align}\label{d7}
	&\omega^{2}\tr[a_{\ell}\deri^{2}\phi_{\ell}]
	\geq 2\theta[\omega^{4}|\deri^{3}u_{\ell}|^{2}+\mu\omega^{2}|\deri^{2}u_{\ell}|^{2}]-2\lambda C_{7}A^{1}_{\varepsilon,\delta}\omega^{2}|\deri^{3}u_{\ell}|\notag\\
	&-\lambda C_{7}[A^{1}_{\varepsilon,\delta}]^{2}-C_{7}(\lambda+\mu)A^{1}_{\varepsilon,\delta}-C_{7}\mu+\omega^{2}\sum_{\kappa\in\mathbb{I}\setminus\{\ell\}}\psi'_{\delta,\ell,\kappa}(\cdot)[\phi_{\ell}-\phi_{\kappa} ]\notag\\
	&+A^{1}_{\varepsilon,\delta}\omega^{2}\psi'_{\varepsilon,\ell}(\cdot)\bigg\{2\omega[\lambda\theta-2]|\deri^{2}u_{\ell}|^{2}-2\lambda C_{7}|\deri^{2}u_{\ell}|-(\lambda+\mu)C_{7}+\frac{2}{A^{1}_{\varepsilon,\delta}}\langle\deri^{1}u_{\ell},\phi_{\ell}\rangle\bigg\}.
	\end{align}
	
\end{lema}

Before providing  the verification of the lemma above, let us first proceed to proving  \eqref{ap3}.
\begin{proof}[Proof of Lemma \ref{Lb1}. Eq. \eqref{ap3}]
Let $\phi_{\ell}$ be as in \eqref{D2.1}, where $\lambda\geq\max\{1,2/\theta\}$ is fixed and  $\mu\geq1$ will be determined later on, and $(x_{0}, \ell_{0})\in\overline{\set}\times\mathbb{I}$ satisfies 
	\begin{equation}\label{D2.0}
	\omega(x_{0})|\deri^{2}u_{\ell_{0}}(x_{0})|=A^{1}_{\varepsilon,\delta}=\max_{(x,\ell)\in\overline{\set}\times\mathbb{I}}\omega(x)|\deri^{2}u_{\ell}(x)|.
	\end{equation}
	Notice that 
	if $x_{0}\in\overline{\set}\setminus B_{\beta' r}$, by  Remark \ref{R1} and \eqref{D2.0}, we obtain $\partial_{ij}u_{\ell}(x)\equiv0$, for each $(x,\ell)\times\overline{\set}\times\mathbb{I}.$ From here, \eqref{ap3} is trivially true. So, assume that $x_{0}$ is in $B_{\beta'r}$. Without loss of generality we also assume that $A^{1}_{\varepsilon,\delta}>1$, since if $A^{1}_{\varepsilon,\delta}\leq1$, we get that  $\omega(x) |\deri^{2}u_{\ell}(x)|\leq A^{1}_{\varepsilon,\delta}\leq1$ for $(x,\ell)\in\overline{\set}\times\mathbb{I}$.  Taking $C_{3}=1$, we obtain the result in {\eqref{ap3}}.   Let $(x_{\mu},\ell_{\mu})\in\overline{\set}\times\mathbb{I}$ be such that $\phi_{\ell_{\mu}}(x_{\mu})=\max_{(x,\ell)\in\overline{\set}\times\mathbb{I}}\phi_{\ell}(x)$. If $x_{\mu}\in\overline{\set}\setminus B_{\beta'r}$,  from \eqref{ap2} and \eqref{D2.1}, it follows that
	\begin{equation}\label{D2.2}
	\omega^{2}|\deri^{2}u_{\ell}|^{2}\leq-\lambda A^{1}_{\varepsilon,\delta}\omega\tr[\alpha_{\ell_{0}}\deri^{2}u_{\ell}]+\mu C_{2}^{2},\ \text{for}\ (x,\ell)\in\overline{\set}\times\mathbb{I}.
	\end{equation}
	Evaluating $(x_{0},l_{0})$ in \eqref{D2.2} and by \eqref{eq2}, \eqref{a4}, \eqref{p12.1}, \eqref{NPD.1} and \eqref{ap2}, it can be verified that  $[A_{\varepsilon,\delta}^{1}]^{2}\leq\lambda\Lambda[1+C_{2}] A_{\varepsilon,\delta}^{1}+\mu C_{2}^{2}$. From here  and due to $A^{1}_{\varepsilon,\delta}>1$, we conclude that $\omega(x)|\deri^{2}u_{\ell}(x)|\leq A_{\varepsilon,\delta}^{1}\leq \lambda\Lambda[1+C_{2}]+\mu C_{2}^{2}=:C_{3},\quad \text{for $(x,\ell)\in\overline{\set}\times\mathbb{I}$.}$   From now, assume that $x_{\mu}\in B_{\beta'r}$. Then, 
	\begin{equation}\label{D2.3}
	\begin{split}
	&\deri^{1}\phi_{\ell_{\mu}}(x_{\mu})=0, \quad \tr[a_{\ell_{\mu}}(x_{\mu})\deri^{2}\phi_{\ell_{\mu}}(x_{\mu})]\leq0.%\\
	%&\phi_{\ell_{\mu}}(x_{\mu})\geq\phi_{\ell_{\mu}}(x_{\mu}-k_{\ell_{\mu},\kappa}) \quad\text{for}\ \kappa\in\mathbb{I}\setminus\{\ell_{\mu}\}.
	\end{split}
	\end{equation}
	 Noting that $2\theta\omega^{4}|\deri^{3}u_{\ell}|^{2}-2\lambda C_{7}A^{1}_{\varepsilon,\delta}\omega^{2}|\deri^{3}u_{\ell}|\geq-\frac{\lambda^{2}C^{2}_{7}}{\theta}[A^{1}_{\varepsilon,\delta}]^{2}$, with $C_{7}>0$ {as} in Lemma \ref{D2.0.0},  and using \eqref{d7} and \eqref{D2.3}, it yields that
	\begin{align*}%\label{d7.0}
	0\geq&2\theta\mu\omega^{2}|\deri^{2}u_{\ell_{\mu}}|^{2}-\lambda^{2} C_{7}\bigg[1+\frac{C_{7}}{\theta}\bigg][A^{1}_{\varepsilon,\delta}]^{2}-C_{7}(\lambda+\mu)A^{1}_{\varepsilon,\delta}-C_{7}\mu\notag\\
	&+A^{1}_{\varepsilon,\delta}\omega^{2}\psi'_{\varepsilon,\ell}(\cdot)\{2\omega[\lambda\theta-2]|\deri^{2}u_{\ell_{\mu}}|^{2}-2\lambda C_{7}|\deri^{2}u_{\ell_{\mu}}|-(\lambda+\mu)C_{7}\},\quad \text{at}\ x_{\mu}.
	\end{align*}
	From here, we have that at least one of the next two inequalities is true:
	\begin{align}
	2\theta\mu\omega^{2}|\deri^{2}u_{\ell_{\mu}}|^{2}-\lambda^{2} C_{7}\bigg[1+\frac{C_{7}}{\theta}\bigg][A^{1}_{\varepsilon,\delta}]^{2}-C_{7}(\lambda+\mu)A^{1}_{\varepsilon,\delta}-C_{7}\mu&\leq0,\quad\text{at}\ x_{\mu},\label{d7.1}\\
	A^{1}_{\varepsilon,\delta}\omega^{2}\psi'_{\varepsilon,\ell}(\cdot)\{2\omega[\lambda\theta-2]|\deri^{2}u_{\ell_{\mu}}|^{2}-2\lambda C_{7}|\deri^{2}u_{\ell_{\mu}}|-(\lambda+\mu)C_{7}\}&\leq0,\quad\text{at}\ x_{\mu}.\label{d7.2}
	\end{align}
	Suppose that \eqref{d7.1} holds. Then, evaluating $(x_{\mu},\ell_{\mu})$ in \eqref{D2.1}, it follows
	\begin{align}\label{d.8}
	\phi_{\ell_{\mu}}%&\leq\omega^{2}|\deri^{2}u_{\ell_{\mu}}|^{2}+\lambda\Lambda A^{1}_{\varepsilon,\delta}\omega|\deri^{2}u_{\ell_{\mu}}|+\mu C^{2}_{2}\notag\\
	&\leq \frac{\lambda^{2} C_{7}}{2\theta\mu}\bigg[1+\frac{C_{7}}{\theta}\bigg][A^{1}_{\varepsilon,\delta}]^{2}+\frac{C_{7}(\lambda+\mu)}{2\theta\mu}A^{1}_{\varepsilon,\delta}+\frac{C_{7}}{2\theta}+\mu C^{2}_{2}\notag\\
	&\quad+\lambda\Lambda A^{1}_{\varepsilon,\delta}\bigg\{\frac{\lambda^{2} C_{7}}{2\theta\mu}\bigg[1+\frac{C_{7}}{\theta}\bigg][A^{1}_{\varepsilon,\delta}]^{2}+\frac{C_{7}(\lambda+\mu)}{2\theta\mu}A^{1}_{\varepsilon,\delta}+\frac{C_{7}}{2\theta}\bigg\}^{1/2},\quad\text{at}\ x_{\mu}.
	\end{align} 
	Meanwhile, evaluating $(x_{0},\ell_{0})$ in \eqref{D2.1} and  using \eqref{eq2} and \eqref{NPD.1}, we get 
	\begin{align}\label{d.9}
	\phi_{\ell_{0}}%=\omega^{2}|\deri^{2}u_{\ell_{0}}|^{2}+\lambda A^{1}_{\varepsilon,\delta}\omega\tr[\alpha_{\ell_{0}}\deri^{2}u_{\ell_{0}}]+\mu|\deri^{1}u_{\ell_{0}}|^{2}
	\geq[A^{1}_{\varepsilon,\delta}]^{2}-\lambda \Lambda A^{1}_{\varepsilon,\delta}[C_{2}+1],\quad\text{at}\ x_{0}.
	\end{align}
	Then, taking $\mu$ large enough such that $\frac{K_{7}^{(\lambda)}}{\mu}\leq\Big[\frac{1}{\lambda\Lambda}\Big[1-\frac{K_{7}^{(\lambda)}}{\mu}\Big]\Big]^{2}$, with $K_{7}^{(\lambda)}\eqdef\frac{\lambda^{2}C_{7}}{2\theta}\big[1+\frac{C_{7}}{\theta}\big]$, using \eqref{d.8}--\eqref{d.9} and since $\phi_{\ell_{0}}(x_{0})\leq\phi_{\ell_{\mu}}(x_{\mu})$ and $\lambda, A^{1}_{\varepsilon,\delta}>1$, we have that 
	\begin{align*}
	\frac{1}{\lambda \Lambda}\bigg[1-\frac{K_{7}^{(\lambda)}}{\mu}\bigg]A^{1}_{\varepsilon,\delta}-K_{8}^{(\mu)}
	\leq\bigg\{\frac{ K_{7}^{(\lambda)}}{\mu}[A^{1}_{\varepsilon,\delta}]^{2}+\frac{C_{7}(\lambda+\mu)}{2\theta\mu}A^{1}_{\varepsilon,\delta}+\frac{C_{7}}{2\theta}\bigg\}^{1/2}.
	\end{align*}
	with $K_{8}^{(\mu)}\eqdef\frac{C_{7}}{2\theta\Lambda}\Big[\frac{1}{\mu}+1\Big]+\frac{C_{7}}{2\theta}+\mu C_{2}$. Then,
	%\begin{align*}
	%\frac{1}{\lambda^{2} %\Lambda^{2}}\bigg[1-\frac{K_{7}^{(\lambda)}}{\mu}\bigg]^{2}[A^{1}_{\varepsilon,\delta}]^{2}-&\frac{2 K_{8}^{(\mu)}}{\lambda %\Lambda}\bigg[1-\frac{K_{7}^{(\lambda)}}{\mu}\bigg]A^{1}_{\varepsilon,\delta}+[K_{8}^{(\mu)}]^{2}\notag\\
	%&\leq\frac{ %K_{7}^{(\lambda)}}{\mu}[A^{1}_{\varepsilon,\delta}]^{2}+\frac{C_{3}(\lambda+\mu)}{2\theta\mu}A^{1}_{\varepsilon,\delta}+\frac{C_{3}}{2\theta}.
	%\end{align*}
	%It implies that 
	\begin{equation*}
	\bigg\{\frac{1}{\lambda^{2} \Lambda^{2}}\bigg[1-\frac{K_{7}^{(\lambda)}}{\mu}\bigg]^{2}-
	\frac{ K_{7}^{(\lambda)}}{\mu}\bigg\}[A^{1}_{\varepsilon,\delta}]^{2} 
	\leq\bigg\{\frac{2K_{8}^{(\mu)}}{\lambda \Lambda}\bigg[1-\frac{K_{7}^{(\lambda)}}{\mu}\bigg]+\frac{C_{7}(\lambda+\mu)}{2\theta\mu}\bigg\}A^{1}_{\varepsilon,\delta}+\frac{C_{7}}{2\theta}.
	\end{equation*}
	From here, we conclude there exists a constant $C_{3}=C_{3}(d,\Lambda,\alpha',K_{3})$ such that
	\begin{equation*}
	\omega(x)|\deri^{2}u_{\ell}(x)|\leq A^{1}_{\varepsilon,\delta}\leq C_{3}\quad \text{for}\  (x,\ell)\in \overline{\set}\times\mathbb{I}. 
	\end{equation*}
	Now, assume that \eqref{d7.2} holds. Then,  $2\omega^{2}[\lambda\theta-2]|\deri^{2}u_{\ell_{\mu}}|^{2}\leq2\lambda C_{7}\omega|\deri^{2}u_{\ell_{\mu}}|+(\lambda+\mu)C_{7}$ at $x_{\mu}$ due to $\psi'_{\varepsilon}\geq0$ and $\omega\leq1$. From here, we have that $\omega|\deri^{2}u_{\ell_{\mu}}|\leq K_{9}^{(\lambda,\mu)}$ at $x_{\mu}$, where $K_{9}^{(\lambda,\mu)}$ is a positive constant independent of $A_{\varepsilon,\delta}^{1}$. Therefore, $[A^{1}_{\varepsilon,\delta}]^{2}-\lambda\Lambda A^{1}_{\varepsilon,\delta}[C_{2}+1]\leq\phi_{\ell_{0}}(x_{0})\leq\phi_{\ell_{\mu}}(x_{\mu})\leq [K_{9}^{(\lambda,\mu)}]^{2}+\lambda\Lambda A^{1}_{\varepsilon,\delta} K_{9}^{(\lambda,\mu)}+\mu C^{2}_{2}$. From here, we conclude there exists a constant $C_{3}=C_{3}(d,\Lambda,\alpha',K_{1})$ such that $\omega|\deri^{2}u_{\ell}|\leq A^{1}_{\varepsilon,\delta}\leq C_{3}$ for all $(x,\ell)\in \overline{\set}\times\mathbb{I}$.
\end{proof}

\begin{proof}[Proof of Lemma \ref{D2.0.0}]
Taking first and second derivatives of $\phi_{\ell}$ on $\overline{B}_{\beta'r}$, it can be verified that 
	\begin{align*}
	\tr[a_{\ell}\deri^{2}|\deri^{1}u_{\ell}|^{2}]&=|\deri^{2}u_{\ell}|^{2}\tr[a_{\ell}\deri^{2}\omega^{2}]+2\langle a_{\ell}\deri^{1}\omega^{2},\deri^{1}|\deri^{2}u_{\ell}|^{2}\rangle+\omega^{2}\tr[a_{\ell}\deri^{2}|\deri^{2}u_{\ell}|^{2}]\notag\\
	&\quad+\lambda A^{1}_{\varepsilon,\delta}\tr[\alpha_{\ell_{0}}\deri^{2}u_{\ell}]\tr[a_{\ell}\deri^{2}\omega]+2\lambda A^{1}_{\varepsilon,\delta}\langle a_{\ell}\deri^{1}\omega,\deri^{1}\tr[\alpha_{\ell_{0}}\deri^{2}u_{\ell}]\rangle\notag\\
	&\quad+\lambda A^{1}_{\varepsilon,\delta}\omega\sum_{ji}\alpha_{\ell_{0}\,ji}\tr[a_{\ell}\deri^{2}\partial_{ji}u_{\ell}]+\mu\tr[a_{\ell}\deri^{2}|\deri^{1}u_{\ell}|^{2}].
	\end{align*}	
	From here and noticing that from \eqref{H2},
	\begin{align*}
	\tr[a_{\ell}\deri^{2}|\deri^{1}u_{\ell}|^{2}]%&=2\sum_{i}[\langle a_{\ell}\deri^{1}\partial_{i}u_{\ell},\deri^{1}\partial_{i}u_{\ell}\rangle+\partial_{i}u_{\ell}\tr[a_{\ell}\deri^{2}\partial_{i}u_{\ell}]]\notag\\
	&\geq 2\theta|\deri^{2}u_{\ell}|^{2}+2\sum_{i}\partial_{i}u_{\ell}\tr[a_{\ell}\deri^{2}\partial_{i}u_{\ell}],\\
	\tr[a_{\ell}\deri^{2}|\deri^{2}u_{\ell}|^{2}]%&=2\sum_{ji}[\langle a_{\ell}\deri^{1}\partial_{ji}u_{\ell},\deri^{1}\partial_{ji}u_{\ell}\rangle+\partial_{ji}u_{\ell}\tr[a_{\ell}\deri^{2}\partial_{ji}u_{\ell}]]\notag\\
	&\geq 2\theta|\deri^{3}u_{\ell}|^{2}+2\sum_{ji}\partial_{ji}u_{\ell}\tr[a_{\ell}\deri^{2}\partial_{ji}u_{\ell}],%\label{de2}
	\end{align*}
	it follows that
	\begin{align}\label{d1.2}
	\tr[a_{\ell}\deri^{2}\phi_{\ell}]&\geq 2\theta[\omega^{2}|\deri^{3}u_{\ell}|^{2}+\mu|\deri^{2}u_{\ell}|^{2}]+|\deri^{2}u_{\ell}|^{2}\tr[a_{\ell}\deri^{2}\omega^{2}]\notag\\
	&\quad+2\langle a_{\ell}\deri^{1}\omega^{2},\deri^{1}|\deri^{2}u_{\ell}|^{2}\rangle+\lambda A^{1}_{\varepsilon,\delta}\tr[\alpha_{\ell_{0}}\deri^{2}u_{\ell}]\tr[a_{\ell}\deri^{2}\omega]\notag\\
	&\quad+2\lambda A^{1}_{\varepsilon,\delta}\langle a_{\ell}\deri^{1}\omega,\deri^{1}\tr[\alpha_{\ell_{0}}\deri^{2}u_{\ell}]\rangle+2\mu\sum_{i}\tr[a_{\ell}\deri^{2}\partial_{i}u_{\ell}]\partial_{i}u_{\ell}\notag\\
	&\quad+\sum_{ji}[2\omega^{2}\partial_{ji}u_{\ell}+\lambda A^{1}_{\varepsilon,\delta}\omega\alpha_{\ell_{0}\,ji}]\tr[a_{\ell}\deri^{2}\partial_{ji}u_{\ell}].
	\end{align}	
	Meanwhile, differentiating twice in \eqref{NPD.1}, we see that 
	\begin{align}\label{d2}
	\tr[a_{\ell}\deri^{2}\partial_{ji}u_{\ell}]
	&=\psi''_{\varepsilon,\ell}(\cdot)\bar{\eta}^{(i)}_{\ell}\bar{\eta}^{(j)}_{\ell}+\psi'_{\varepsilon,\ell}(\cdot)\partial_{ji}[|\deri^{1}u_{\ell}|^{2}-g_{\ell}^{2}]+\sum_{\kappa\in\mathbb{I}\setminus\{\ell\}}\psi''_{\delta,\ell,\kappa}(\cdot)\bar{\eta}^{(i)}_{\ell,\kappa}\bar{\eta}^{(j)}_{\ell,\kappa}\notag\\
	&\quad+\sum_{\kappa\in\mathbb{I}\setminus\{\ell\}}\psi'_{\delta,\ell,\kappa}(\cdot)\partial_{ji}[u_{\ell}-u_{\kappa} ]-\tr[\partial_{j}[a_{\ell}]\deri^{2}\partial_{i}u_{\ell}]-\tr[[\partial_{ji}a_{\ell}]\deri^{2}u_{\ell}]\notag\\
	&\quad-\tr[[\partial_{i}a_{\ell}]\deri^{2}\partial_{j}u_{\ell}]-\partial_{ji}[h^{\ell}-\langle b_{\ell},\deri^{1}u_{\ell}\rangle-c_{\ell}u_{\ell}],
	\end{align}  
	where $\bar{\eta}_{\ell}=(\bar{\eta}_{\ell}^{(1)},\dots,\bar{\eta}_{\ell}^{(d)})$ and $\bar{\eta}_{\ell,\kappa}=(\bar{\eta}_{\ell,\kappa}^{(1)},\dots,\bar{\eta}_{\ell,\kappa}^{(d)})$ with  $\bar{\eta}_{\ell}^{(i)}\eqdef\partial_{i}[|\deri^{1}u_{\ell}|^{2}-g_{\ell}^{2}]$ {and} $\bar{\eta}_{\ell,\kappa}^{(i)}\eqdef\partial_{i}[u_{\ell}-u_{\kappa} ].$  From \eqref{dercot2} and \eqref{d1.2}--\eqref{d2}, it follows that 
	\begin{align}\label{d3}
	\omega^{2}\tr[a_{\ell}\deri^{2}\phi_{\ell}]
	&\geq 2\theta[\omega^{4}|\deri^{3}u_{\ell}|^{2}+\mu\omega^{2}|\deri^{2}u_{\ell}|^{2}]+\widetilde{D}_{3}+\widetilde{D}_{4}\notag\\
	&\qquad+\omega^{2}\bigg\{\psi''_{\varepsilon,\ell}(\cdot)\langle[2\omega^{2}\deri^{2}u_{\ell}+\lambda A^{1}_{\varepsilon,\delta}\omega\alpha_{\ell_{0}}]\bar{\eta}_{\ell},\bar{\eta}_{\ell}\rangle\notag\\
	&\qquad+\sum_{\kappa\in\mathbb{I}\setminus\{\ell\}}\psi''_{\delta,\ell,\kappa}(\cdot)\langle[2\omega^{2}\deri^{2}u_{\ell}+\lambda A^{1}_{\varepsilon,\delta}\omega\alpha_{\ell_{0}}]\bar{\eta}_{\ell,\kappa},\bar{\eta}_{\ell,\kappa}\rangle\bigg\}\notag\\
	&\qquad+\omega^{2}\psi'_{\varepsilon,\ell}(\cdot)\widetilde{D}_{5}+\omega^{2}\sum_{\kappa\in\mathbb{I}\setminus\{\ell\}}\psi'_{\delta,\ell,\kappa}(\cdot)\widetilde{D}_{6\,\kappa},
	\end{align}
	where
	\begin{align*}
	\widetilde{D}_{3}&\eqdef2\omega^{2}\langle a_{\ell}\deri^{1}\omega^{2},\deri^{1}|\deri^{2}u_{\ell}|^{2}\rangle+2\lambda A^{1}_{\varepsilon,\delta}\omega^{2}\langle a_{\ell}\deri^{1}\omega,\deri^{1}\tr[\alpha_{\ell_{0}}\deri^{2}u_{\ell}]\rangle\notag\\
	&\quad-\sum_{ij}[2\omega^{4}\partial_{ij}u_{\ell}+\lambda A^{1}_{\varepsilon,\delta}\omega^{3}\alpha_{\ell_{0}\,ij}][2\tr[\partial_{j}a_{\ell}\deri^{2}\partial_{i}u_{\ell}]-\partial_{ij}\langle b_{\ell},\deri^{1}u_{\ell}\rangle],\\
	\widetilde{D}_{4}&\eqdef\omega^{2}|\deri^{2}u_{\ell}|^{2}\tr[a_{\ell}\deri^{2}\omega^{2}]\notag-\mu\omega^{2}\widetilde{D}_{2}u_{\ell}+\lambda A^{1}_{\varepsilon,\delta}\omega^{2}\tr[\alpha_{\ell_{0}}\deri^{2}u_{\ell}]\tr[a_{\ell}\deri^{2}\omega]\notag\\
	&\quad-\sum_{ji}[2\omega^{4}\partial_{ji}u_{\ell}+\lambda A^{1}_{\varepsilon,\delta}\omega^{3}\alpha_{\ell_{0}\,ji}]\left\{\tr[[\partial_{ji}a_{\ell}]\deri^{2}u_{\ell}]+\partial_{ji}[h_{\ell}-c_{\ell}u_{\ell}]\right\},\\
	\widetilde{D}_{5}&\eqdef 2\mu\langle\deri^{1}u_{\ell},\bar{\eta}_{\ell}\rangle+\tr[[2\omega^{2}\deri^{2}u_{\ell}+\lambda A^{1}_{\varepsilon,\delta}\omega\alpha_{\ell_{0}}]\deri^{2}[|\deri^{1}u_{\ell}|^{2}-g_{\ell}^{2}]],\\
	\widetilde{D}_{6\,\kappa}&\eqdef 2\mu\langle\deri^{1}u_{\ell},\bar{\eta}_{\ell,\kappa}\rangle+\tr[[2\omega^{2}\deri^{2}u_{\ell}+\lambda A^{1}_{\varepsilon,\delta}\omega\alpha_{\ell_{0}}]\deri^{2}[u_{\ell}-u_{\kappa} ]].
	\end{align*}
	Recall that $\widetilde{D}_{2}u_{\ell}$ is given in \eqref{dercot2.1}. To obtain the next inequalities, we shall recurrently use \eqref{a4}, \eqref{h2}, Remark \ref{R1},  \eqref{ap1}, \eqref{ap2} and $\lambda,\mu\geq1$. Then,
	\begin{align}
	\widetilde{D}_{3}%&=8\omega^{3}\sum_{ijkm}a_{\ell\,ij}\partial_{j}\omega\partial_{km}u_{\ell}\partial_{kmi}u_{\ell}+2\lambda A^{1}_{\varepsilon,\delta}\omega^{2}\sum_{ijkm} a_{\ell\,im}\partial_{i}\omega\alpha_{\ell_{0}\,km}\partial_{kmj}u_{\ell}\notag\\
	%&\quad-2\sum_{ij}[2\omega^{4}\partial_{ij}u_{\ell}+\lambda A^{1}_{\varepsilon,\delta}\omega^{3}\alpha_{\ell_{0}\,ij}]\left\{\sum_{km}\partial_{j}a_{\ell\,km}\partial_{kmi}u_{\ell}\right.\notag\\
	%&\quad\left.+\sum_{k}[\partial_{k}u_{\ell}\partial_{ij}b_{\ell\,k}+\partial_{j}b_{\ell\,k}\partial_{ki}u_{\ell}+\partial_{i}b_{\ell\,k}\partial_{kj}u_{\ell}+b_{\ell\,k}\partial_{ijk}u_{\ell}]\right\}\notag\\
	%&\geq-8\Lambda K_{1}d^{4}A^{1}_{\varepsilon,\delta}\omega^{2}|\deri^{3}u_{\ell}|-2\lambda\Lambda^{2} d^{4}A^{1}_{\varepsilon,\delta}\omega^{2}|\deri^{3}u_{\ell}|-2d^{3}[2+\lambda\Lambda][d+\Lambda]A^{1}_{\varepsilon,\delta}\omega^{2}|\deri^{3}u_{\ell}|\notag\\
	%&\quad-2d^{2}A^{1}_{\varepsilon,\delta}[2+\lambda\Lambda]dC_{2}\Lambda-4d^{3}\Lambda[A^{1}_{\varepsilon,\delta}]^{2}[2+\lambda\Lambda]\notag\\
	%&\geq-2\left\{4\Lambda K_{1}d^{4}+\lambda\Lambda^{2} d^{4}+d^{3}[2+\lambda\Lambda][d+\Lambda]\right\}A^{1}_{\varepsilon,\delta}\omega^{2}|\deri^{3}u_{\ell}|\notag\\
	%&\quad-2d^{2}A^{1}_{\varepsilon,\delta}[2+\lambda\Lambda]dC_{2}\Lambda-4d^{3}\Lambda[A^{1}_{\varepsilon,\delta}]^{2}[2+\lambda\Lambda]\notag\\
	&\geq-2\left\{4\Lambda K_{1}d^{4}+\Lambda^{2} d^{4}+d^{3}[2+\Lambda][d+\Lambda]\right\}\lambda A^{1}_{\varepsilon,\delta}\omega^{2}|\deri^{3}u_{\ell}|\notag\\
	&\quad-2d^{2}\lambda A^{1}_{\varepsilon,\delta}[2+\Lambda]dC_{2}\Lambda-4d^{3}\Lambda\lambda[A^{1}_{\varepsilon,\delta}]^{2}[2+\Lambda],
	\intertext{and by \eqref{H2},}
	\widetilde{D}_{4}%&\geq-2[A^{1}_{\varepsilon,\delta}]^{2}\Lambda d^{2}K_{1}-2\mu\{[2C_{2}\Lambda d^{3}+2d^{1/2}\Lambda C_{2}]A^{1}_{\varepsilon,\delta}+2C_{2}\Lambda d^{2}+2C_{1}C_{2}d^{1/2}\Lambda-2C_{2}\Lambda d^{1/2}\}\notag\\
	%&\quad-\{\lambda d^{4}\Lambda^{2}K_{1}+d^{2}\Lambda[2+\lambda\Lambda][d^{2}+1]\}[A^{1}_{\varepsilon,\delta}]^{2}-d^{2}\Lambda[2+\lambda\Lambda][2C_{2}+C_{1}]A^{1}_{\varepsilon,\delta}\notag\\
	%&\geq-\{2\Lambda d^{2}K_{1}+\lambda d^{4}\Lambda^{2}K_{1}+d^{2}\Lambda[2+\lambda\Lambda][d^{2}+1]\}[A^{1}_{\varepsilon,\delta}]^{2}\notag\\
	%&\quad-\{2\mu[2C_{2}\Lambda d^{3}+2d^{1/2}\Lambda C_{2}]+d^{2}\Lambda[2+\lambda\Lambda][2C_{2}+C_{1}]\}A^{1}_{\varepsilon,\delta}\notag\\
	%&\quad-2\mu\{2C_{2}\Lambda d^{2}+2C_{1}C_{2}d^{1/2}\Lambda-2C_{2}\Lambda d^{1/2}\}\notag\\
	&\geq-\{2\Lambda d^{2}K_{1}+ d^{4}\Lambda^{2}K_{1}+d^{2}\Lambda[2+\Lambda][d^{2}+1]\}\lambda[A^{1}_{\varepsilon,\delta}]^{2}\notag\\
	&\quad-\{2\mu[2C_{2}\Lambda d^{3}+2d^{1/2}\Lambda C_{2}]+d^{2}\lambda\Lambda[2+\Lambda][2C_{2}+C_{1}]\}A^{1}_{\varepsilon,\delta}\notag\\
	&\quad-2\mu\{2C_{2}\Lambda d^{2}+2C_{1}C_{2}d^{1/2}\Lambda-2C_{2}\Lambda d^{1/2}\}.
	\end{align}	 
	On the other hand, since $\lambda\geq\frac{2}{\theta}$ and using \eqref{H2}, we have that
	\begin{align}\label{d4}
	\langle[2\omega^{2}\deri^{2}u_{\ell}+\lambda A^{1}_{\varepsilon,\delta}\omega\alpha_{\ell_{0}}]\gamma,\gamma\rangle&\geq\omega[\lambda A^{1}_{\varepsilon,\delta}\theta-2\omega|\deri^{2}u_{\ell}|]\,|\gamma|^{2} \geq\omega A^{1}_{\varepsilon,\delta}[\lambda\theta-2]\,|\gamma|^{2}\geq0,
	\end{align}
	for $\gamma\in\R^{d}$. From here and since $\psi''_{\varepsilon,\ell}(\cdot)\geq0$ and $\psi''_{\delta,\ell,\kappa}(\cdot)\geq0$, it follows that 
	\begin{multline}\label{d5}
	\psi''_{\varepsilon,\ell}(\cdot)\langle[2\omega^{2}\deri^{2}u_{\ell}+\lambda A^{1}_{\varepsilon,\delta}\omega\alpha_{\ell_{0}}]\bar{\eta}_{\ell},\bar{\eta}_{\ell}\rangle\\
	+\sum_{\kappa\in\mathbb{I}\setminus\{\ell\}}\psi''_{\delta,\ell,\kappa}(\cdot)\langle[2\omega^{2}\deri^{2}u_{\ell}+\lambda A^{1}_{\varepsilon,\delta}\omega\alpha_{\ell_{0}}]\bar{\eta}_{\ell,\kappa},\bar{\eta}_{\ell,\kappa}\rangle\geq0.
	\end{multline}
	It is easy to verify that
	\begin{multline}\label{d8}
	\omega^{2}\langle\deri^{1}u_{\ell},\deri^{1}|\deri^{2}u_{\ell}|^{2}\rangle+\lambda A^{1}_{\varepsilon,\delta}\omega\langle\deri^{1}u_{\ell},\deri^{1}\tr[\alpha_{\ell_{0}}\deri^{2}u_{\ell}]\rangle+\mu\langle\deri^{1}u_{\ell},\deri^{1}|\deri^{1}u_{\ell}|^{2}\rangle\\
	=\langle\deri^{1}u_{\ell},\deri^{1}\phi_{\ell}\rangle-\langle\deri^{1}u_{\ell},\deri^{1}\omega^{2}\rangle|\deri^{2}u_{\ell}|^{2}-\lambda A^{1}_{\varepsilon,\delta}\tr[\alpha_{\ell_{0}}\deri^{2}u_{\ell}]\langle\deri^{1}u_{\ell},\deri^{1}\omega\rangle
	\end{multline}
	due to  $\partial_{i}\phi_{\ell}=|\deri^{2}u_{\ell}|^{2}\partial_{i}\omega^{2}+\omega^{2}\partial_{i}|\deri^{2}u_{\ell}|^{2}+\lambda A^{1}_{\varepsilon,\delta}\tr[\alpha_{\ell_{0}}\deri^{2}u_{\ell}]\partial_{i}\omega+\lambda A^{1}_{\varepsilon,\delta}\omega\tr[\alpha_{\ell_{0}}\deri^{2}\partial_{i}u_{\ell}]+\mu\partial_{i}|\deri^{1}u_{\ell}|^{2}\quad\text{on}\ B_{\beta r}$. Then,  by \eqref{d4} and \eqref{d8},
	\begin{align}%\label{d6}
	\widetilde{D}_{5}%&=\tr[[2\omega^{2}\deri^{2}u_{\ell}+\lambda A^{1}_{\varepsilon,\delta}\omega\alpha_{\ell_{0}}]\deri^{2}|\deri^{1}u_{\ell}|^{2}]+2\mu\langle\deri^{1}u_{\ell},\deri^{1}|\deri^{1}u_{\ell}|^{2}\rangle\notag\\
	&\quad\geq2\omega A^{1}_{\varepsilon,\delta}[\lambda\theta-2]|\deri^{2}u_{\ell}|^{2}+2\langle\deri^{1}u_{\ell},\deri^{1}\phi_{\ell}\rangle\notag\\
	&\qquad-4\lambda d^{1/2}C_{2}K_{1}A^{1}_{\varepsilon,\delta}|\deri^{2}u_{\ell}|-2\lambda\Lambda d^{5/2}C_{2}K_{1}A^{1}_{\varepsilon,\delta}|\deri^{2}u_{\ell}|\notag\\
	&\qquad-4\mu  \Lambda^{2}d^{1/2}A^{1}_{\varepsilon,\delta}C_{2}-2\lambda d\Lambda^{2}A^{1}_{\varepsilon,\delta}-\lambda\Lambda^{3}d^{2}A^{1}_{\varepsilon,\delta}\notag\\
	&\qquad-4d^{2}\lambda\Lambda^{2}A^{1}_{\varepsilon,\delta}-2\Lambda^{3}d^{2}\lambda A^{1}_{\varepsilon,\delta}.
	\end{align}
	Using the following properties $|A|^{2}-2\tr[AB]+|B|^{2}=\sum_{ij}(A_{ij}-B_{ij})^{2}\geq0$ and $|y_{1}|^{2}-2\langle y_{1},y_{2}\rangle+|y|^{2}=\sum_{i}(y_{1,i}-y_{2,i})^{2}\geq0$,
		where $A=(A_{ij})_{d\times d},B=(B_{ij})_{d\times d}$ and $y_{1}=(y_{1,2},\dots,y_{1,d}),y_{2}=(y_{2,1},\dots,y_{2,d})$ belong $\mathcal{S}(d)$ and $\R^{d}$, respectively, and by definition of $\phi_{\ell}$,  it is easy to corroborate the following identity
	\begin{equation}\label{d6}
	\widetilde{D}_{6\,\kappa}\geq \phi_{\ell}-\phi_{\kappa} ,\ \text{for}\ \kappa\neq\ell.
	\end{equation}
	Applying \eqref{d5}--\eqref{d6} in \eqref{d3} and considering that all  constants that appear in those inequalities (i.e. \eqref{d5}--\eqref{d6}) are bounded by an universal constant $C_{7}=C_{7}(d,\Lambda,\alpha',K_{1})$, we obtain the desired result in the lemma above. With this remark, the proof is concluded.
\end{proof}

\end{document}